\documentclass[12pt]{article}

\usepackage{amsmath,amssymb,bm,amsthm,amscd}
\usepackage{graphicx}

\allowdisplaybreaks

\numberwithin{equation}{subsection}

\newcommand{\Fg}{\mathfrak{g}}
\newcommand{\Fh}{\mathfrak{h}}

\newcommand{\Fsl}{\mathfrak{sl}}

\newcommand{\BZ}{\mathbb{Z}}
\newcommand{\BQ}{\mathbb{Q}}
\newcommand{\BR}{\mathbb{R}}
\newcommand{\BC}{\mathbb{C}}
\newcommand{\BB}{\mathbb{B}}

\newcommand{\CL}{\mathcal{L}}
\newcommand{\CB}{\mathcal{B}}
\newcommand{\CT}{\mathcal{T}}

\newcommand{\ve}{\varepsilon}
\newcommand{\vp}{\varphi}
\newcommand{\vpi}{\varpi}

\newcommand{\bzero}{{\bf 0}}
\newcommand{\bc}{\mathbf{c}_{0}}

\newcommand{\GL}{\mathrm{GL}}
\newcommand{\Hom}{\mathrm{Hom}}
\newcommand{\End}{\mathrm{End}}
\newcommand{\Span}{\mathrm{Span}}
\newcommand{\wt}{\mathop{\rm wt}\nolimits}
\newcommand{\norm}{\mathrm{norm}}
\newcommand{\af}{\mathrm{af}}
\newcommand{\Par}{\mathrm{Par}}

\newcommand{\Conn}{\mathrm{Conn}}
\newcommand{\ch}{\mathop{\rm ch}\nolimits}
\newcommand{\gch}{\mathop{\rm gch}\nolimits}
\newcommand{\cl}{\mathop{\rm cl}\nolimits}
\newcommand{\QLS}{\mathrm{QLS}}
\newcommand{\Deg}{\mathop{\rm Deg}\nolimits}
\newcommand{\Degt}{\mathop{\rm Deg}\nolimits^{\mathrm{tail}}}
\newcommand{\Img}{\mathop{\rm Image}\nolimits}

\newcommand{\PJ}{\Pi^{J}}
\newcommand{\jad}{Q^{\vee,\,\text{\rm $J$-ad}}}

\newcommand{\SB}{\mathrm{SiB}^{J}}
\newcommand{\SBa}{\mathrm{SiB}(\lambda\,;\,a)}
\newcommand{\SBb}[1]{\mathrm{SiB}(\lambda\,;\,#1)}
\newcommand{\SBc}[2]{\mathrm{SiB}(#1\,;\,#2)}

\newcommand{\B}{\mathcal{B}(\lambda)}
\newcommand{\Bo}{\mathcal{B}_{0}(\lambda)}

\newcommand{\sLS}{\mathbb{B}^{\si}(\lambda)}
\newcommand{\sLSo}{\mathbb{B}^{\si}_{0}(\lambda)}

\newcommand{\si}{\frac{\infty}{2}}
\newcommand{\sell}{\ell^{\si}}
\newcommand{\sil}{\prec}
\newcommand{\sile}{\preceq}
\newcommand{\sig}{\succ}
\newcommand{\sige}{\succeq}

\newcommand{\rr}{\Delta_{\af}}
\newcommand{\prr}{\Delta_{\af}^{+}}

\newcommand{\mcr}[1]{\lfloor #1 \rfloor}
\newcommand{\edge}[1]{\xrightarrow{\,#1\,}}
\newcommand{\QJp}[1]{Q_{#1}^{\vee+}}
\newcommand{\tw}[1]{\tau_{#1}}

\newcommand{\pair}[2]{\langle #1,\,#2 \rangle}

\newcommand{\ol}[1]{\overline{#1}}
\newcommand{\ti}[1]{\widetilde{#1}}

\newcommand{\ts}[1]{{\textstyle #1}}

\newcommand{\bqed}{\quad \hbox{\rule[-0.5pt]{3pt}{8pt}}}

\theoremstyle{plain}
\newtheorem{thm}{Theorem}[subsection]
\newtheorem{lem}[thm]{Lemma}
\newtheorem{prop}[thm]{Proposition}
\newtheorem{cor}[thm]{Corollary}

\newtheorem{claim}{Claim}[thm]

\newtheorem{ithm}{Theorem}

\theoremstyle{definition}
\newtheorem{dfn}[thm]{Definition}

\theoremstyle{remark}
\newtheorem{rem}[thm]{Remark}

\newenvironment{enu}{%
 \begin{enumerate}%
}{\end{enumerate}}

\newcommand{\vsp}{\vspace{3mm}}

\begin{document}

\baselineskip=17.5pt

\title{\Large\bf 
Demazure submodules \\ of level-zero extremal weight modules \\
and specializations of Macdonald polynomials%
}
\author{
 Satoshi Naito \\ 
 \small Department of Mathematics, Tokyo Institute of Technology, \\
 \small 2-12-1 Oh-okayama, Meguro-ku, Tokyo 152-8551, Japan \ 
 (e-mail: {\tt naito@math.titech.ac.jp}) \\[3mm]
and \\[3mm]
Daisuke Sagaki \\ 
 \small Institute of Mathematics, University of Tsukuba, \\
 \small Tsukuba, Ibaraki 305-8571, Japan \ 
 (e-mail: {\tt sagaki@math.tsukuba.ac.jp})
}
\date{}
\maketitle

%
\begin{abstract} \setlength{\baselineskip}{16pt}
In this paper, we give a characterization of the crystal bases 
$\CB_{x}^{+}(\lambda)$, $x \in W_{\af}$, 
of Demazure submodules $V_{x}^{+}(\lambda)$, $x \in W_{\af}$, 
of a level-zero extremal weight module $V(\lambda)$ over a quantum affine algebra $U_{q}$,
where $\lambda$ is an arbitrary level-zero dominant integral weight,
and $W_{\af}$ denotes the affine Weyl group.
This characterization is given in terms of the initial direction of
a semi-infinite Lakshmibai-Seshadri path,
and is established under a suitably normalized isomorphism between the
crystal basis $\CB(\lambda)$ of the level-zero extremal weight module
$V(\lambda)$ and the crystal $\sLS$ of 
semi-infinite Lakshmibai-Seshadri paths of shape $\lambda$,
which is obtained in our previous work.
As an application, we obtain a formula 
expressing the graded character of the Demazure submodule
$V_{w_0}^{+}(\lambda)$ in terms of the specialization at $t=0$
of the symmetric Macdonald polynomial $P_{\lambda}(x\,;\,q,\,t)$.
\end{abstract}
%
%
\section{Introduction.} 
\label{sec:intro}
Extremal weight modules over the quantized universal enveloping algebra of 
a symmetrizable Kac-Moody algebra were introduced by \cite{K-mod}.
Since then, the study of level-zero extremal weight modules 
over a quantum affine algebra $U_{q}$ 
has especially been the subject of a number of papers.
Among them, we would like to mention \cite{K-lv0} and \cite{BN}, 
in which many of the crucial results on the structure of 
level-zero extremal weight modules and their crystal bases are obtained.

In our previous paper \cite{INS}, 
for an arbitrary level-zero dominant integral weight $\lambda$, 
we gave an explicit realization of the crystal basis $\CB(\lambda)$ 
of the extremal weight module $V(\lambda)$ of extremal weight $\lambda$ over $U_{q}$,
in terms of semi-infinite Lakshmibai-Seshadri paths (SiLS paths for short)
of shape $\lambda$; here, SiLS paths are analogs of Littelmann's LS paths, 
which are defined by using the semi-infinite Bruhat order 
in place of the ordinary Bruhat order on the affine Weyl group $W_{\af}$, 
and Peterson's coset representatives in place of the usual minimal(-length)
coset representatives. 
Namely, we proved that the crystal basis $\CB(\lambda)$ is 
isomorphic as a crystal to the crystal $\BB^{\si}(\lambda)$ of 
SiLS paths of shape $\lambda$. Note that both of (the crystal graphs of) 
these crystals have infinitely many connected components in general, and 
hence an isomorphism between these crystals is not uniquely determined.

The purpose of this paper is to give a characterization of 
the crystal bases $\CB_{x}^{+}(\lambda)$ of Demazure(-type) submodules 
$V_{x}^{+}(\lambda) := U_{q}^{+} S_{x}^{\norm} v_{\lambda}$ of 
the extremal weight module $V(\lambda)$ of extremal weight $\lambda$, 
where $x$ runs over the affine Weyl group $W_{\af}$. 
Here, $v_{\lambda}$ denotes the generating extremal weight vector of 
$V(\lambda)$ of weight $\lambda$, and $S_{x}^{\norm} v_{\lambda} \in V(\lambda)$ 
is an extremal weight vector of weight $x \lambda$ in the $W_{\af}$-orbit 
of $v_{\lambda}$; also, $U_{q}^{+}$ denotes the positive part of 
the quantum affine algebra $U_{q}$. This characterization is given 
in terms of the initial direction of a SiLS path, and is established 
by normalizing suitably the isomorphism 
$\CB(\lambda) \cong \BB^{\si}(\lambda)$ of crystals.

To be more precise, let $\lambda = \sum_{i \in I} m_{i} \vpi_{i}$, 
with $m_{i} \in \BZ_{\ge 0}$ for $i \in I$, be 
an arbitrary level-zero dominant integral weight, 
where the $\vpi_{i}$, $i \in I:=I_{\af} \setminus \{0\}$, are 
the level-zero fundamental weights.
We set $J := \bigl\{ i \in I \mid m_{i} = 0 \bigr\}$, and
$\BB_{x \sige}^{\si}(\lambda) := 
 \bigl\{\eta \in \BB^{\si}(\lambda) \mid x \sige \iota(\eta) \bigr\}$ 
for $x \in (W^{J})_{\af}$, where $\iota(\eta) \in (W^{J})_{\af}$ 
denotes the initial direction of a SiLS path $\eta \in \BB^{\si}(\lambda)$; 
here, $(W^{J})_{\af}$ denotes the set of Peterson's coset representatives 
for the cosets in $W_{\af}/(W_{J})_{\af}$, 
with $(W_{J})_{\af}:=W_{J} \ltimes Q_{J}^{\vee}$ 
a subgroup of $W_{\af}=W \ltimes Q^{\vee}$, 
where $W$ is th finite Weyl group, and $Q^{\vee}$ is the coroot lattice 
corresponding to the subset $I$ of $I_{\af}$. 
The following is the main result of this paper.

\begin{ithm} \label{ithm}
Let $x \in (W^{J})_{\af}$. Then, under a suitably normalized isomorphism
\begin{equation*}
\Psi_{\lambda}^{\vee} : \CB(\lambda) \stackrel{\sim}{\rightarrow} \BB^{\si}(\lambda)
\end{equation*}
of crystals, there holds the equality
\begin{equation*}
\Psi_{\lambda}^{\vee}(\CB_{x}^{+}(\lambda)) = \BB_{x \sige}^{\si}(\lambda).
\end{equation*}
\end{ithm}

\noindent
Here we should mention that although the statement of the theorem above is 
of the form similar to that of Kashiwara's result in \cite{K-Lit}
(see also \cite[Chapitre~9]{K-Fr}) or Littelmann's result in \cite[\S5]{Lit94} 
in the case of integrable highest weight modules, 
its proof is much more difficult because both of the crystals 
$\CB(\lambda)$ and $\BB^{\si}(\lambda)$ may have infinitely many
connected components, and because the $\lambda$-weight space of $V(\lambda)$ 
may be infinite-dimensional even if these crystals are connected,
in contrast to the case of ordinary highest weight crystals.

As an application of Theorem~\ref{ithm} above, 
we compute the graded character $\gch V_{w_{0}}^{+}(\lambda)$ 
of the Demazure submodule $V_{w_{0}}^{+}(\lambda)$ for 
the longest element $w_{0} \in W$, and obtain a formula expressing 
$\gch V_{w_{0}}^{+}(\lambda)$ in terms of the specialization 
$P_{\lambda}(x\,;\,q,\,0)$ of the symmetric Macdonald polynomial 
$P_{\lambda}(x\,;\,q,\,t)$. Namely, we prove (see Theorem~\ref{thm:grch2}) that
\begin{equation*}
\gch V_{w_{0}}^{+}(\lambda) =
\biggl(\prod_{i \in I} \prod_{r=1}^{m_{i}}(1-q^{r})\biggr)^{-1} 
P_{\lambda}(x\,;\,q,\,0), 
\end{equation*}
where $\lambda = \sum_{i \in I} m_{i} \vpi_{i}$ is as above. 
Here the right-hand side of the equality above 
is called a $q$-Whittaker function in \cite{BF}, where 
the simply-laced cases are mainly treated; 
hence our result gives a new representation-theoretic interpretation of 
$q$-Whittaker functions for all untwisted cases.

Also, for each $w \in W^{J}$, we introduce a certain quotient 
$U_{w}^{+}(\lambda)$ of $V_{w}^{+}(\lambda)$, and give 
a characterization (Theorem~\ref{thm:quotient2}) 
of its crystal basis as a subset of 
$\BB^{\si}(\lambda)$ (or, more precisely, 
its connected component $\BB^{\si}_{0}(\lambda)$)
in terms of an initial direction. 
In our forthcoming paper \cite{LNSSS3}, 
we will prove that the graded character of $U_{w}^{+}(\lambda)$ is identical 
to the specialization at $t = 0$ of 
the nonsymmetric Macdonald polynomial $E_{w\lambda}(x;\,q,\,t)$; 
this result generalizes \cite[Corollary~7.10]{LNSSS2}, 
since $E_{w_{0}\lambda}(x;\,q,\,0)=P_{\lambda}(x;\,q,\,0)$, 
where $w_{0}$ is the longest element of the finite Weyl group $W$. 

This paper is organized as follows.
In \S\ref{sec:SiLS}, we fix our notation 
for untwisted affine root data, and recall the definition of SiLS paths.
Also, we briefly review our results in \cite{INS} that we use in this paper.
In \S\ref{sec:extremal}, we first recall basic properties of 
extremal weight modules and their crystal bases.
Then, we review some results in \cite{BN} that we need in this paper.
In \S\ref{sec:main}, we introduce Demazure submodules 
$V_{x}^{+}(\lambda)$, $x \in (W^{J})_{\af}$, of 
the extremal weight module $V(\lambda)$, and 
their crystal bases $\CB_{x}^{+}(\lambda) \subset \CB(\lambda)$, $x \in (W^{J})_{\af}$. 
Also, we state our main result, i.e., Theorem~\ref{ithm} above.
In \S\ref{sec:proof}, we prove some fundamental properties of 
the crystal bases $\CB_{x}^{+}(\lambda)$, $x \in (W^{J})_{\af}$, of 
Demazure submodules $V_{x}^{+}(\lambda)$, $x \in (W^{J})_{\af}$, 
and similar properties for the crystals 
$\BB^{\si}_{x \sige}(\lambda) \subset \sLS$, $x \in (W^{J})_{\af}$.
Then, after a suitable normalization of the isomorphism 
$\CB(\lambda) \cong \sLS$ of crystals, 
we finally establish our main result ($=$ Theorem~\ref{ithm}) 
stated in \S\ref{sec:main}.
In \S\ref{sec:gch}, we obtain the graded character formula 
above for the Demazure submodule $V_{w_{0}}^{+}(\lambda)$; 
in \S\ref{sec:quotient}, we introduce the quotient 
$U_{w}^{+}(\lambda)$ of $V_{w}^{+}(\lambda)$ 
for each $w \in W^{J}$, and give a characterization of 
its crystal basis. 
%
%
\section{Semi-infinite Lakshmibai-Seshadri paths.}
\label{sec:SiLS}
%
%
\subsection{Untwisted affine root data.}
\label{subsec:affalg}
Let $\Fg_{\af}$ be an untwisted affine Lie algebra 
over $\BC$ with Cartan matrix $A=(a_{ij})_{i,\,j \in I_{\af}}$. 
Let 
$\Fh_{\af}=
\bigl(\bigoplus_{j \in I_{\af}} \BC \alpha_{j}^{\vee}\bigr) \oplus \BC D$
denote the Cartan subalgebra of $\Fg_{\af}$, where 
$\bigl\{\alpha_{j}^{\vee}\bigr\}_{j \in I_{\af}} \subset \Fh_{\af}$ is 
the set of simple coroots, and 
$D \in \Fh_{\af}$ is the scaling element (or the degree operator). 
We denote by 
$\bigl\{\alpha_{j}\bigr\}_{j \in I_{\af}} \subset \Fh_{\af}^{\ast}$ 
the set of simple roots, and by 
$\Lambda_{j} \in \Fh_{\af}^{\ast}$, $j \in I_{\af}$, 
the fundamental weights; 
note that $\pair{D}{\alpha_{j}}=\delta_{j,0}$ and 
$\pair{D}{\Lambda_{j}}=0$ for $j \in I_{\af}$, 
where $\pair{\cdot\,}{\cdot}:
\Fh_{\af} \times \Fh_{\af}^{\ast} \rightarrow \BC$ denotes 
the canonical pairing of $\Fh_{\af}$ and 
$\Fh_{\af}^{\ast}:=\Hom_{\BC}(\Fh_{\af},\,\BC)$. 
Let $\delta=\sum_{j \in I} a_{j}\alpha_{j} \in \Fh_{\af}^{\ast}$ and 
$c=\sum_{j \in I} a^{\vee}_{j} \alpha_{j}^{\vee} \in \Fh_{\af}$ denote 
the null root and the canonical central element of 
$\Fg_{\af}$, respectively. 
The Weyl group $W_{\af}$ of $\Fg_{\af}$ is defined to be the subgroup
$\langle r_{j} \mid j \in I_{\af} \rangle \subset \GL(\Fh_{\af}^{\ast})$ 
generated by the simple reflections $r_{j}$
associated with $\alpha_{j}$ for $j \in I_{\af}$, 
with length function $\ell:W_{\af} \rightarrow \BZ_{\ge 0}$ and 
unit element $e \in W_{\af}$. 
Denote by $\rr$ the set of real roots, i.e., 
$\rr:=\bigl\{x\alpha_{j} \mid x \in W_{\af},\,j \in I_{\af}\bigr\}$, 
and by $\prr \subset \rr$ the set of positive real roots; 
for $\beta \in \rr$, we denote by $\beta^{\vee} \in \Fh_{\af}$ 
the dual root of $\beta$, and by $r_{\beta} \in W_{\af}$ 
the reflection with respect to $\beta$. 
We take a dual weight lattice $P_{\af}^{\vee}$ 
and a weight lattice $P_{\af}$ for $\Fg_{\af}$ as follows:
%
%
\begin{equation} \label{eq:lattices}
P_{\af}^{\vee}=
\left(\bigoplus_{j \in I_{\af}} \BZ \alpha_{j}^{\vee}\right) \oplus \BZ D \, 
\subset \Fh
\quad \text{and} \quad 
P_{\af} = 
\left(\bigoplus_{j \in I_{\af}} \BZ \Lambda_{j}\right) \oplus 
   \BZ \delta \subset \Fh^{\ast}; 
\end{equation}
it is clear that $P_{\af}$ contains $\alpha_{j}$ for all $j \in I_{\af}$, 
and that $P_{\af} \cong \Hom_{\BZ}(P_{\af}^{\vee},\,\BZ)$. Also, we set
\begin{equation*}
Q_{\af}:=\bigoplus_{j \in I_{\af}} \BZ \alpha_{j} \quad \text{and} \quad
Q_{\af}^{\pm}:=\pm \sum_{j \in I_{\af}} \BZ_{\ge 0} \alpha_{j}.
\end{equation*}

We take and fix a distinguished vertex $0 \in I_{\af}$ such that 
$a_0 = a_0^{\vee} =1$, and set $I := I_{\af} \setminus \{ 0 \}$; 
note that the subset $I$ of $I_{\af}$ is 
the index set for the canonical finite-dimensional 
simple Lie subalgebra $\Fg$ of $\Fg_{\af}$. 
For each $i \in I$, we define $\varpi_i := 
\Lambda_i - a_{i}^{\vee} \Lambda_0$; 
note that $\pair{c}{\varpi_i}=0$ for all $i \in I$. 
Set 
\begin{equation*}
Q := \bigoplus_{j \in I} \BZ \alpha_j, \qquad
Q^{+} := \sum_{j \in I} \BZ_{\ge 0} \alpha_j, \qquad
Q^{\vee} := \bigoplus_{j \in I} \BZ \alpha_j^{\vee}, \qquad
P^{+} := \sum_{i \in I} \BZ_{\ge 0} \varpi_i; 
\end{equation*} 
we call an element of $P^{+}$ a level-zero dominant integral weight, 
which can be thought of as a dominant integral weight for $\Fg$. 
Also, we set $W := \langle r_{j} \mid j \in I \rangle \subset W_{\af}$, 
which can be thought of as the (finite) Weyl group of $\Fg$. 
For $\xi \in Q^{\vee}$, let $t_{\xi} \in W_{\af}$ denote 
the translation in $\Fh^{\ast}_{\af}$ by $\xi$ (see \cite[\S6.5]{Kac}). 
Then we know from \cite[Proposition 6.5]{Kac} that 
$\bigl\{ t_{\xi} \mid \xi \in Q^{\vee} \bigr\}$ forms 
an abelian normal subgroup of $W_{\af}$, 
for which $t_{\xi} t_{\zeta} = t_{\xi + \zeta}$, 
$\xi,\,\zeta \in Q^{\vee}$, and 
$W_{\af} = W \ltimes \bigl\{ t_{\xi} \mid \xi \in Q^{\vee} \bigr\}$ hold; 
remark that for $w \in W$ and $\xi \in Q^{\vee}$, we have 
%
%
\begin{align}\label{eq:lv0action}
w t_{\xi} \mu = w \mu - \pair{\xi}{\mu}\delta \quad 
\text{if $\mu \in \Fh_{\af}^{*}$ satisfies $\pair{c}{\mu}=0$}.
\end{align}
We know from \cite[Proposition 6.3]{Kac} that
\begin{align*}
\begin{split}
\rr & = 
\bigl\{ \alpha + n \delta \mid \alpha \in \Delta,\, n \in \BZ \bigr\}, \\
\prr & = 
\Delta^{+} \sqcup 
\bigl\{ \alpha + n \delta \mid \alpha \in \Delta,\, n \in \BZ_{> 0}\bigr\},
\end{split}
\end{align*}
where $\Delta := \Delta_{\af} \cap Q$ is 
the (finite) root system for $\Fg$, and 
$\Delta^{+} := \Delta \cap \sum_{i \in I} \BZ_{\ge 0} \alpha_i$. 
Note that if $\beta \in \Delta_{\af}$ is 
of the form $\beta = \alpha + n \delta$ 
with $\alpha \in \Delta $ and $n \in \BZ$, then 
$r_{\beta} =r_{\alpha} t_{n\alpha^{\vee}}$.

For a subset $J$ of $I$, we set 
\begin{align*}
Q_J &:= \bigoplus_{j \in J} \BZ \alpha_j, &
Q_J^{\vee} &:= \bigoplus_{j \in J} \BZ \alpha_j^{\vee}, &
\QJp{J} &:= \sum_{j \in J} \BZ_{\ge 0} \alpha_j^{\vee}, \\[3mm]
\Delta_J &:= \Delta \cap Q_J, & 
\Delta_J^+ &:= 
 \Delta^{+} \cap \sum_{i \in I} \BZ_{\ge 0}  \alpha_i , &
W_J &:= \langle r_j \mid j \in J \rangle . &
\end{align*}
Also, denote by 
%
%
\begin{equation} \label{eq:prj}
[\,\cdot\,]=[\,\cdot\,]_{I \setminus J} : 
Q^{\vee} \twoheadrightarrow Q_{I \setminus J}^{\vee}
\end{equation}
the projection from $Q^{\vee}=Q_{I \setminus J}^{\vee} \oplus Q_{J}^{\vee}$
onto $Q_{I \setminus J}^{\vee}$ with kernel $Q_{J}^{\vee}$. 
Let $W^J$ denote the set of minimal(-length) coset representatives 
for the quotient $W / W_J$; we know from \cite[\S2.4]{BB} that 
%
%
\begin{equation} \label{eq:mcr}
W^J = \bigl\{ w \in W \mid 
\text{$w \alpha \in \Delta^+$ for all $\alpha \in \Delta_J^+$}\bigr\}.
\end{equation}
For $w \in W$, we denote by $\mcr{w}=\mcr{w}^{J} \in W^J$ 
the minimal coset representative for the coset $w W_J$ in $W/W_J$.
%
%
\subsection{Peterson's coset representatives.}
\label{subsec:W^J_af}

Let $J$ be a subset of $I$. Following \cite{Pet97} 
(see also \cite[\S10]{LS10}), we define
\begin{align}
(\Delta_J)_{\af} 
  & := \bigl\{ \alpha + n \delta \mid 
  \alpha \in \Delta_J , n \in \BZ \bigr\} \subset \Delta_{\af}, \\
(\Delta_J)_{\af}^{+}
  &:= (\Delta_J)_{\af} \cap \prr = 
  \Delta_J^+ \sqcup \bigl\{ \alpha + n \delta \mid 
  \alpha \in \Delta_J,\,n \in \BZ_{> 0} \bigr\}, \\
\label{eq:stabilizer}
(W_J)_{\af} 
 & := W_J \ltimes \bigl\{ t_{\xi} \mid \xi \in Q_J^{\vee} \bigr\}
   = \langle r_{\beta} \mid \beta \in (\Delta_J)_{\af}^{+} \rangle, \\
\label{eq:Pet}
(W^J)_{\af}
 &:= \bigl\{ x \in W_{\af} \mid 
 \text{$x\beta \in \prr$ for all $\beta \in (\Delta_J)_{\af}^+$} \bigr\}. 
\end{align}
Then we know the following from \cite{Pet97} 
(see also \cite[Lemma~10.6]{LS10}).
%
%
\begin{prop} \label{prop:P}
For each $x \in W_{\af}$, there exist a unique 
$x_1 \in (W^J)_{\af}$ and a unique $x_2 \in (W_J)_{\af}$ 
such that $x = x_1 x_2$.
\end{prop}

We define a (surjective) map $\PJ : W_{\af} \rightarrow (W^J)_{\af}$ 
by $\PJ (x) := x_1$ if $x= x_1 x_2$ with $x_1 \in (W^J)_{\af}$ and 
$x_2 \in (W_J)_{\af}$.

An element $\xi \in Q^{\vee}$ is said to be $J$-adjusted 
if $\pair{\xi}{\gamma} \in \bigl\{ -1,\,0 \bigr\}$ 
for all $\gamma \in \Delta_{J}^{+}$ (see \cite[Lemma~3.8]{LNSSS}). 
Let $\jad$ denote the set of $J$-adjusted elements.
%
%
\begin{lem}[{\cite[Lemma~2.3.5]{INS}}] \label{lem:J-adj}
\mbox{}
\begin{enu}

\item For each $\xi \in Q^{\vee}$, there exists 
a unique $\phi_J (\xi) \in Q_J^{\vee}$ such that 
$\xi + \phi_J (\xi) \in \jad$. In particular, 
$\xi \in \jad$ if and only if $\phi_{J}(\xi)=0$. 

\item For each $\xi \in Q^{\vee}$, 
the element $\PJ (t_{\xi}) \in (W^{J})_{\af}$ is of the form 
$\PJ (t_{\xi}) = z_{\xi} t_{\xi + \phi_J (\xi)}$ 
for a specific element $z_{\xi} \in W_J$. 
Also, $\PJ (w t_{\xi}) = 
\mcr{w} z_{\xi} t_{\xi + \phi_J (\xi)}$ 
for every $w \in W$ and $\xi \in Q^{\vee}$.

\item 
We have 
%
%
\begin{equation}\label{eq:W^J_af}
(W^J)_{\af} = 
\bigl\{ w z_{\xi} t_{\xi} \mid w \in W^J,\,\xi \in \jad \bigr\}.
\end{equation}

\end{enu}
\end{lem}
%
%
\begin{lem}[{\cite[Lemma~2.3.6]{INS}}] \label{lem:r_i}
Let $x \in (W^J)_{\af}$ and $j \in I_{\af}$. 
Then, $x^{-1} \alpha_j \notin (\Delta_J)_{\af}$ if and only if 
$r_j x \in (W^J)_{\af}$.
\end{lem}

%
\subsection{Parabolic semi-infinite Bruhat graph.} 
\label{subsec:SBG}
%
%
\begin{dfn}[{\cite{Pet97}}] \label{dfn:sell}
Let $x \in W_{\af}$, and 
write it as $x = w t_{\xi}$ for $w \in W$ and $\xi \in Q^{\vee}$. 
Then we define the semi-infinite length $\sell(x)$ of $x$ by
$\sell (x) := \ell (w) + 2 \pair{\xi}{\rho}$,
where $\rho := (1/2)\sum_{\alpha \in \Delta^+} \alpha$. 
\end{dfn}
%
%
\begin{dfn}\label{def:SiB}
Let $J$ be a subset of $I$. 
\begin{enu}

\item
Define the (parabolic) semi-infinite Bruhat graph $\SB$ 
to be the $\prr$-labeled, directed graph with vertex set $(W^J)_{\af}$ 
and $\prr$-labeled, directed edges of the following form:
$x \edge{\beta} r_{\beta} x$ for $x \in (W^J)_{\af}$ and $\beta \in \prr$, 
where $r_{\beta } x \in (W^J)_{\af}$ and 
$\sell (r_{\beta} x) = \sell (x) + 1$.

\item
The semi-infinite Bruhat order is a partial order 
$\sile$ on $(W^J)_{\af}$ defined as follows: 
for $x,\,y \in (W^J)_{\af}$, we write $x \sile y$ 
if there exists a directed path from $x$ to $y$ in $\SB$; 
also, we write $x \sil y$ if $x \sile y$ and $x \ne y$. 
(In \cite{INS}, we used the symbol $\le_{\si}$ 
for the semi-infinite Bruhat order, but in the present paper, 
we use the symbol $\sile$ instead of $\le_{\si}$.)
\end{enu}
\end{dfn}
%
%
\begin{rem}[{\cite[Corollary~4.2.2]{INS}}] \label{rem:SiB}
Let $J$ be a subset of $I$. 
Let $x \in (W^{J})_{\af}$ and $\beta \in \prr$ be such that 
$x \edge{\beta} r_{\beta}x$ in $\SB$. 
Then, $\beta$ is either of 
the following forms: $\beta=\alpha$ with $\alpha \in \Delta^{+}$, 
or $\beta=\alpha+\delta$ with $-\alpha \in \Delta^{+}$.
Moreover, if $x=wz_{\xi}t_{\xi}$ for $w \in W^{J}$ and 
$\xi \in \jad$ (see \eqref{eq:W^J_af}), then $w^{-1}\alpha \in 
\Delta^{+} \setminus \Delta_{J}^{+}$ in both cases above. 
Hence if we write $r_{\beta}x \in (W^{J})_{\af}$ as 
$r_{\beta}x=vz_{\zeta}t_{\zeta}$ 
with some $v \in W^{J}$ and $\zeta \in \jad$, then 
$[\zeta-\xi] \in Q^{\vee+}_{I \setminus J}$ in both cases above, 
where $[\,\cdot\,]:Q^{\vee} = Q^{\vee}_{I \setminus J} \oplus Q_{J}^{\vee} 
\twoheadrightarrow Q^{\vee}_{I \setminus J}$ is the projection 
(see \eqref{eq:prj}). 
\end{rem}
%
%
\begin{lem} \label{lem:SiB2}
Let $J$ be a subset of $I$. 
Let $w_{1},\,w_{2} \in W^{J}$, and fix $\xi \in \jad$. 
Then, $w_{1}z_{\xi}t_{\xi} \sige w_{2}z_{\xi}t_{\xi}$ 
if and only if $w_{1} \ge w_{2}$ 
in the (ordinary) Bruhat order $\ge$ on $W^{J}$. 
\end{lem}

\begin{proof}
We first show the ``if'' part; 
by the chain property (see, for example, \cite[Theorem~2.5.5]{BB}),
we may assume that $w_{1}=r_{\alpha}w_{2}$ 
for some $\alpha \in \Delta^{+} \setminus \Delta_{J}^{+}$, and 
$\ell(w_{1})=\ell(w_{2})+1$. Then we see easily that
$w_{2}z_{\xi}t_{\xi} \edge{\alpha} w_{1}z_{\xi}t_{\xi}$ 
in the semi-infinite Bruhat graph $\SB$, and hence 
$w_{1}z_{\xi}t_{\xi} \sige w_{2}z_{\xi}t_{\xi}$. 

We next show the ``only if'' part; 
for simplicity, we give a proof only in the case where 
$w_{1}z_{\xi}t_{\xi}$ covers $w_{2}z_{\xi}t_{\xi}$, 
that is, $w_{2}z_{\xi}t_{\xi} \edge{\beta} w_{1}z_{\xi}t_{\xi}$ 
for some $\beta \in \Delta_{\af}^{+}$ 
(the proof for the general case is similar). 
By Remark~\ref{rem:SiB}, 
$\beta$ is either of 
the following forms: $\beta=\alpha$ with $\alpha \in \Delta^{+}$, 
or $\beta=\alpha+\delta$ with $-\alpha \in \Delta^{+}$.
Suppose that 
$\beta=\alpha+\delta$ with $-\alpha \in \Delta^{+}$. 
Then, 
\begin{equation*}
r_{\beta}w_{2}z_{\xi}t_{\xi} = r_{-\alpha}t_{\alpha}w_{2}z_{\xi}t_{\xi} = 
r_{-\alpha}w_{2}z_{\xi}t_{\xi+z_{\xi}^{-1}w_{2}^{-1}\alpha}. 
\end{equation*}
Since $r_{\beta}w_{2}z_{\xi}t_{\xi} = w_{1}z_{\xi}t_{\xi}$ with 
$\xi+z_{\xi}^{-1}w_{2}^{-1}\alpha \ne \xi$, this is a contradiction. 
Thus, $\beta = \alpha$ with $\alpha \in \Delta^{+}$. 
Then we can easily check that 
$w_{1}=r_{\alpha}w_{2}$, and $\ell(w_{1})=\ell(w_{2})+1$, 
which implies that $w_{1} \ge w_{2}$.
This proves the lemma. 
\end{proof}
%
%
\begin{lem}[{\cite[Remark~4.1.3]{INS}}] \label{lem:si2}
Let $\lambda \in P^{+}$, 
and set $J=J_{\lambda}:=\bigl\{i \in I \mid 
\pair{\alpha_{i}^{\vee}}{\lambda}=0\bigr\}$.
For $x \in (W^{J})_{\af}$ and $j \in I_{\af}$, 
the element $r_{j}x$ is contained in $(W^{J})_{\af}$ if and only if 
$\pair{\alpha_{j}^{\vee}}{x\lambda} \ne 0$ (see also Lemma~\ref{lem:r_i}).
Moreover, in this case, 
%
%
\begin{equation} \label{eq:simple2}
\begin{cases}
x \sil r_{j}x
\iff \pair{\alpha_{j}^{\vee}}{x\lambda} > 0, & \\[1.5mm]
%
r_{j}x \sil x
\iff \pair{\alpha_{j}^{\vee}}{x\lambda} < 0. & 
\end{cases}
\end{equation}
\end{lem}
%
%
\begin{lem} \label{lem:si-L41}
Let $\lambda \in P^{+}$, and set
$J=J_{\lambda}:=\bigl\{i \in I \mid 
\pair{\alpha_{i}^{\vee}}{\lambda}=0\bigr\}$.
Assume that $x,\,y \in W_{\af}$ satisfy $x \sile y$. 
Let $j \in I_{\af}$. 

\begin{enu}

\item 
If $\pair{\alpha_{j}^{\vee}}{x\lambda} > 0$ and $\pair{\alpha_{j}^{\vee}}{y\lambda} \le 0$, 
then $r_{j}x \sile y$.

\item 
If $\pair{\alpha_{j}^{\vee}}{x\lambda} \ge 0$ and $\pair{\alpha_{j}^{\vee}}{y\lambda} < 0$, 
then $x \sile r_{j}y$. 

\item 
If $\pair{\alpha_{j}^{\vee}}{x\lambda} > 0$ and $\pair{\alpha_{j}^{\vee}}{y\lambda} > 0$, or 
if $\pair{\alpha_{j}^{\vee}}{x\lambda} < 0$ and $\pair{\alpha_{j}^{\vee}}{y\lambda} < 0$, 
then $r_{j}x \sile r_{j}y$.
\end{enu}
\end{lem}
\begin{proof}
Parts (1) and (2) follow immediately from \cite[Lemma~4.1.6]{INS}.
Let us prove part (3). We give a proof only for the case that 
$\pair{\alpha_{j}^{\vee}}{x\lambda} > 0$ and $\pair{\alpha_{j}^{\vee}}{y\lambda} > 0$; 
the proof for the case that 
$\pair{\alpha_{j}^{\vee}}{x\lambda} < 0$ and $\pair{\alpha_{j}^{\vee}}{y\lambda} < 0$ is similar. 
By \eqref{eq:simple2} and the assumption that $\pair{\alpha_{j}^{\vee}}{y\lambda} > 0$, 
we have $y \sil r_{j}y$, and hence $x \sil r_{j}y$. 
Since $\pair{\alpha_{j}^{\vee}}{x\lambda} > 0$ and 
$\pair{\alpha_{j}^{\vee}}{r_{j}y\lambda} < 0$ by the assumption, 
we deduce by part (1) that $r_{j}x \sile r_{j}y$. 
This proves the lemma. 
\end{proof}
%
%
\subsection{Semi-infinite Lakshmibai-Seshadri paths.}
\label{subsec:SLS}

In this subsection, we fix 
$\lambda \in P^+$, and set 
$J=J_{\lambda}:= \bigl\{ i \in I \mid 
\pair{\alpha_i^{\vee}}{\lambda}=0 \bigr\} \subset I$.

\begin{dfn}\label{def:QBaf(a)}
For a rational number $0 < a < 1$, 
define $\SBa$ to be the subgraph of $\SB$ 
with the same vertex set but having only the edges of the form:
$x \edge{\beta} y$ with 
$a\pair{\beta^{\vee}}{x\lambda} \in \BZ$.
\end{dfn}
%
%
\begin{dfn}\label{dfn:SiLS}
A semi-infinite Lakshmibai-Seshadri path (SiLS path for short) of 
shape $\lambda $ is, by definition, a pair $(\bm{x}\,;\,\bm{a})$ of 
a (strictly) decreasing sequence $\bm{x} : x_1 \sig \cdots \sig x_s$ 
of elements in $(W^J)_{\af}$ and an increasing sequence 
$\bm{a} : 0 = a_0 < a_1 < \cdots  < a_s =1$ of rational numbers 
satisfying the condition that there exists a directed path 
from $x_{u+1}$ to  $x_u$ in $\SBb{a_u}$ 
for each $u = 1,2,\ldots , s-1$. 
We denote by $\sLS$ the set of all SiLS paths of shape $\lambda$.
\end{dfn}

Following \cite[\S3.1]{INS}, 
we equip the set $\sLS$ with a crystal structure 
(with weights in $P_{\af}$) in the following way; 
for the definition of crystals, 
see \cite[\S7.2]{K-OnCB} and \cite[Definition~4.5.1]{HK} for example. 
For $\eta=(x_{1},\,\dots,\,x_{s}\,;\,a_{0},\,a_{1},\,\dots,\,a_{s}) \in \sLS$, 
we define $\ol{\eta}:[0,1] \rightarrow \BR \otimes_{\BZ} P_{\af}$ 
to be the piecewise-linear, continuous map 
whose ``direction vector'' for the interval 
$[a_{u-1},\,a_{u}]$ is equal to $x_{u}\lambda \in P_{\af}$ 
for each $1 \le u \le s$, that is, 
\begin{equation} \label{eq:LSpath}
\ol{\eta} (t) := 
\sum_{p = 1}^{u-1}(a_p - a_{p-1}) x_{p}\lambda + (t - a_{u-1}) x_{u}\lambda
\quad
\text{for $t \in [a_{u-1},\,a_u]$, $1 \le u \le s$};
\end{equation}
note that $\ol{\eta}$ is a Lakshmibai-Seshadri path (LS path for short) 
of shape $\lambda$ by \cite[Proposition~3.1.3 together with Eq.\,(2.2.2)]{INS}. 
Then we define $\wt : \sLS \rightarrow P_{\af}$ by 
$\wt(\eta) := \ol{\eta} (1) \in P_{\af}$ 
(see \cite[Lemma~4.5\,a)]{Lit95}). 

Now we define operators $e_{j}$, $f_{j}$, $j \in I_{\af}$, 
which we call root operators for $\sLS$. Set 
\begin{equation} \label{eq:H}
\begin{cases}
H_j^{\eta}(t) := \pair{\alpha_{j}^{\vee}}{\ol{\eta}(t)} \quad 
\text{for $t \in [0,1]$}, \\[1.5mm]
m_{j}^{\eta} := 
 \min \bigl\{ H_{j}^{\eta} (t) \mid t \in [0,1] \bigr\}. 
\end{cases}
\end{equation}
%
%
\begin{rem} \label{rem:local}
Since $\ol{\eta}$ is an LS path of shape $\lambda$, 
we see from \cite[Lemma 4.5\,d)]{Lit95} that 
all local minima of the function $H^{\eta}_{j}(t)$, 
$t \in [0,1]$, are integers. 
In particular, the minimum $m^{\eta}_{j}$ 
is a nonpositive integer (recall that $\ol{\eta}(0)=0$, 
and hence $H^{\eta}_{j}(0)=0$).
\end{rem}

We define $e_{j}\eta$ as follows. 
If $m^{\eta}_{j}=0$, then we set $e_j \eta := \bzero$, 
where $\bzero$ is an additional element not contained in any crystal. 
If $m_{j}^{\eta} \le -1$, then set
%
%
\begin{equation} \label{eq:t-e}
\begin{cases}
t_1 := 
  \min \bigl\{ t \in [0,\,1] \mid 
    H_{j}^{\eta}(t) = m_{j}^{\eta} \bigr\}, \\[1.5mm]
t_0 := 
  \max \bigl\{ t \in [0,\,t_{1}] \mid 
    H_{j}^{\eta}(t) = m_{j}^{\eta} + 1 \bigr\} ;
\end{cases}
\end{equation}
from Remark~\ref{rem:local}, it follows that $H_{j}^{\eta}(t)$ is 
strictly decreasing on $[t_0,\,t_1]$. 
Let $1 \le p \le q \le s$ be such that 
$a_{p-1} \le t_0 < a_p$ and $t_1 = a_{q}$. 
Then we define $e_{j}\eta$ by
\begin{align*}
e_{j} \eta := ( 
& x_{1},\,\ldots,\,x_{p},\,r_{j}x_{p},\,r_{j}x_{p+1},\,\ldots,\,
  r_{j}x_{q},\,x_{q+1},\,\ldots,\,x_{s} ; \\
& a_{0},\,\ldots,\,a_{p-1},\,t_{0},\,a_{p},\,\ldots,\,a_{q}=t_{1},\,
\ldots,\,a_{s});
\end{align*}
if $t_{0} = a_{p-1}$, then we drop $x_{p}$ and $a_{p-1}$, and 
if $r_{j} x_{q} = x_{q+1}$, then we drop $x_{q+1}$ and $a_{q}=t_{1}$.

Similarly, we define $f_{j}\eta$ as follows. 
If $H_{j}^{\eta}(1) - m_{j}^{\eta} = 0$, 
then we set $f_{j} \eta := \bzero$. 
If $H_{j}^{\eta}(1) - m_{j}^{\eta}  \ge 1$, 
then set
%
%
\begin{equation} \label{eq:t-f}
\begin{cases}
t_0 := 
 \max \bigl\{ t \in [0,1] \mid H_{j}^{\eta}(t) = m_{j}^{\eta} \bigr\}, \\[1.5mm]
t_1 := 
 \min \bigl\{ t \in [t_0,1] \mid H_{j}^{\eta}(t) = m_{j}^{\eta} + 1 \bigr\};
\end{cases}
\end{equation}
from Remark \ref{rem:local}, it follows that 
$H_{j}^{\eta}(t)$ is strictly increasing on $[t_0,\,t_1]$. 
Let $0 \le p \le q \le s-1$ be such that $t_{0} = a_{p}$, and 
$a_{q} < t_{1} \le a_{q+1}$. Then we define $f_{j}\eta$ by
\begin{align*}
f_{j} \eta := ( 
& x_{1},\,\ldots,\,x_{p},\,r_{j}x_{p+1},\,\dots,\,
  r_{j} x_{q},\,r_{j} x_{q+1},\,x_{q+1},\,\ldots,\,x_{s} ; \\
& a_{0},\,\ldots,\,a_{p}=t_{0},\,\ldots,\,a_{q},\,t_{1},\,
  a_{q+1},\,\ldots,\,a_{s});
\end{align*}
if $t_{1} = a_{q+1}$, then we drop $x_{q+1}$ and $a_{q+1}$, and 
if $x_{p} = r_j x_{p+1}$, then we drop $x_{p}$ and $a_{p}=t_{0}$.
Set $e_{j} \bzero = f_{j} \bzero := \bzero$ for all $j \in I_{\af}$.
%
%
\begin{thm}[{see \cite[Theorem~3.1.5]{INS}}]
\label{thm:stability}
\mbox{}
\begin{enu}
\item
The set $\sLS \sqcup \bigl\{ \bzero \bigr\}$ is 
stable under the action of the root operators 
$e_j$ and $f_j$, $j \in I_{\af}$.

\item
For each $\eta \in \sLS$ and $j \in I_{\af}$, we set 
\begin{equation*}
\begin{cases}
\ve_{j} (\eta) := 
 \max \bigl\{ k \ge 0 \mid e_{j}^{k} \eta \neq \bzero \bigr\}, \\[1.5mm]
\vp_{j} (\eta) := 
 \max \bigl\{ k \ge 0 \mid f_{j}^{k} \eta \neq \bzero \bigr\}.
\end{cases}
\end{equation*}
Then, the set $\sLS$, equipped with the maps $\wt$, $e_{j}$, $f_{j}$, 
$j \in I_{\af}$, and $\ve_{j}$, $\vp_{j}$, $j \in I_{\af}$, 
defined above, is a crystal with weights in $P_{\af}$.
\end{enu}
\end{thm}

For $\eta=(x_{1},\,\dots,\,x_{s}\,;\,a_{0},\,a_{1},\,\dots,\,a_{s}) \in 
\sLS$, we set 
%
%
\begin{equation} \label{eq:ik}
\iota(\eta):=x_{1} \quad \text{and} \quad \kappa(\eta):=x_{s}.
\end{equation}
The following lemma will be used in the proof of Lemma~\ref{lem:seq} below. 
%
%
\begin{lem} \label{lem:vevp}
Let $\eta \in \sLS$, and $j \in I_{\af}$. 
If $\pair{\alpha_{j}^{\vee}}{\kappa(\eta)\lambda} > 0$, then 
$f_{j}\eta \ne \bzero$. Moreover, 
$\kappa(f_{j}^{\max}\eta) = r_{j}\kappa(\eta)$ 
if and only if $\pair{\alpha_{j}^{\vee}}{\kappa(\eta)\lambda} > 0$, 
where $f_{j}^{\max}\eta:=f_{j}^{\vp_{j}(\eta)}\eta$. 
\end{lem}

\begin{proof}
If $\pair{\alpha_{j}^{\vee}}{\kappa(\eta)\lambda} > 0$, then 
we see that $H^{\eta}_{j}(1) - m^{\eta}_{j} > 0$, and hence 
$H^{\eta}_{j}(1) - m^{\eta}_{j} \ge 1$ by Remark~\ref{rem:local}. 
Therefore, $f_{j}\eta \ne \bzero$ by the definition of the root operator $f_{j}$. 

Assume first that $\pair{\alpha_{j}^{\vee}}{\kappa(\eta)\lambda} > 0$, 
and suppose, for a contradiction, 
that $\kappa(f_{j}^{\max}\eta) = \kappa(\eta)$. 
Then we have 
$\pair{\alpha_{j}^{\vee}}{\kappa(f_{j}^{\max}\eta)\lambda} = 
 \pair{\alpha_{j}^{\vee}}{\kappa(\eta)\lambda} > 0$, 
and hence $f_{j}f_{j}^{\max}\eta \ne \bzero$ by the assertion just shown. 
However, this contradicts the definition of $f_{j}^{\max}\eta$. 
Thus we obtain $\kappa(f_{j}^{\max}\eta) = r_{j}\kappa(\eta)$. 

Assume next that $\pair{\alpha_{j}^{\vee}}{\kappa(\eta)\lambda} \le 0$; 
we will show by induction on $k$ that 
$\kappa(f_{j}^{k}\eta) = \kappa(\eta)$ for all $0 \le k \le \vp_{j}(\eta)$. 
If $k = 0$, then the assertion is obvious. Assume that $k > 0$, 
and set $\eta':=f_{j}^{k-1}\eta$; by our induction hypothesis, 
we have $\kappa(\eta')=\kappa(\eta)$. 
Take $t_{0},\,t_{1} \in [0,1]$ as in \eqref{eq:t-f}, 
with $\eta'$ in place of $\eta$. 
Then the function $H^{\eta'}_{j}(t)$ is strictly increasing 
on $[t_{0},\,t_{1}]$ (see the comment following \eqref{eq:t-f}). 
Since $\kappa(\eta')=\kappa(\eta)$, we have 
$\pair{\alpha_{j}^{\vee}}{\kappa(\eta')\lambda} = 
\pair{\alpha_{j}^{\vee}}{\kappa(\eta)\lambda} \le 0$, 
which implies that the function $H^{\eta'}_{j}(t)$
is weakly decreasing on $[1-\epsilon,\,1]$ for a sufficiently small 
$\epsilon > 0$. Therefore, we deduce that $t_{1} < 1$, and hence 
$\kappa(f_{j}\eta')=\kappa(\eta')$ by the definition of the root operator $f_{j}$. 
Combining the above, we obtain 
$\kappa(f_{j}^{k}\eta)=\kappa(f_{j}\eta')=\kappa(\eta') = \kappa(\eta)$. 
This proves the lemma. 
\end{proof}
%
%
\subsection{SiLS paths associated with multi-partitions.}
\label{subsec:Parp}

As in the previous subsection, we fix 
$\lambda \in P^{+}$, and set 
$J=J_{\lambda}:= \bigl\{ i \in I \mid 
\pair{\alpha_i^{\vee}}{\lambda}=0 \bigr\} \subset I$.
We write $\lambda \in P^{+}$ as 
$\lambda = \sum_{i \in I} m_{i} \varpi_{i}$ 
with $m_{i} \in \BZ_{\ge 0}$ for $i \in I$, and define
%
%
\begin{equation} \label{eq:olpar}
\ol{\Par(\lambda)} := 
 \bigl\{ \bc = (\rho^{(i)})_{i \in I} \mid 
 \text{$\rho^{(i)}$ is a partition of length $\le m_{i}$ 
 for each $i \in I$} \bigr\},
\end{equation}
%
%
\begin{equation} \label{eq:par}
\Par(\lambda) := 
 \bigl\{ \bc = (\rho^{(i)})_{i \in I} \mid 
 \text{$\rho^{(i)}$ is a partition of length $< m_{i}$ 
 for each $i \in I$} \bigr\};
\end{equation}
we understand that a partition of length less than $0$ 
is the empty partition $\emptyset$ 
(the sets $\ol{\Par(\lambda)}$ and $\Par(\lambda)$ are 
identical to $\mathbf{N}^{\mathcal{R}_{0}}(\lambda)$ and 
$\mathbf{N}^{\mathcal{R}_{0}}(\lambda)'$ in the notation of 
\cite[Definition~4.2 and page 371]{BN}, respectively).
Note that $\Par(\lambda) \subset \ol{\Par(\lambda)}$.
For $\bc = (\rho^{(i)})_{i \in I} \in \ol{\Par(\lambda)}$, we set 
$|\bc|:=\sum_{i \in I} |\rho^{(i)}|$, where for a partition 
$\chi = (\chi_1 \ge \chi_2 \ge \cdots \ge \chi_{m})$, 
we set $|\chi| := \chi_{1}+\cdots+\chi_{m}$. 
We equip the set $\Par(\lambda)$ with a crystal structure as follows: 
for each $\bc = (\rho^{(i)})_{i \in I} \in \Par(\lambda)$, we set
\begin{equation}
\begin{cases}
e_{j} \bc = f_{j} \bc := \bzero, \quad 
\ve_{j} (\bc) = \vp_{j} (\bc) := -\infty 
  & \text{for $j \in I_{\af}$}, \\[1.5mm]
\wt(\bc) := - |\bc| \delta. &
\end{cases}
\end{equation}

Let $\Conn(\sLS)$ denote the set of all connected components of $\sLS$, 
and let $\sLSo \in \Conn(\sLS)$ denote the connected component of $\sLS$
containing $\eta_{e}:=(e\,;\,0,\,1) \in \sLS$. 
%
%
\begin{prop} \label{prop:SLS}
Keep the notation above. 
\begin{enu}

\item Each connected component $C \in \Conn(\sLS)$ 
of $\sLS$ contains a unique element of the form:
%
%
\begin{equation} \label{eq:etaC}
\eta^{C} = 
 (z_{\xi_1}t_{\xi_1},\, 
  z_{\xi_2}t_{\xi_2},\,\dots,\,z_{\xi_{s-1}}t_{\xi_{s-1}},\,e \,;\, 
  a_{0},\,a_{1},\,\dots,\,a_{s-1},\,a_{s})
\end{equation}
for some $s \ge 1$ and $\xi_{1},\,\xi_{2},\,\ldots,\,\xi_{s-1} \in \jad$ 
(see \cite[Proposition~7.1.2]{INS}); 
recall that $e$ denotes the unit element of $W_{\af}$. 

\item There exists a bijection 
$\Theta:\Conn(\sLS) \rightarrow \Par(\lambda)$ such that 
$\wt(\eta^{C})=\lambda-|\Theta(C)|\delta = \lambda+\wt(\Theta(C))$ 
(see \cite[Proposition~7.2.1 and its proof]{INS}). 

\item Let $C \in \Conn(\sLS)$. 
Then, there exists an isomorphism $C \stackrel{\sim}{\rightarrow}
\bigl\{\Theta(C)\bigr\} \otimes \sLSo$ of crystals that maps 
$\eta^{C}$ to $\Theta(C) \otimes \eta_{e}$. 
Consequently, $\sLS$ is isomorphic as a crystal to 
$\Par(\lambda) \otimes \sLSo$ 
(see \cite[Proposition~3.2.4 and its proof]{INS}). 
\end{enu}
\end{prop}
%
%
\subsection{Dual crystal of $\sLS$.}
\label{subsec:dual-sLS}

As in the previous subsection, we fix 
$\lambda \in P^{+}$, and set 
$J=J_{\lambda}:= \bigl\{ i \in I \mid 
\pair{\alpha_i^{\vee}}{\lambda}=0 \bigr\} \subset I$.
Let $w_{0} \in W$ denote the longest element of the (finite) Weyl group $W$, 
and define an involution $\sigma:I \rightarrow I$ by 
$w_{0}\alpha_{j} = -\alpha_{\sigma(j)}$ for $j \in I$; 
recall that $w_{0}^{2}=e$, and hence $\sigma^{2}$ is 
the identity map on $I$. 
Note that $-w_{0}\lambda \in P^{+}$, and $J_{-w_{0}\lambda}=
\bigl\{ i \in I \mid \pair{\alpha_i^{\vee}}{-w_0\lambda}=0 \bigr\} = \sigma(J)$. 
Also, let $w_{\sigma(J),0} \in W_{\sigma(J)}$ denote 
the longest element of the (finite) Weyl group $W_{\sigma(J)}$; 
we see from \eqref{eq:mcr} that 
the minimal coset representative 
$\mcr{w_{0}}^{\sigma(J)} \in W^{\sigma(J)}$ 
is identical to $w_{0}w_{\sigma(J),0}$. 

Now, for $x \in (W^{J})_{\af}$, 
we set $x^{\vee}:=x\mcr{w_{0}}^{\sigma(J)}=xw_{0}w_{\sigma(J),0}$. 
Then it follows easily from definition \eqref{eq:Pet} that 
$x^{\vee} \in (W^{\sigma(J)})_{\af}$; notice that 
$\mcr{w_{0}}^{\sigma(J)}\beta = w_{0}w_{\sigma(J),0}\beta 
\in (\Delta_{J})_{\af}^{+}$ for all 
$\beta \in (\Delta_{\sigma(J)})_{\af}^{+}$. 
Moreover, we have (cf. \cite[Proposition~4.3\,(2)]{LNSSS})
%
%
\begin{equation} \label{eq:sil-vee}
\sell(x^{\vee})=
\underbrace{\ell(w_{0})-\ell(w_{\sigma(J),0})}_{
 = \ell(\mcr{w_0}^{\sigma(J)})
}-\sell(x) \quad 
\text{for every $x \in (W^{J})_{\af}$}.
\end{equation}
Indeed, if $x=wz_{\xi}t_{\xi}$ for $w \in W^{J}$ and $\xi \in \jad$ 
(see \eqref{eq:W^J_af}), then 
\begin{align}
\sell(x^{\vee}) & 
  = \sell(wz_{\xi}t_{\xi}w_{0}w_{\sigma(J),0}) 
  = \sell(wz_{\xi}w_{0}w_{\sigma(J),0}t_{(w_{0}w_{\sigma(J),0})^{-1}\xi}) \nonumber \\
& = \ell(\mcr{wz_{\xi}w_{0}w_{\sigma(J),0}}^{\sigma(J)})+
  2\pair{(w_{0}w_{\sigma(J),0})^{-1}\xi}{\rho-\rho_{\sigma(J)}} \label{eq:sv1}
\end{align}
by \cite[Lemma~A.2.1]{INS}, 
where for a subset $K$ of $I$, we set $\rho_{K}:=\sum_{\alpha \in \Delta_{K}^{+}}\alpha$. 
Since $-w_{0}(\Delta_{K}^{+})=\Delta_{\sigma(K)}^{+}$ 
for a subset $K$ of $I$, it follows that 
$w_{0}\rho_{\sigma(J)}= w_{0}^{-1}\rho_{\sigma(J)}=-\rho_{J}$
and $w_{0}\rho=-\rho$. Also, notice that 
$\pair{\alpha_{j}^{\vee}}{\rho-\rho_{\sigma(J)}}= 0$ for all $j \in \sigma(J)$,
which implies that $z(\rho-\rho_{\sigma(J)})=\rho-\rho_{\sigma(J)}$ 
for all $z \in W_{\sigma(J)}$. Combining these, we deduce that 
%
%
\begin{equation} \label{eq:sv2}
\pair{(w_{0}w_{\sigma(J),0})^{-1}\xi}{\rho-\rho_{\sigma(J)}} = 
\pair{\xi}{w_{0}(\rho-\rho_{\sigma(J)})} = 
-\pair{\xi}{\rho-\rho_{J}}.
\end{equation}
Next, we compute 
\begin{equation*}
\ell(\mcr{wz_{\xi}w_{0}w_{\sigma(J),0}}^{\sigma(J)}) 
   = \ell(\mcr{wz_{\xi}w_{0}}^{\sigma(J)}) 
   = \ell(\mcr{w_{0}(w_{0}ww_{0})(w_{0}z_{\xi}w_{0})}^{\sigma(J)}).
\end{equation*}
Since $z_{\xi} \in W_{J}$, it follows immediately that 
$w_{0}z_{\xi}w_{0} \in W_{\sigma(J)}$. Also, since $w \in W^{J}$, 
we see that $w_{0}ww_{0} \in W^{\sigma(J)}$. 
Therefore, we have 
%
%
\begin{align}
& \ell(\mcr{ wz_{\xi}w_{0}w_{\sigma(J),0} }^{\sigma(J)}) =
  \ell(\mcr{ w_{0} 
     (\underbrace{w_{0}ww_{0}}_{\in W^{\sigma(J)}})
     (\underbrace{w_{0}z_{\xi}w_{0}}_{\in W_{\sigma(J)}})}^{\sigma(J)}) = 
  \ell(\mcr{ w_{0}\underbrace{(w_{0}ww_{0})}_{\in W^{\sigma(J)}} }^{\sigma(J)}) \nonumber \\[3mm]
& \hspace*{15mm} 
  = \ell(w_{0})-\ell(w_{\sigma(J),0})-\ell(w_{0}ww_{0}) \quad 
   \text{by \cite[Proposition~4.3\,(2)]{LNSSS}} \nonumber \\
& \hspace*{15mm} 
  = \ell(w_{0})-\ell(w_{\sigma(J),0})-\ell(w). \label{eq:sv3}
\end{align}
Substituting \eqref{eq:sv2} and \eqref{eq:sv3} into \eqref{eq:sv1}, 
we obtain \eqref{eq:sil-vee}. 
From the definition of the (parabolic) semi-infinite Bruhat graph, 
by using \eqref{eq:sil-vee}, we easily obtain the next lemma. 
%
%
\begin{lem} \label{lem:dual}
Let $0 < a \le 1$ be a rational number. 
Let $x,\,y \in (W^{J})_{\af}$, and $\beta \in \prr$. Then, 
$x \edge{\beta} y$ in $\SBa$ if and only if 
$y^{\vee} \edge{\beta} x^{\vee}$ in $\SBc{-w_{0}\lambda}{a}$. 
\end{lem}

For $\eta=(x_{1},\,\dots,\,x_{s}\,;\,a_{0},\,a_{1},\,\dots,\,a_{s}) \in 
\sLS$, we set
%
%
\begin{equation} \label{eq:etavee}
\eta^{\vee}:= 
(x_{s}^{\vee},\,\dots,\,x_{1}^{\vee}\,;\,1-a_{s},\,\dots,\,1-a_{1},\,1-a_{0}).
\end{equation}
By Lemma~\ref{lem:dual}, we see that $\eta^{\vee} \in \BB^{\si}(-w_{0}\lambda)$. 
Also, in the same way as \cite[Lemma~2.1\,e)]{Lit95} (cf. \cite[\S7.4]{K-OnCB}), 
it is easily shown that for $\eta \in \BB^{\si}(\lambda)$, 
%
%
\begin{equation} \label{eq:veep}
\begin{cases}
\wt(\eta^{\vee})=-\wt(\eta), \quad \text{and} \\[1.5mm]
(e_{j}\eta)^{\vee}=f_{j}\eta^{\vee}, 
(f_{j}\eta)^{\vee}=e_{j}\eta^{\vee} \quad 
\text{for all $j \in I_{\af}$}, 
\end{cases}
\end{equation}
where we set $\bzero^{\vee}:=\bzero$. 
%
%
\section{Extremal weight modules and their crystal bases.}
\label{sec:extremal}
%
%
\subsection{Quantized universal enveloping algebras.}
\label{subsec:quantum}

Let $(\cdot\,,\,\cdot)$ denote
the nondegenerate, symmetric, $\BC$-bilinear form on $\Fh_{\af}^{\ast}$, 
normalized as in \cite[\S6]{Kac}, and fix a positive integer $d \in \BZ_{> 0}$
such that $(\alpha_{j},\,\alpha_{j})/2 \in \BZ d^{-1}$ 
for all $j \in I_{\af}$. Let $q$ be an indeterminate, 
and set $q_{s}:=q^{1/d}$. 
Denote by $U_{q}=U_{q}(\Fg_{\af})= \langle E_{j},\,F_{j},\,q^{h} \mid 
j \in I_{\af},\,h \in d^{-1}P^{\vee} \rangle$ the quantized universal 
enveloping algebra  over $\BQ(q_{s})$ associated with $\Fg_{\af}$, 
where $E_{j}$ and $F_{j}$ denote the Chevalley generators corresponding to 
the simple root $\alpha_{j}$ for $j \in I$, and denote by 
$U_{q}^{+}=\langle E_{j} \mid j \in I_{\af} \rangle$ 
(resp., $U_{q}^{-}=\langle F_{j} \mid j \in I_{\af} \rangle$) 
the $\BQ(q_{s})$-subalgebra of $U_{q}$ generated by 
$E_{j}$ (resp., $F_{j}$), $j \in I_{\af}$. 
We define a $\BQ(q_{s})$-algebra involutive automorphism 
$\vee:U_{q} \rightarrow U_{q}$ and a $\BQ(q_{s})$-algebra involutive 
anti-automorphism $\ast:U_{q} \rightarrow U_{q}$ by:
\begin{align}
& E_{j}^{\vee}=F_{j}, \quad F_{j}^{\vee}=E_{j}, \quad (q^{h})^{\vee}=q^{-h}, \quad \text{and} \label{eq:vee} \\
& E_{j}^{\ast}=E_{j}, \quad F_{j}^{\ast}=F_{j}, \quad (q^{h})^{\ast}=q^{-h} \label{eq:ast}
\end{align}
for $j \in I_{\af}$ and $h \in d^{-1}P^{\vee}$. 
Also, we define a $\BQ$-algebra involutive automorphism 
$\ol{\phantom{x}}:U_{q} \rightarrow U_{q}$ by:
\begin{equation}
\ol{E_{j}}=E_{j}, \quad 
\ol{F_{j}}=F_{j}, \quad 
\ol{q^{h}}=q^{-h}, \quad \ol{q_{s}}=q_{s}^{-1} \label{eq:bar}
\end{equation}
for $j \in I_{\af}$ and $h \in d^{-1}P^{\vee}$.

Let $(\CL(\pm\infty),\,\CB(\pm\infty))$ denote 
the crystal basis of $U_{q}^{\mp}$, with 
$u_{\pm\infty} \in \CB(\pm\infty)$ 
the element corresponding to $1 \in U_{q}^{\mp}$. 
Recall from \cite[Theorem 2.1.1]{K-Lit} and \cite[\S8.3]{K-OnCB}
that $\ast:U_{q} \rightarrow U_{q}$ 
induces an involution on $\CB(\pm \infty)$, 
which is also denoted by $\ast$; we call this involution 
the $\ast$-operation on $\CB(\pm \infty)$. 
%
%
\subsection{Extremal weight vectors and extremal elements.}
\label{subsec:extremal}

Let $M$ be an integrable $U_{q}$-module. 
A nonzero weight vector $v \in M$ of weight $\lambda \in P_{\af}$ 
is said to be extremal (see \cite[\S3.1]{K-lv0} and \cite[\S2.6]{K-rims}) 
if there exists a family $\bigl\{ v_{x} \bigr\}_{x \in W_{\af}}$ 
of weight vectors in $M$ such that $v_{e}=v$, 
and such that for every $j \in I_{\af}$ and $x \in W_{\af}$ with 
$n:=\pair{\alpha_{j}^{\vee}}{x\lambda} \ge 0$ (resp., $\le 0$),
the equalities $E_{j}v_{x}=0$ and $F_{j}^{(n)}v_{x}=v_{r_{j}x}$ 
(resp., $F_{j}v_{x}=0$ and $E_{j}^{(-n)}v_{x}=v_{r_{j}x}$) hold, 
where $E_{j}^{(k)}$ and $F_{j}^{(k)}$ are the divided powers of 
$E_{j}$ and $F_{j}$, respectively, for $k \in \BZ_{\ge 0}$; 
observe that the weight of $v_{x}$ is equal to $x\lambda$. 
Then the Weyl group $W_{\af}$ acts on 
the set of extremal weight vectors in $M$ by
%
%
\begin{equation} \label{eq:Wact1}
S_{r_{j}}^{\norm}v:=
\begin{cases}
F_{j}^{(n)}v & \text{if $n:=\pair{\alpha_{j}^{\vee}}{\mu} \ge 0$}, \\[1.5mm]
E_{j}^{(-n)}v & \text{if $n:=\pair{\alpha_{j}^{\vee}}{\mu} \le 0$}
\end{cases}
\end{equation}
for an extremal weight vector $v \in M$ of weight $\mu \in P_{\af}$ and $j \in I_{\af}$ 
(see \cite[(2.23)]{K-rims}); if $\bigl\{v_{x}\bigr\}_{x \in W_{\af}}$ is 
the family of weight vectors associated with an extremal weight vector $v$, 
then $v_{x}=S_{x}^{\norm}v$ for all $x \in W_{\af}$. 

Now, let $\CB$ be a regular (or normal) crystal in the sense of 
\cite[\S2.2]{K-lv0} (or \cite[p.\,389]{K-mod}); 
in particular, as a crystal for $U_{q}(\Fg) \subset U_{q}(\Fg_{\af})$, 
it decomposes into a disjoint union of ordinary highest weight crystals.
By \cite[\S7]{K-mod}, the Weyl group $W_{\af}$ acts on $\CB$ by
%
%
\begin{equation} \label{eq:Wact2}
S_{r_{j}}b:=
\begin{cases}
f_{j}^{n}b & \text{if $n:=\pair{\alpha_{j}^{\vee}}{\wt b} \ge 0$}, \\[1.5mm]
e_{j}^{-n}b & \text{if $n:=\pair{\alpha_{j}^{\vee}}{\wt b} \le 0$}
\end{cases}
\end{equation}
for $b \in \CB$ and $j \in I_{\af}$, where $e_{j}$ and $f_{j}$, $j \in I_{\af}$, 
are the Kashiwara operators on $\CB$. 
An element $b \in \CB$ of weight $\lambda \in P_{\af}$ 
is said to be extremal (see \cite[\S2.6]{K-rims}; cf. \cite[\S3.1]{K-lv0})
if $e_{j}S_{x}b = \bzero$ (resp., $f_{j}S_{x}b = \bzero$) 
for all $x \in W_{\af}$ and $j \in I_{\af}$ such that 
$\pair{\alpha_{j}^{\vee}}{x\lambda} \ge 0$ (resp., $\le 0$). 
%
%
\begin{lem} \label{lem:extg}
Let $M$ be an integrable $U_{q}$-module, and assume that 
$M$ has a crystal basis $(\CL,\,\CB)$ 
with global basis $\bigl\{G(b) \mid b \in \CB\bigr\}$ 
(see \cite[Definitions~2.2.2 and 2.2.3]{K-lv0} for example); 
note that $\CB$ is a regular crystal. If $b \in \CB$ is an extremal 
element of weight $\lambda$, then the global basis element $G(b)$ is 
an extremal weight vector of weight $\lambda$. Moreover, we have 
$G(S_{x}b)=S_{x}^{\norm}G(b)$ for all $x \in W_{\af}$.
\end{lem}
\begin{proof}
We set $v:=G(b)$, and $v_{x}:=G(S_{x}b)$ for $x \in W_{\af}$; it suffices to
show that the family $\bigl\{v_{x}\bigr\}_{x \in W_{\af}}$ 
satisfies the condition for $v$ to be an extremal weight vector. 
It is obvious that $v_{e}=v$. Let $x \in W_{\af}$, and $j \in I_{\af}$. 
Assume that $n:=\pair{\alpha_{j}^{\vee}}{x\lambda} \ge 0$.
Then, $e_{j}S_{x}b = \bzero$ and $f_{j}^{n+1}S_{x}b = f_{j}S_{r_{j}x}b = \bzero$ 
by the definition of an extremal element. 
Hence, by exactly the same argument as for \cite[Lemma~5.1.1]{K-GB}, we obtain
$E_{j}v_{x}=E_{j}G(S_{x}b)=0$ and $F_{j}^{(n)}v_{x}=F_{j}^{(n)}G(S_{x}b)=
G(f_{j}^{n}S_{x}b)=G(S_{r_{j}x}b)=v_{r_{j}x}$. 
Similarly, it is shown that 
if $n:=\pair{\alpha_{j}^{\vee}}{x\lambda} \le 0$, then 
$F_{j}v_{x}=0$ and $E_{j}^{(-n)}v_{x}=v_{r_{j}x}$. 
This proves the lemma. 
\end{proof}
%
%
\subsection{Extremal weight modules.}
\label{subsec:extmodl}

Let $\lambda \in P_{\af}$ be an arbitrary integral weight. 
Let $V(\lambda)$ denote the extremal weight module of 
extremal weight $\lambda$ over $U_q$, which is 
an integrable $U_q$-module generated 
by a single element $v_{\lambda}$ with 
the defining relation that $v_{\lambda}$ is 
an ``extremal weight vector'' of weight $\lambda$
(for details, see \cite[\S8]{K-mod} and \cite[\S3]{K-lv0}).
We know from \cite[Proposition~8.2.2]{K-mod} that $V(\lambda)$ has 
a crystal basis $(\CL(\lambda),\,\CB(\lambda))$ 
with global basis 
$\bigl\{G(b) \mid b \in \CB(\lambda)\bigr\}$. 
Denote by $u_{\lambda}$ the element of $\CB(\lambda)$ 
such that $G(u_{\lambda})=v_{\lambda} \in V(\lambda)$. 

%
\begin{rem} \label{rem:Gext}
Let $\lambda \in P_{\af}$. 
The crystal basis $\CB(\lambda)$ is a regular crystal, and 
$u_{\lambda} \in \CB(\lambda)$ is an extremal element of weight $\lambda$. 
Also, by Lemma~\ref{lem:extg}, we have
%
%
\begin{equation} \label{eq:Gext}
G(S_{x}u_{\lambda})=
S^{\norm}_{x}G(u_{\lambda})=
S^{\norm}_{x}v_{\lambda} \qquad
\text{for all $x \in W_{\af}$}. 
\end{equation}
%
\end{rem}

For $\mu \in P_{\af}$, let 
$\CT_{\mu}=\bigl\{ \tw{\mu} \bigr\}$ denote
the crystal consisting of a single element of weight $\mu$ 
(see \cite[Example~7.3]{K-OnCB}). 
We see from \cite[Theorems~2.1.1\,(v) and 3.1.1]{K-mod} that 
$\CB(\infty) \otimes \CT_{\mu} \otimes \CB(-\infty)$ is 
a regular crystal for each $\mu \in P_{\af}$. 
The $\ast$-operation on the crystal 
$\ti{\CB}:=
 \bigsqcup_{\mu \in P_{\af}} \CB(\infty) \otimes \CT_{\mu} \otimes \CB(-\infty)$ 
is given as follows (see \cite[Corollary~4.3.3]{K-mod}): 
for $b_{1} \otimes \tw{\mu} \otimes b_{2} 
\in \CB(\infty) \otimes \CT_{\mu} \otimes \CB(-\infty)$ with $\mu \in P_{\af}$, 
%
%
\begin{equation} \label{eq:astm}
(b_{1} \otimes \tw{\mu} \otimes b_{2})^{\ast}:=
b_{1}^{\ast} \otimes \tw{-\mu-\wt(b_1) - \wt(b_2)} \otimes b_{2}^{\ast}.
\end{equation}
For each $j \in I_{\af}$, we define maps $e_{j}^{\ast}$ and $f_{j}^{\ast}$ 
from $\ti{\CB} \sqcup \bigl\{\bzero\bigr\}$ to itself by 
$e_{j}^{\ast}:=\ast \circ e_{j} \circ \ast$ and 
$f_{j}^{\ast}:=\ast \circ e_{j} \circ \ast$, 
where we understand that $\bzero^{\ast}=\bzero$; 
we know from \cite[Theorem~5.1.1]{K-mod} that 
for every $j \in I_{\af}$, the maps $e_{j}^{\ast}$ and $f_{j}^{\ast}$ 
are strict morphisms of crystals. 
Also, for each $x \in W_{\af}$, 
the map $S_{x}^{\ast}:=\ast \circ S_{x} \circ \ast$ is 
a strict automorphism of the crystal $\ti{\CB}$ that 
maps $\CB(\infty) \otimes \CT_{\mu} \otimes \CB(-\infty)$ 
onto $\CB(\infty) \otimes \CT_{x\mu} \otimes \CB(-\infty)$. 
We know the following from \cite[Proposition~8.2.2]{K-mod}, 
\cite[\S3.1]{K-lv0}, and \cite[\S2.6]{K-rims}. 
%
%
\begin{prop} \label{prop:822}
\mbox{}
\begin{enu}

\item For each $\lambda \in P_{\af}$, the subset 
$\bigl\{b \in \CB(\infty) \otimes \CT_{\lambda} \otimes \CB(-\infty)
\mid \text{\rm $b^{\ast}$ is extremal}\bigr\}$
is a subcrystal of $\CB(\infty) \otimes \CT_{\lambda} \otimes \CB(-\infty)
\subset \ti{\CB}$. Moreover, it is isomorphic as a crystal to 
the crystal basis $\CB(\lambda)$; hence we regard $\CB(\lambda)$ as a 
subcrystal of $\CB(\infty) \otimes \CT_{\lambda} \otimes \CB(-\infty) 
\subset \ti{\CB}$. 

\item Let $\lambda \in P_{\af}$, and $x \in W_{\af}$. 
If $b \in \CB(\lambda) \subset 
\CB(\infty) \otimes \CT_{\lambda} \otimes \CB(-\infty)$, then 
$S_{x}^{\ast}(b) \in \CB(x\lambda) \subset 
\CB(\infty) \otimes \CT_{x\lambda} \otimes \CB(-\infty)$. 
Thus, $S_{x}^{\ast}$ gives an isomorphism of crystals from $\CB(\lambda)$ 
onto $\CB(x\lambda)$, that is, 
\begin{equation*}
S_{x}^{\ast} : \CB(\lambda) \stackrel{\sim}{\rightarrow} \CB(x\lambda).
\end{equation*}

\item Let $\lambda \in P_{\af}$, and $x \in W_{\af}$. 
There exists a $U_{q}$-module isomorphism 
%
%
\begin{equation} \label{eq:isom-ext}
V(\lambda) \stackrel{\sim}{\rightarrow} V(x\lambda)
\end{equation}
that maps $v_{\lambda} \in V(\lambda)$ to 
$S_{x^{-1}}^{\norm}v_{x\lambda} \in V(x\lambda)$. 
Moreover, this isomorphism is compatible with the global bases; 
namely, for $b \in \CB(\lambda)$, the global basis element $G(b) \in V(\lambda)$ 
is sent to $G(S_{x}^{\ast}(b)) \in V(x\lambda)$ under this isomorphism. 
\end{enu}
\end{prop}

%
\subsection{Dual crystal of $\B$.}
\label{subsec:dual-B}

It is easily seen that 
the $\BQ(q_{s})$-algebra involutive automorphism 
$\vee:U_{q} \stackrel{\sim}{\rightarrow} U_{q}$ (see \eqref{eq:vee}) 
induces a bijection $\vee : \CB(\pm \infty) \rightarrow \CB(\mp \infty)$; 
we see that for $b \in \CB(\pm\infty)$, 
%
%
\begin{equation} \label{eq:vee-inf}
\begin{cases}
\wt(b^{\vee})=-\wt(b), \quad \text{and} \\[1.5mm]
(e_{j}b)^{\vee}=f_{j}b^{\vee}, 
(f_{j}b)^{\vee}=e_{j}b^{\vee} \quad 
\text{for all $j \in I_{\af}$},
\end{cases}
\end{equation}
where we set $\bzero^{\vee}:=\bzero$. 
We define an involution $\vee$ on the crystal $\bigsqcup_{\mu \in P_{\af}} 
\CB(\infty) \otimes \CT_{\mu} \otimes \CB(-\infty)$ as follows: 
for $b_{1} \otimes \tw{\mu} \otimes b_{2} \in 
\CB(\infty) \otimes \CT_{\mu} \otimes \CB(-\infty)$, $\mu \in P_{\af}$, 
%
%
\begin{equation} \label{eq:veem}
(b_{1} \otimes \tw{\mu} \otimes b_{2})^{\vee} : = 
b_{2}^{\vee} \otimes \tw{-\mu} \otimes b_{1}^{\vee} \in 
\CB(\infty) \otimes \CT_{-\mu} \otimes \CB(-\infty); 
\end{equation}
we see that the same equalities as those in \eqref{eq:vee-inf} hold for 
$b \in \bigsqcup_{\mu \in P_{\af}} 
\CB(\infty) \otimes \CT_{\mu} \otimes \CB(-\infty)$ and $j \in I_{\af}$. 
%
%
\begin{rem} \label{rem:ext-vee}
We deduce that $(S_{x}b)^{\vee} =S_{x}b^{\vee}$ 
for all $x \in W_{\af}$ and $b \in \bigsqcup_{\mu \in P_{\af}} 
\CB(\infty) \otimes \CT_{\mu} \otimes \CB(-\infty)$, 
which implies that 
if $b \in \bigsqcup_{\mu \in P_{\af}} \CB(\infty) \otimes \CT_{\mu} \otimes \CB(-\infty)$ 
is an extremal element, then so is $b^{\vee}$. 
\end{rem}

From the definitions \eqref{eq:vee} and \eqref{eq:ast}, we see easily 
that $\vee \circ \ast = \ast \circ \vee$ holds on $U_{q}$, and hence 
$(b^{\ast})^{\vee}=(b^{\vee})^{\ast}$ for all $b \in \CB(\pm\infty)$. 
Hence, from the definitions \eqref{eq:astm} and \eqref{eq:veem}, 
it follows that 
%
%
\begin{equation} \label{eq:vee-ast}
\vee \circ \ast = \ast \circ \vee \quad 
\text{holds on} \quad \bigsqcup_{\mu \in P_{\af}} 
\CB(\infty) \otimes \CT_{\mu} \otimes \CB(-\infty).
\end{equation}
Therefore, we deduce from Remark~\ref{rem:ext-vee} that 
if $b \in \CB(\lambda) \cong 
\bigl\{b \in \CB(\infty) \otimes \CT_{\lambda} \otimes \CB(-\infty)
\mid \text{$b^{\ast}$ is extremal}\bigr\}$, 
then $b^{\vee} \in \CB(-\lambda) \cong 
\bigl\{b \in \CB(\infty) \otimes \CT_{-\lambda} \otimes \CB(-\infty)
\mid \text{$b^{\ast}$ is extremal}\bigr\}$. 
Note that the following equalities hold for $b \in \CB(\lambda)$:
%
%
\begin{equation} \label{eq:veee}
\begin{cases}
\wt(b^{\vee})=-\wt(b), \quad \text{and} \\[1.5mm]
(e_{j}b)^{\vee}=f_{j}b^{\vee}, 
(f_{j}b)^{\vee}=e_{j}b^{\vee} \quad 
\text{for all $j \in I_{\af}$}.
\end{cases}
\end{equation}

%
\subsection{Isomorphism theorem.}
\label{subsec:isom}
Let $\lambda 
=\sum_{i \in I} m_{i}\vpi_{i} \in P^{+}$ 
be a level-zero dominant integral weight, 
and $\ol{\Par(\lambda)}$ and $\Par(\lambda)$ 
as defined in \eqref{eq:olpar} and \eqref{eq:par}, respectively. 
For each $\bc \in \ol{\Par(\lambda)}$, 
we define an element $S_{\bc} \in U_{q}^{+}$ 
of weight $|\bc|\delta$ as on page 352 of \cite{BN}; 
this is a basis element of the ``imaginary part'' 
of $U_{q}^{+}$, and is identical to 
$B_{\mathbf{c}}=L(\mathbf{c},\,0)$ 
for $\mathbf{c}=(0,\,\bc,\,0)=\bc$ 
(see \cite[the paragraph including Eq.\,(3.11)]{BN}). 
Also, we set $S_{\bc}^{-}:=\ol{S_{\bc}^{\vee}} \in 
U_{q}^{-}$ (see \cite[Remark~4.1]{BN}); 
note that the weight of $S_{\bc}^{-}$ is equal to $-|\bc|\delta$. 
We deduce from \cite[Proposition~3.27]{BN} that 
%
%
\begin{equation} \label{eq:bc}
b(\bc) : = S_{\bc}^{-}+q_{s}\CL(\infty)
\end{equation}
is contained in $\CB(\infty)$.

Let $\Bo$ denote the connected component of $\CB(\lambda)$ 
containing $u_{\lambda}$. We know the following 
from \cite[Proposition~4.3 and Theorem~4.16 (and their proofs)]{BN} . 
%
%
\begin{prop} \label{prop:ext}
Keep the notation and setting above. 
\begin{enu}

\item 
For each $\bc \in \Par(\lambda)$, 
the element $u^{\bc}:=b(\bc) \otimes \tw{\lambda} \otimes u_{-\infty}$ is 
an extremal element of weight $\lambda-|\bc|\delta$ contained in 
$\CB(\lambda) \subset \CB(\infty) \otimes \CT_{\lambda} \otimes \CB(-\infty)$. 

\item Each connected component of $\CB(\lambda)$ contains 
a unique element of the form 
$u^{\bc}=b(\bc) \otimes \tw{\lambda} \otimes u_{-\infty}$ 
with $\bc \in \Par(\lambda)$. Moreover, in this case, 
there exists an isomorphism of crystals from 
the connected component containing $u^{\bc}$ 
onto $\bigl\{\bc\bigr\} \otimes \CB_{0}(\lambda)$
that maps $u^{\bc}$ to $\bc \otimes u_{\lambda}$. 
Consequently, $\CB(\lambda)$ is isomorphic as a crystal to 
$\Par(\lambda) \otimes \CB_{0}(\lambda)$. 
\end{enu}
\end{prop}

We know from \cite[Proposition~3.2.2]{INS} that 
there exists an isomorphism $\Bo \stackrel{\sim}{\rightarrow}
\sLSo$ of crystals that maps $u_{\lambda} \in \Bo$ to 
$\eta_{e}=(e\,;\,0,\,1) \in \sLSo$. 
Therefore, by combining Propositions~\ref{prop:SLS} and \ref{prop:ext}, 
we deduce that there exists an isomorphism 
%
%
\begin{equation} \label{eq:Psi}
\Psi_{\lambda}:\B \stackrel{\sim}{\rightarrow} \sLS
\end{equation}
of crystals that maps $u^{\bc} \in \B$ to 
$\eta^{\Theta^{-1}(\bc)} \in \sLS$
for each $\bc \in \Par(\lambda)$. Also, we define a bijection 
$\Psi_{\lambda}^{\vee}:\B \rightarrow \sLS$ by 
the following commutative diagram: 
%
%
\begin{equation} \label{eq:Psivee}
\begin{CD}
\B @>{\Psi_{\lambda}^{\vee}}>> \sLS \\
@V{S_{w_0}^{\ast} \circ \vee}VV @VV{\vee}V \\
\CB(-w_{0}\lambda) @>{\Psi_{-w_{0}\lambda}}>> \BB^{\si}(-w_{0}\lambda).
\end{CD}
\end{equation}
Then we see from \eqref{eq:veep} and \eqref{eq:veee} that 
the map $\Psi_{\lambda}^{\vee}$ above is an isomorphism of crystals. 

%
\begin{rem} \label{rem:extp}
Keep the notation and setting above.
\begin{enu}

\item
We know from \cite[Remark~7.2.2]{INS} that $\eta^{C}$ is 
an extremal element for every $C \in \Conn(\sLS)$. 

\item
Let $\eta \in \sLS$ be such that $\ol{\eta}(t) \equiv t\lambda \mod \BR\delta$ for 
all $t \in [0,1]$. Then it is easily seen by using \eqref{eq:W^J_af} that 
$\eta$ is of the form:
\begin{equation*}
\eta = 
 (z_{\zeta_{1}}t_{\zeta_{1}},\,\dots,\,z_{\zeta_{s}}t_{\zeta_{s}}\,;\,
  a_{0},\,a_{1},\,\dots,\,a_{s})
\end{equation*}
for some $s \ge 1$ and 
$\zeta_{1},\,\dots,\,\zeta_{s} \in \jad$ (see also \cite[Proposition~7.1.1]{INS}). 
Moreover, by the same argument as for \cite[Eq.\,(5.1.6)]{INS}, we can show that 
%
%
\begin{equation} \label{eq:SxetaC}
S_{x}\eta=
 \bigl(\PJ(xz_{\zeta_1}t_{\zeta_1}),\,\dots,\,\PJ(xz_{\zeta_{s}}t_{\zeta_{s}}) \,;\, 
  a_{0},\,a_{1},\,\dots,\,a_{s}\bigr)
\end{equation}
for all $x \in W_{\af}$. In particular, 
$\eta=S_{z_{\zeta_{s}}t_{\zeta_{s}}}\eta^{C}$, 
with $C$ the connected component containing $\eta$. 
\end{enu}
\end{rem}
%
%
\section{Characterization of Demazure subcrystals in terms of SiLS paths.}
\label{sec:main}
%
%
\subsection{Demazure subcrystals of $\B$.}
\label{subsec:demazure}

Let $\lambda \in P_{\af}$. 
For each $x \in W_{\af}$, we set
%
%
\begin{equation} \label{eq:dem}
V_{x}^{\pm}(\lambda):=U_{q}^{\pm}S_{x}^{\norm}v_{\lambda} 
\subset V(\lambda).
\end{equation}
Under the $U_{q}$-module isomorphism $V(\lambda) 
\stackrel{\sim}{\rightarrow} V(x\lambda)$ in \eqref{eq:isom-ext}, 
we have 
\begin{equation*}
\begin{split}
& V(\lambda) \supset V_{x}^{\pm}(\lambda)=
  U_{q}^{\pm}S_{x}^{\norm}v_{\lambda} \stackrel{\sim}{\rightarrow} \\
& \hspace*{15mm}
 U_{q}^{\pm}S_{x}^{\norm}S_{x^{-1}}^{\norm}v_{x\lambda} = U_{q}^{\pm}v_{x\lambda} =
 V_{e}^{\pm}(x\lambda)=:V^{\pm}(x\lambda) \subset V(x\lambda).
\end{split}
\end{equation*}
We know from \cite[\S2.8]{K-rims} that 
$V^{\pm}(x\lambda)=U_{q}^{\pm}v_{x\lambda}$ 
is compatible with the global basis of $V(x\lambda)$, 
that is, there exists a subset $\CB^{\pm}(x\lambda)$ 
of the crystal basis $\CB(x\lambda)$ such that 
%
%
\begin{equation} \label{eq:deme}
V^{\pm}(x\lambda) = 
\bigoplus_{b \in \CB^{\pm}(x\lambda)} \BQ(q_{s}) G(b) 
\subset
V(x\lambda) = 
\bigoplus_{b \in \CB(x\lambda)} \BQ(q_{s}) G(b).
\end{equation}
Since the $U_{q}$-module isomorphism $V(\lambda) 
\stackrel{\sim}{\rightarrow} V(x\lambda)$ in \eqref{eq:isom-ext} 
is compatible with the global bases, it follows that 
$V_{x}^{\pm}(\lambda)=U_{q}^{\pm}S_{x}^{\norm}v_{\lambda}$ 
is also compatible with the global basis of $V(\lambda)$; namely, 
if we define $\CB_{x}^{\pm}(\lambda)$ to be 
the inverse image of $\CB^{\pm}(x\lambda)$ under 
the isomorphism $S_{x}^{\ast}:\B \stackrel{\sim}{\rightarrow} \CB(x\lambda)$ of crystals, 
i.e., $\CB_{x}^{\pm}(\lambda):=S_{x^{-1}}^{\ast}(\CB^{\pm}(x\lambda))$, then 
%
%
\begin{equation} \label{eq:demc}
V_{x}^{\pm}(\lambda)=
\bigoplus_{b \in \CB_{x}^{\pm}(\lambda)} \BQ(q_{s}) G(b)
\subset
V(\lambda)=
\bigoplus_{b \in \CB(\lambda)} \BQ(q_{s}) G(b).
\end{equation}
Here, the subsets $\CB^{\pm}(x\lambda)$ 
in \eqref{eq:deme} can be described as follows 
(see page 234 of \cite{K-rims}): 
\begin{align}
& \CB^{+}(x\lambda)=
  \CB(x\lambda) \cap (u_{\infty} \otimes \tw{x\lambda} \otimes \CB(-\infty)), 
  \label{eq:be+} \\
& \CB^{-}(x\lambda)=
  \CB(x\lambda) \cap (\CB(\infty) \otimes \tw{x\lambda} \otimes u_{-\infty}). 
  \label{eq:be-}
\end{align}
From these, we obtain 
\begin{align}
& \CB_{x}^{+}(\lambda) =
    S_{x^{-1}}^{\ast}
      \Bigl( 
        \underbrace{\CB(x\lambda) \cap 
        (u_{\infty} \otimes \tw{x\lambda} \otimes \CB(-\infty))}_{%
         =\CB^{+}(x\lambda)}
      \Bigr),
  \label{eq:b+} \\[3mm]
& \CB_{x}^{-}(\lambda) = 
    S_{x^{-1}}^{\ast}
      \Bigl(
       \underbrace{\CB(x\lambda) \cap 
       (\CB(\infty) \otimes \tw{x\lambda} \otimes u_{-\infty})}_{%
         =\CB^{-}(x\lambda)} 
      \Bigr).
  \label{eq:b-}
\end{align}
%
%
\begin{rem} \label{rem:demazure}
From \eqref{eq:b-} (resp., \eqref{eq:b+}), using 
the tensor product rule for crystals, we see that 
the set $\CB_{x}^{-}(\lambda) \cup \bigl\{\bzero\bigr\}$ 
(resp., $\CB_{x}^{+}(\lambda) \cup \bigl\{\bzero\bigr\}$) is 
stable under the action of the Kashiwara operator $f_{j}$ (resp., $e_{j}$) 
for all $j \in I_{\af}$ (see also \cite[Lemma~2.6\,(i)]{K-rims}). 
\end{rem}
%
%
\begin{lem} \label{lem:dem=}
Let $\lambda \in P^{+}$, and set 
$J=J_{\lambda}:=
\bigl\{i \in I \mid \pair{\alpha_{i}^{\vee}}{\lambda}=0\bigr\}$. 
For $x,\,y \in W_{\af}$, 
%
%
\begin{equation} \label{eq:dem=}
V_{x}^{\pm}(\lambda) = V_{y}^{\pm}(\lambda) \iff 
\CB_{x}^{\pm}(\lambda) = \CB_{y}^{\pm}(\lambda) \iff 
x^{-1}y \in (W_{J})_{\af}.
\end{equation}
\end{lem}

\begin{proof}
It is obvious from the definitions that 
$V_{x}^{\pm}(\lambda) = V_{y}^{\pm}(\lambda)$ if and only if 
$\CB_{x}^{\pm}(\lambda) = \CB_{y}^{\pm}(\lambda)$. 
First we show that if $\CB_{x}^{\pm}(\lambda) = \CB_{y}^{\pm}(\lambda)$, 
then $x^{-1}y \in (W_{J})_{\af}$. 
We see from the definitions that the weights of elements in 
$\CB_{x}^{\pm}(\lambda)$ are all contained in $x\lambda \pm Q_{\af}^{+}$. 
Moreover, since $V_{x}^{\pm}(\lambda)_{x\lambda} = \BQ(q_{s})S_{x}^{\norm}v_{\lambda} = 
\BQ(q_{s})G(S_{x}u_{\lambda})$ by Remark~\ref{rem:Gext}, we deduce that 
$S_{x}u_{\lambda}$ is a unique element of weight $x\lambda$ 
in $\CB_{x}^{\pm}(\lambda)$. 
Similarly, the weights of elements in $\CB_{y}^{\pm}(\lambda)$ 
are all contained in $y\lambda \pm Q_{\af}^{+}$, and 
$S_{y}u_{\lambda}$ is a unique element of weight $y\lambda$
in $\CB_{y}^{\pm}(\lambda)$. 
Since $\CB_{x}^{\pm}(\lambda) = \CB_{y}^{\pm}(\lambda)$ by the assumption, 
we conclude from the above that $x\lambda=y\lambda$, and hence 
$\CB_{x}^{\pm}(\lambda)_{x\lambda}=\CB_{y}^{\pm}(\lambda)_{y\lambda}$. 
Because $\CB_{x}^{\pm}(\lambda)_{x\lambda}=\bigl\{S_{x}u_{\lambda}\bigr\}$ and 
$\CB_{y}^{\pm}(\lambda)_{y\lambda}=\bigl\{S_{y}u_{\lambda}\bigr\}$ as seen above, 
it follows immediately that $S_{x}u_{\lambda}=S_{y}u_{\lambda}$, and hence 
$S_{x^{-1}y}u_{\lambda}=u_{\lambda}$. 
Therefore, by \cite[Proposition~5.1.1]{INS}, 
we obtain $x^{-1}y \in (W_{J})_{\af}$, as desired. 

Next we show that 
if $x^{-1}y \in (W_{J})_{\af}$, 
then $V_{x}^{\pm}(\lambda) = V_{y}^{\pm}(\lambda)$. 
If $x^{-1}y \in (W_{J})_{\af}$, then we have 
$S_{x}u_{\lambda}=S_{y}u_{\lambda}$ by \cite[Proposition~5.1.1]{INS}, 
and hence 
\begin{equation*}
S_{x}^{\norm}v_{\lambda} = 
G(S_{x}u_{\lambda})=G(S_{y}u_{\lambda})= 
S_{y}^{\norm}v_{\lambda} \quad \text{by Remark~\ref{rem:Gext}}.
\end{equation*}
Therefore, we obtain 
$V_{x}^{\pm}(\lambda) = 
 U_{q}^{\pm}S_{x}^{\norm}v_{\lambda} = 
 U_{q}^{\pm}S_{y}^{\norm}v_{\lambda} = 
 V_{y}^{\pm}(\lambda)$, as desired. 
\end{proof}

\begin{rem}
Let $\lambda \in P^{+}$, and set $J=J_{\lambda}:=
\bigl\{i \in I \mid \pair{\alpha_{i}^{\vee}}{\lambda}=0\bigr\}$. 
\begin{enu}

\item 
We have $\bigl\{ x \in W_{\af} \mid x\lambda=\lambda \bigr\} \supseteq (W_J)_{\af}$, 
in which the equality holds if and only if 
$\lambda$ is of the form $\lambda=m\vpi_{i}$ 
for some $m \in \BZ_{\ge 0}$ and $i \in I$. 

\item
Let $x,\,y \in W_{\af}$. In view of Lemma~\ref{lem:dem=} and part (1), 
the equality $x\lambda=y\lambda$ does not necessarily imply the equality 
$V_{x}^{\pm}(\lambda) = V_{y}^{\pm}(\lambda)$. 
However, in this case, 
there exists a $U_{q}$-module automorphism 
$V(\lambda) \stackrel{\sim}{\rightarrow} V(\lambda)$ such that 
$v_{\lambda} \mapsto S_{x^{-1}y}^{\norm}v_{\lambda}$ (see \eqref{eq:isom-ext}). 
Under this automorphism, 
$V_{x}^{\pm}(\lambda)=U_{q}^{\pm}S_{x}^{\norm}v_{\lambda}$ is mapped to 
$U_{q}^{\pm}S_{x}^{\norm}S_{x^{-1}y}^{\norm}v_{\lambda} = 
 U_{q}^{\pm}S_{y}^{\norm}v_{\lambda} = V_{y}^{\pm}(\lambda)$. 
Thus we conclude that if $x\lambda=y\lambda$, then 
$V_{x}^{\pm}(\lambda)$ is conjugate to $V_{y}^{\pm}(\lambda)$ 
under a $U_{q}$-module automorphism of $V(\lambda)$. 
Similarly, if $x\lambda=y\lambda$, then 
$\CB_{x}^{\pm}(\lambda)$ is conjugate to $\CB_{y}^{\pm}(\lambda)$ 
under the crystal automorphism $S_{y^{-1}x}^{\ast}$ of $\CB(\lambda)$.
\end{enu}
\end{rem}
%
%
\subsection{Demazure subcrystals of $\sLS$.}
\label{subsec:demazurep}

Let $\lambda \in P^{+}$, and set 
$J=J_{\lambda}:=
\bigl\{i \in I \mid \pair{\alpha_{i}^{\vee}}{\lambda}=0\bigr\}$. 
For $\eta \in \sLS$, we define 
$\iota(\eta) \in (W^{J})_{\af}$ and 
$\kappa(\eta) \in (W^{J})_{\af}$ as in \eqref{eq:ik}. 
For each $x \in (W^{J})_{\af}$, we set 
%
%
\begin{equation} \label{eq:def-demp}
\begin{split}
& \BB^{\si}_{x \sige}(\lambda):=
 \bigl\{\eta \in \sLS \mid x \sige \iota(\eta) \bigr\}, \\
& \BB^{\si}_{\sige x}(\lambda):=
 \bigl\{\eta \in \sLS \mid \kappa(\eta) \sige x \bigr\}. 
\end{split}
\end{equation}
We are now ready to state 
the main result of this paper. 
%
%
\begin{thm} \label{thm:main}
For every $\lambda \in P^{+}$ and 
$x \in (W^{J_{\lambda}})_{\af}$, 
there hold the equalities 
%
%
\begin{equation} \label{eq:main}
\Psi_{\lambda}(\CB_{x}^{-}(\lambda))=\BB^{\si}_{\sige x}(\lambda)
\qquad \text{\rm and} \qquad
\Psi_{\lambda}^{\vee}(\CB_{x}^{+}(\lambda))=\BB^{\si}_{x \sige}(\lambda). 
\end{equation}
\end{thm}

\begin{rem}
In view of Lemma~\ref{lem:dem=}, together with Proposition~\ref{prop:P},
we may assume that $x \in (W^{J_{\lambda}})_{\af}$ in 
Theorem~\ref{thm:main}. 
\end{rem}

Here, let us remark that the second equality in \eqref{eq:main} 
follows from the first one in \eqref{eq:main}. 
Let $\lambda \in P^{+}$, and $x \in (W^{J_{\lambda}})_{\af}$; 
recall that $\mu:=-w_{0}\lambda \in P^{+}$, and $x^{\vee} \in (W^{J_{\mu}})_{\af}$ 
(for the definitions, see \S\ref{subsec:dual-sLS}). 
Since $S_{y}^{\ast} \circ \vee = \vee \circ S_{y}^{\ast}$ holds on 
$\ti{\CB}=\bigsqcup_{\mu \in P_{\af}} \CB(\infty) \otimes \CT_{\mu} \otimes \CB(-\infty)$ 
for all $y \in W_{\af}$ (see Remark~\ref{rem:ext-vee} and \eqref{eq:vee-ast}), 
we have 
\begin{align*}
S_{w_{0}}^{\ast}\bigl(\CB_{x}^{+}(\lambda)\bigr)^{\vee} 
 & = S_{w_{0}}^{\ast} \Bigl\{ S_{x^{-1}}^{\ast}
   \Bigl(
    \CB(x\lambda) \cap (u_{\infty} \otimes \tw{x\lambda} \otimes \CB(-\infty))
   \Bigr)\Bigr\}^{\vee} \quad \text{by \eqref{eq:b+}} \\
 & = S_{w_{0}}^{\ast}S_{x^{-1}}^{\ast} 
   \bigl\{
    \CB(x\lambda) \cap (u_{\infty} \otimes \tw{x\lambda} \otimes \CB(-\infty))
   \bigr\}^{\vee} \\
 & = S_{w_{0}x^{-1}}^{\ast}
   \bigl(
    \CB(-x\lambda) \cap (\CB(\infty) \otimes \tw{-x\lambda} \otimes u_{-\infty})
   \bigr) \\
 & = S_{w_{0}x^{-1}}^{\ast}
   \bigl(
    \CB(xw_0\mu) \cap (\CB(\infty) \otimes \tw{xw_0\mu} \otimes u_{-\infty}
   \bigr) \\
 & = \CB_{xw_{0}}^{-}(\mu) \quad \text{by \eqref{eq:b-}} \\
 & = \CB_{x^{\vee}}^{-}(\mu) \quad \text{by Lemma~\ref{lem:dem=}}. 
\end{align*}
Also, it is easily seen from the definitions \eqref{eq:etavee} and 
\eqref{eq:def-demp}, by using Lemma~\ref{lem:dual}, that 
$\BB^{\si}_{x \sige}(\lambda) = 
\bigl(\BB^{\si}_{\sige x^{\vee}}(\mu)\bigr)^{\vee}$. 
Therefore, if the equality $\Psi_{\mu}(\CB_{x^{\vee}}^{-}(\mu)) = 
\BB^{\si}_{\sige x^{\vee}}(\mu)$ holds, then the equality 
$\Psi_{\lambda}^{\vee}(\CB_{x}^{+}(\lambda))=\BB^{\si}_{x \sige}(\lambda)$
also holds by the definition \eqref{eq:Psivee} of $\Psi_{\lambda}^{\vee}$. 
Thus, the remaining task is to prove the first equality in \eqref{eq:main}. 

%
\section{Proof of Theorem~\ref{thm:main}.}
\label{sec:proof}
Throughout this section, 
we fix $\lambda=\sum_{i \in I} m_{i}\vpi_{i} \in P^{+}$, 
with $m_{i} \in \BZ_{\ge 0}$ for $i \in I$, 
and then set $J=J_{\lambda}:=\bigl\{ i \in I \mid 
\pair{\alpha_{i}^{\vee}}{\lambda}=0 \bigr\}$. 

%
\subsection{Fundamental properties of $\CB_{x}^{-}(\lambda)$, part~1.}
\label{subsec:properties1a}

Proposition~\ref{prop:key} and Corollary~\ref{cor:key} below 
are easy consequences of \cite[\S2.8]{K-rims}, but 
we include their proofs for the convenience of the reader. 
%
%
\begin{prop} \label{prop:key}
Let $x \in W_{\af}$ and $j \in I_{\af}$ 
be such that $\pair{\alpha_{j}^{\vee}}{x\lambda} \ge 0$. Then, 
%
%
\begin{equation} \label{lem2b}
\CB_{x}^{-}(\lambda)=
 \bigl\{e_{j}^{k}b \mid b \in \CB_{r_{j}x}^{-}(\lambda),\,k \in \BZ_{\ge 0} \bigr\}
 \setminus \{\bzero\}; 
\end{equation}
in particular, $\CB_{x}^{-}(\lambda) \supset \CB_{r_{j}x}^{-}(\lambda)$. 
Consequently, the set $\CB_{x}^{-}(\lambda) \cup \bigl\{\bzero\bigr\}$ is stable under 
the action of the Kashiwara operator $e_{j}$ for $j \in I_{\af}$ such that 
$\pair{\alpha_{j}^{\vee}}{x\lambda} \ge 0$
(see also Remark~\ref{rem:demazure}). 
\end{prop}

\begin{proof}
First, we show the equality 
%
%
\begin{equation} \label{eq:lem2-1}
U_{q}^{-}S_{x}^{\norm}v_{\lambda}=
U_{q}^{(j)} U_{q}^{-}S_{r_{j}x}^{\norm}v_{\lambda}.
\end{equation}
Here, $U_{q}^{(j)}$ denotes the $\BQ(q_{s})$-subalgebra of $U_{q}$ 
generated by $E_{j}$, $F_{j}$, and $q^{\alpha_{j}^{\vee}}$, 
which is isomorphic to the quantized universal enveloping algebra 
associated with $\Fsl_{2}$; recall the triangular decomposition 
$U_{q}^{(j)} = 
\langle E_{j} \rangle 
\langle q^{\alpha_{j}^{\vee}} \rangle
\langle F_{j} \rangle$ of $U_{q}^{(j)}$. 
We show the inclusion $\supset$ as follows: 
\begin{align*}
& U_{q}^{(j)} U_{q}^{-}S_{r_{j}x}^{\norm}v_{\lambda}
  = U_{q}^{(j)} U_{q}^{-}S_{r_{j}}^{\norm}S_{x}^{\norm}v_{\lambda} 
  = U_{q}^{(j)} U_{q}^{-}F_{j}^{(\pair{\alpha_{j}^{\vee}}{x\lambda})}S_{x}^{\norm}v_{\lambda}
  \tag{$\ast$} \label{eq:ast1} \\
& \quad \subset U_{q}^{(j)} U_{q}^{-}S_{x}^{\norm}v_{\lambda} = 
\langle E_{j} \rangle 
\langle q^{\alpha_{j}^{\vee}} \rangle
\langle F_{j} \rangle 
U_{q}^{-}S_{x}^{\norm}v_{\lambda} \\
& \quad = \langle E_{j} \rangle U_{q}^{-}S_{x}^{\norm}v_{\lambda} 
  \subset U_{q}^{-} \langle E_{j} \rangle S_{x}^{\norm}v_{\lambda} 
\quad \text{since $[E_{j},\,F_{l}] \in \langle q^{\alpha_{j}^{\vee}} \rangle$ 
 for all $l \in I_{\af}$} \\
& \quad = U_{q}^{-}S_{x}^{\norm}v_{\lambda} 
\quad \text{since $E_{j}S_{x}^{\norm}v_{\lambda}=0$ 
  by the assumption $\pair{\alpha_{j}^{\vee}}{x\lambda} \ge 0$};
\end{align*}
the second equality in \eqref{eq:ast1} 
follows from the fact that $S_{x}^{\norm}v_{\lambda}$ 
is an extremal weight vector of weight $x\lambda$, and 
the assumption that $\pair{\alpha_{j}^{\vee}}{x\lambda} \ge 0$.
Similarly, we show the opposite inclusion $\subset$ as follows: 
\begin{align*}
U_{q}^{-}S_{x}^{\norm}v_{\lambda}
 & = U_{q}^{-}S_{r_{j}}^{\norm}S_{r_{j}x}^{\norm}v_{\lambda} 
   = U_{q}^{-}E_{j}^{(\pair{\alpha_{j}^{\vee}}{x\lambda})}S_{r_{j}x}^{\norm}v_{\lambda} \\
 & \subset U_{q}^{-} \langle E_{j} \rangle S_{r_{j}x}^{\norm}v_{\lambda}
   \subset \langle E_{j} \rangle U_{q}^{-} S_{r_{j}x}^{\norm}v_{\lambda} 
   \subset U_{q}^{(j)} U_{q}^{-}S_{r_{j}x}^{\norm}v_{\lambda}. 
\end{align*}
Thus we obtain the equality \eqref{eq:lem2-1}.
Since the left-hand side of \eqref{eq:lem2-1} is identical to 
$V^{-}_{x}(\lambda)$, we see that $U_{q}^{(j)} U_{q}^{-}S_{r_{j}x}^{\norm}v_{\lambda}$ 
is compatible with the global basis of $V(\lambda)$, and has the crystal basis 
$\CB^{-}_{x}(\lambda)$. 

Now, recall from \eqref{eq:isom-ext} that there exists a $U_{q}$-module 
isomorphism $V(\lambda) \stackrel{\sim}{\rightarrow} V(r_{j}x\lambda)$ 
that maps $v_{\lambda}$ to $S^{\norm}_{x^{-1}r_{j}}v_{r_{j}x\lambda}$. 
Under this isomorphism, the right-hand side 
$U_{q}^{(j)} U_{q}^{-}S_{r_{j}x}^{\norm}v_{\lambda}$ $(\subset V(\lambda))$ 
of \eqref{eq:lem2-1} is mapped to 
$U_{q}^{(j)} U_{q}^{-} v_{r_{j}x\lambda} = 
U_{q}^{(j)} V^{-}(r_{j}x\lambda)$ $(\subset V(r_{j}x\lambda))$. 
When we regard $V(r_{j}x\lambda)$ as a $U_{q}^{(j)}$-module by restriction, 
$V^{-}(r_{j}x\lambda)$ is regarded as an $\langle F_{j} \rangle$-module, 
which is compatible with the global basis of $V(r_{j}x\lambda)$. 
Therefore, we deduce from (the dual version of) 
\cite[Eq.\,(2.28) and comments following it]{K-rims} 
that $U_{q}^{(j)} V^{-}(r_{j}x\lambda) = 
 \langle E_{j} \rangle V^{-}(r_{j}x\lambda)$ is also compatible 
with the global basis of $V(r_{j}x)$, and has the crystal basis 
\begin{equation*}
\bigl\{e_{j}^{k}b \mid b \in \CB^{-}(r_{j}x\lambda),\,k \in \BZ_{\ge 0} \bigr\} 
\setminus \bigl\{\bzero\bigr\}.
\end{equation*}
Recall that the $U_{q}$-module 
isomorphism $V(\lambda) \stackrel{\sim}{\rightarrow} V(r_{j}x\lambda)$ 
is compatible with the global bases, and induces the crystal isomorphism 
$S_{r_{j}x}^{\ast}: \CB(\lambda) \stackrel{\sim}{\rightarrow} \CB(r_{j}x\lambda)$.
Namely, we have the following correspondences of modules and their crystal bases:
\begin{equation*}
\begin{array}{ccc}
V(\lambda) & \stackrel{\sim}{\longrightarrow} & V(r_{j}x\lambda) \\
\rotatebox[origin=c]{90}{$\subset$} & & \rotatebox[origin=c]{90}{$\subset$} \\
U_{q}^{(j)} U_{q}^{-}S_{r_{j}x}^{\norm}v_{\lambda} 
& \stackrel{\sim}{\longrightarrow} &
U_{q}^{(j)} V^{-}(r_{j}x\lambda), \\[1mm]
(=V_{x}^{-}(\lambda)) & & \\[3mm]

\CB_{x}^{-}(\lambda) & \stackrel{S_{r_{j}x}^{\ast}}{\longrightarrow} & 
\bigl\{e_{j}^{k}b \mid b \in \CB^{-}(r_{j}x\lambda),\,k \in \BZ_{\ge 0} \bigr\}
\setminus \bigl\{\bzero\bigr\}. 
\end{array}
\end{equation*}
From this, we conclude that 
\begin{align*}
\CB_{x}^{-}(\lambda) & = 
(S_{r_{j}x}^{\ast})^{-1}
 \bigl\{e_{j}^{k}b \mid b \in \CB^{-}(r_{j}x\lambda),\,k \in \BZ_{\ge 0} \bigr\}
 \setminus \bigl\{\bzero\bigr\} \\
& = 
 \bigl\{e_{j}^{k}(S_{r_{j}x}^{\ast})^{-1}(b) \mid 
 b \in \CB^{-}(r_{j}x\lambda),\,k \in \BZ_{\ge 0}\bigr\}
 \setminus \bigl\{\bzero\bigr\} \\
& =  
 \bigl\{e_{j}^{k}b \mid 
 b \in \CB_{r_{j}x}^{-}(\lambda),\,k \ge 0\bigr\}
 \setminus \bigl\{\bzero\bigr\} \quad \text{by \eqref{eq:b-}}. 
\end{align*}
This completes the proof of the proposition. 
\end{proof}
%
%
\begin{cor} \label{cor:key}
Let $x \in W_{\af}$, and $j \in I_{\af}$. 
For every $b \in \CB_{x}^{-}(\lambda)$, we have 
$f_{j}^{\max}b:=f_{j}^{\vp_j(b)}b \in \CB_{r_{j}x}^{-}(\lambda)$, 
where $\vp_{j}(b):=\max \bigl\{k \in \BZ_{\ge 0} \mid f_{j}^{k}b \ne \bzero \bigr\}$. 
\end{cor}

\begin{proof}
If $\pair{\alpha_{j}^{\vee}}{x\lambda} \ge 0$, 
then the assertion follows from Proposition~\ref{prop:key} 
and Remark~\ref{rem:demazure}. 
Assume that $\pair{\alpha_{j}^{\vee}}{x\lambda} < 0$. 
Let $b \in \CB_{x}^{-}(\lambda)$. 
We see from Remark~\ref{rem:demazure} that 
$f_{j}^{\max}b \in \CB_{x}^{-}(\lambda)$. 
Also, it follows from Proposition~\ref{prop:key} 
that $\CB_{x}^{-}(\lambda) \subset \CB_{r_{j}x}^{-}(\lambda)$.
Combining these, we obtain 
$f_{j}^{\max}b \in \CB_{x}^{-}(\lambda)
 \subset \CB_{r_{j}x}^{-}(\lambda)$, as desired. 
\end{proof}

%
\subsection{Fundamental properties of $\CB_{x}^{-}(\lambda)$, part~2.}
\label{subsec:properties1b}

Now we recall some results in \cite{K-lv0} and \cite{BN}.
We define a $\BQ(q_{s})$-subalgebra $U_{q}'$ of $U_{q}$ by
\begin{equation*}
U_{q}':=\langle E_{j},\,F_{j},\,q^{h} \mid 
j \in I_{\af},\,h \in d^{-1}
(\ts{\bigoplus_{j \in I_{\af}}}\BZ \alpha_{j}^{\vee})\rangle \subset U_{q}, 
\end{equation*}
which can be thought of as 
the quantized universal enveloping algebra
$U_{q}(\Fg_{\af}')$ associated with the derived subalgebra 
$\Fg_{\af}':=[\Fg_{\af},\,\Fg_{\af}]$ of $\Fg_{\af}$. 
We know from \cite[p.\,142]{K-lv0} that 
for each $i \in I$, there exists a $U_{q}'$-module automorphism 
$z_{i}:V(\vpi_i) \rightarrow V(\vpi_i)$ 
that maps $v_{\vpi_{i}}$ to $v_{\vpi_{i}}^{[1]}$, 
where for $k \in \BZ$, we denote by $u_{\vpi_{i}}^{[k]}$ 
the (unique) element of weight $\vpi_{i}+k\delta$ in $\CB(\vpi_{i})$, and 
set $v_{\vpi_{i}}^{[k]}:=G(u_{\vpi_{i}}^{[k]})$ (see \cite[Proposition~5.8]{K-lv0});
note that $z_{i}$ commutes with the Kashiwara operators 
$e_{j}$, $f_{j}$, $j \in I_{\af}$, on $V(\vpi_{i})$. 
Moreover, we see from \cite[Propositions~5.12 and 5.15]{K-lv0} 
that $z_{i}$ preserves the crystal lattice $\CL(\vpi_{i})$ of $V(\vpi_{i})$, 
and hence induces a $\BQ$-linear automorphism 
$z_{i}:\CL(\vpi_{i})/q_{s}\CL(\vpi_{i}) \stackrel{\sim}{\rightarrow}
 \CL(\vpi_{i})/q_{s}\CL(\vpi_{i})$. 
Because $z_{i}$ commutes with the Kashiwara operators 
$e_{j}$, $f_{j}$, $j \in I_{\af}$, 
on $\CL(\vpi_{i})/q_{s}\CL(\vpi_{i})$, and 
because $z_{i}(u_{\vpi_{i}})=u_{\vpi_{i}}^{[1]} \in \CB(\vpi_{i})$, 
it follows immediately 
from \cite[Proposition~5.12]{K-lv0} that 
$z_{i}$ preserves the crystal basis $\CB(\vpi_{i}) \subset 
\CL(\vpi_{i})/q_{s}\CL(\vpi_{i})$ of $V(\vpi_{i})$ 
(see also \cite[Theorem~5.17]{K-lv0}). 

Recall that $\lambda \in P^{+}$ is of the form 
$\lambda = \sum_{i \in I} m_{i} \vpi_{i}$, 
with $m_{i} \in \BZ_{\ge 0}$ for $i \in I$. 
We fix an arbitrary total ordering on $I$, and then set 
$\ti{V}(\lambda):=
\bigotimes_{i \in I} V(\vpi_{i})^{\otimes m_{i}}$. 
We can easily show (see also \cite[the comment preceding Eq.\,(4.8)]{BN}) that 
$\ti{v}_{\lambda} := \ts{\bigotimes_{i \in I} v_{\vpi_{i}}^{\otimes m_{i}}}
\in \ti{V}(\lambda)$ is an extremal weight vector of weight $\lambda$. 
By \cite[Eq.\,(4.8) and Corollary~4.15]{BN}, 
there exists a $U_{q}$-module embedding 
%
%
\begin{equation} \label{eq:Phi}
\Phi_{\lambda}:V(\lambda) \hookrightarrow 
\ti{V}(\lambda)=
\bigotimes_{i \in I} V(\vpi_{i})^{\otimes m_{i}}
\end{equation}
that maps $v_{\lambda}$ to $\ti{v}_{\lambda}:=
\bigotimes_{i \in I} v_{\vpi_{i}}^{\otimes m_{i}}$. 
%
%
\begin{rem} \label{rem:vlambda}
We can show by induction on $\ell(x)$ 
that for every $x \in W_{\af}$, 
\begin{equation*}
S_{x}^{\norm}\ti{v}_{\lambda}=
S_{x}^{\norm}
 \bigl(\ts{\bigotimes_{i \in I} v_{\vpi_{i}}^{\otimes m_{i}}}\bigr) \in 
\BQ(q_{s})
\Bigl(\ts{\bigotimes_{i \in I} 
 (S_{x}^{\norm}v_{\vpi_{i}})^{\otimes m_{i}}}
\Bigr); 
\end{equation*}
cf.~\cite[Lemma~1.6\,(1)]{AK}.
Therefore, under the $U_{q}$-module embedding 
$\Phi_{\lambda}:V(\lambda) \hookrightarrow \ti{V}(\lambda)$ 
in \eqref{eq:Phi}, $V^{-}_{x}(\lambda) \subset V(\lambda)$ 
for $x \in W_{\af}$ is mapped as follows: 
\begin{align}
& V^{-}_{x}(\lambda) = 
  U_{q}^{-}S_{x}^{\norm}v_{\lambda} 
  \stackrel{\Phi_{\lambda}}{\hookrightarrow} \nonumber \\
& \qquad 
U_{q}^{-}S_{x}^{\norm}\ti{v}_{\lambda} = 
U_{q}^{-}
\Bigl(\ts{\bigotimes_{i \in I} 
 (S_{x}^{\norm}v_{\vpi_{i}})^{\otimes m_{i}}}
\Bigr) \subset 
\bigotimes_{i \in I} 
(\underbrace{U_{q}^{-}S_{x}^{\norm}v_{\vpi_{i}}}_{=V_{x}^{-}(\vpi_{i})})^{\otimes m_{i}}. 
\label{eq:emb}
\end{align}
\end{rem}

Recall that $\ti{V}(\lambda)$ has the crystal basis 
$(\ti{\CL}(\lambda),\,\ti{\CB}(\lambda))$, where 
\begin{equation*}
\ti{\CL}(\lambda) : =
 \bigotimes_{i \in I} \CL(\vpi_{i})^{\otimes m_{i}}, \qquad
\ti{\CB}(\lambda) : =
 \bigotimes_{i \in I} \CB(\vpi_{i})^{\otimes m_{i}}. 
\end{equation*}
We see from \cite[page 369, the 2nd line from below]{BN} that
$\Phi_{\lambda}(\CL(\lambda)) \subset \ti{\CL}(\lambda)$, 
and hence $\Phi_{\lambda}$ induces a $\BQ$-linear embedding of 
$\CL(\lambda)/q_{s}\CL(\lambda)$ into 
$\ti{\CL}(\lambda)/q_{s}\ti{\CL}(\lambda)$, which we denote 
by $\Phi_{\lambda}|_{q=0}$. Note that we have the following 
commutative diagram for all $j \in I_{\af}$:
\begin{equation*}
\begin{CD}
\CL(\lambda)/q_{s}\CL(\lambda) 
 @>{\Phi_{\lambda}|_{q=0}}>> 
\ti{\CL}(\lambda)/q_{s}\ti{\CL}(\lambda) \\
@V{e_{j},\,f_{j}}VV @VV{e_{j},\,f_{j}}V \\
\CL(\lambda)/q_{s}\CL(\lambda) 
 @>{\Phi_{\lambda}|_{q=0}}>> 
\ti{\CL}(\lambda)/q_{s}\ti{\CL}(\lambda).
\end{CD}
\end{equation*}

For each $i \in I$ and $1 \le l \le m_{i}$, we define $z_{i,l}$ to be 
the $U_{q}'$-module automorphism of $\ti{V}(\lambda)$ 
which acts as $z_{i}$ only on the $l$-th factor of $V(\vpi_{i})^{\otimes m_{i}}$ 
in $\ti{V}(\lambda)$, and as the identity map on the other factors of $\ti{V}(\lambda)$; 
notice that $z_{i,l}$ commutes with the Kashiwara operators 
$e_{j}$, $f_{j}$, $j \in I_{\af}$, on $\ti{V}(\lambda)$. 
Since $z_{i,l}$ preserves $\ti{\CL}(\lambda)$ by the definition above, 
we deduce that $z_{i,l}$ induces a $\BQ$-linear automorphism  
$z_{i,l} : \ti{\CL}(\lambda)/q_{s}\ti{\CL}(\lambda) \rightarrow 
\ti{\CL}(\lambda)/q_{s}\ti{\CL}(\lambda)$, which commutes with 
the Kashiwara operators $e_{j}$, $f_{j}$, $j \in I_{\af}$, 
on $\ti{\CL}(\lambda)/q_{s}\ti{\CL}(\lambda)$. 
Also, notice that the $z_{i,l}$'s, $i \in I$, $1 \le l \le m_{i}$, 
commute with each other. For $\bc=(\rho^{(i)})_{i \in I} \in \Par(\lambda)$, 
we set
%
%
\begin{equation} \label{eq:Schur}
s_{\bc}(z^{-1}):=\prod_{i \in I} 
s_{\rho^{(i)}}(z_{i,1}^{-1},\,\dots,\,z_{i,m_i}^{-1})
 \in \End_{\BQ}(\ti{\CL}(\lambda)/q_{s}\ti{\CL}(\lambda)). 
\end{equation}
Here, for a partition 
$\rho=(\rho_{1} \ge \cdots \ge \rho_{m-1} \ge 0)$ 
of length less than $m \in \BZ_{\ge 1}$, 
$s_{\rho}(x)=s_{\rho}(x_{1},\,\dots,\,x_{m})$ denotes 
the Schur polynomial in the variables $x_{1},\,\dots,\,x_{m}$ 
corresponding to the partition $\rho$, that is, 
the character of the finite-dimensional, irreducible, 
polynomial representation of $\GL(m)$ whose highest weight 
corresponds to $\rho$ (see \cite[\S8.2]{F}); note that 
for each $\nu = (\nu_{1},\,\dots,\,\nu_{m}) \in \BZ_{\ge 0}^{m}$, 
the coefficient $c_{\nu}$ of 
$x^{\nu}=x_{1}^{\nu_{1}} \cdots x_{m}^{\nu_{m}}$ in $s_{\rho}(x)$ is 
equal to the dimension of the $\nu$-weight space, and in particular, 
$c_{\nu_{\rho}}=1$ for $\nu_{\rho}:=(\rho_{1},\,\dots,\,\rho_{m-1},\,0) 
\in \BZ_{\ge 0}^{m}$, which is the highest weight. In particular, we have 
%
%
\begin{equation} \label{eq:srho}
s_{\rho}(x) = x^{\nu_{\rho}} + 
\sum_{\nu \in \BZ_{\ge 0}^{m},\,\nu \ne \nu_{\rho}} c_{\nu} x^{\nu}. 
\end{equation}

By \cite[Proposition~4.13]{BN}, 
together with the fact that ``$\mathrm{sgn}(\mathbf{c},\,p)=1$'' 
shown on page 375 of \cite{BN}, 
the image of the crystal basis $\CB(\lambda) \subset 
\CL(\lambda)/q_{s}\CL(\lambda)$ under the $\BQ$-linear embedding 
$\Phi_{\lambda}|_{q=0} : \CL(\lambda)/q_{s}\CL(\lambda) \hookrightarrow 
\ti{\CL}(\lambda)/q_{s}\ti{\CL}(\lambda)$ is identical to 
%
%
\begin{equation} \label{eq:BN413}
\bigl\{ s_{\bc}(z^{-1})b \mid 
 \bc \in \Par(\lambda),\,b \in \ti{\CB}_{0}(\lambda) \bigr\}
\subset \ti{\CL}(\lambda)/q_{s}\ti{\CL}(\lambda), 
\end{equation}
where $\ti{\CB}_{0}(\lambda)$ denotes the connected component of 
$\ti{\CB}(\lambda)$ containing $\ti{u}_{\lambda}:=
\bigotimes_{i \in I} u_{\vpi_{i}}^{\otimes m_{i}}$, and 
%
%
\begin{equation} \label{eq:scu}
\Phi_{\lambda}|_{q=0} (u^{\bc}) = 
 s_{\bc}(z^{-1})\ti{u}_{\lambda} \quad 
\text{for each $\bc \in \Par(\lambda)$};
\end{equation}
recall from \S\ref{subsec:isom} that 
$u^{\bc} = b(\bc) \otimes \tw{\lambda} \otimes u_{-\infty}$ 
is an extremal element of weight $\lambda-|\bc|\delta$ contained in 
$\CB(\lambda)$. 
For later use, we rewrite the right-hand side 
$s_{\bc}(z^{-1})\ti{u}_{\lambda}$ of \eqref{eq:scu} as follows. 
We set 
\begin{equation*}
\BZ_{\ge 0}^{\lambda}:=\bigl\{
\bm{\nu}=(\nu^{(i)})_{i \in I} = 
(\nu^{(i)}_{1},\,\dots,\,\nu^{(i)}_{m_{i}})_{i \in I} \mid 
\nu^{(i)} \in \BZ_{\ge 0}^{m_{i}},\,i \in I \bigr\},
\end{equation*}
and then 
\begin{equation*}
\ti{u}_{\lambda}^{[\bm{\nu}]}: = 
\ts{
\bigotimes_{i \in I} 
(u_{\vpi_{i}}^{[-\nu^{(i)}_{1}]} \otimes \cdots \otimes 
 u_{\vpi_{i}}^{[-\nu^{(i)}_{m_i}]})
} \in \ti{\CB}(\lambda) \quad 
\text{for 
$\bm{\nu}=(\nu^{(i)}_{1},\,\dots,\,\nu^{(i)}_{m_{i}})_{i \in I} 
 \in \BZ_{\ge 0}^{\lambda}$}.
\end{equation*}
Also, if $\bc=(\rho^{(i)})_{i \in I} \in \Par(\lambda)$, 
with $\rho^{(i)}=(\rho^{(i)}_{1} \ge \cdots \ge 
\rho^{(i)}_{m_{i}-1} \ge 0)$ for $i \in I$, then 
we set $\bm{\nu}_{\bc}:=(\nu_{\rho^{(i)}})_{i \in I} = 
(\rho^{(i)}_{1},\,\dots,\,\rho^{(i)}_{m_{i}-1},\,0)_{i \in I} 
\in \BZ_{\ge 0}^{\lambda}$. 
Since $z_{i}$ maps $u_{\vpi_{i}} \in \CB(\vpi_{i})$ to 
$u_{\vpi_{i}}^{[1]} \in \CB(\vpi_{i})$ for each $i \in I$, 
we deduce from \eqref{eq:srho} that 
%
%
\begin{equation} \label{eq:scu2}
s_{\bc}(z^{-1})\ti{u}_{\lambda} = 
\ti{u}_{\lambda}^{[\bm{\nu}_{\bc}]} + 
\sum_{
 \bm{\nu} \in \BZ_{\ge 0}^{\lambda},\, 
 \bm{\nu} \ne \bm{\nu}_{\bc}} c_{\bm{\nu}} \ti{u}_{\lambda}^{[\bm{\nu}]},
\end{equation}
where $c_{\bm{\nu}}= \prod_{i \in I} c_{\nu^{(i)}} \in \BZ_{\ge 0}$ if 
$\bm{\nu}=(\nu^{(i)})_{i \in I} \in \BZ_{\ge 0}^{\lambda}$. 
%
%
\begin{prop} \label{prop:zt}
Let $\bc \in \Par(\lambda)$, and $x,\,y \in (W^{J})_{\af}$. 
Then, 
%
%
\begin{equation} \label{eq:zt}
S_{y}(u^{\bc}) \in \CB_{x}^{-}(\lambda) 
\iff
y \sige x.
\end{equation}
\end{prop}

In order to show the proposition above, we need some lemmas. 
%
%
\begin{lem} \label{lem:zt0}
Let $x,\,y \in (W^{J})_{\af}$. If $y \sige x$, 
then $\CB_{y}^{-}(\lambda) \subset \CB_{x}^{-}(\lambda)$. 
\end{lem}

\begin{proof}
We may assume that $x \edge{\beta} y$ in $\SB$ for some $\beta \in \prr$. 
Write $x \in (W^{J})_{\af}$ in the form  $x = wz_{\xi}t_{\xi}$, 
with $w \in W^{J}$ and $\xi \in \jad$ (see \eqref{eq:W^J_af}). 
By Remark~\ref{rem:SiB}, $\beta$ is either of the following forms: 
$\beta=\alpha$ for some $\alpha \in \Delta^{+}$, or 
$\beta=\alpha+\delta$ for some $-\alpha \in \Delta^{+}$; 
in both cases, we have 
$w^{-1}\alpha \in \Delta^{+} \setminus \Delta_{J}^{+}$, and hence 
$\pair{\beta^{\vee}}{x\lambda}= 
\pair{\alpha^{\vee}}{wz_{\xi}t_{\xi}\lambda}=
\pair{w^{-1}\alpha^{\vee}}{\lambda} > 0$. Thus we obtain 
$\pair{\beta^{\vee}}{y\lambda} = \pair{\beta^{\vee}}{r_{\beta}x\lambda} < 0$. 
Also, by \eqref{eq:b-}, we have 
\begin{align*}
\CB_{y}^{-}(\lambda) & =
  S_{y^{-1}}^{\ast}
  \bigl( \CB(y\lambda) \cap 
  (\CB(\infty) \otimes \tw{y\lambda} \otimes u_{-\infty}) \bigr), \\
\CB_{x}^{-}(\lambda)=
\CB_{r_{\beta}y}^{-}(\lambda) & =
  S_{(r_{\beta}y)^{-1}}^{\ast}
  \bigl( \CB(r_{\beta}y\lambda) \cap 
  (\CB(\infty) \otimes \tw{r_{\beta}y\lambda} \otimes u_{-\infty}) \bigr) \\
& = 
  S_{y^{-1}}^{\ast}S_{r_{\beta}}^{\ast}
  \bigl( \CB(r_{\beta}y\lambda) \cap 
  (\CB(\infty) \otimes \tw{r_{\beta}y\lambda} \otimes u_{-\infty}) \bigr).
\end{align*}
Therefore, in order to show that $\CB_{x}^{-}(\lambda) \supset \CB_{y}^{-}(\lambda)$, 
it suffices to show that
%
%
\begin{equation} \label{eq:bruhat1}
S_{r_{\beta}}^{\ast}
\bigl( \CB(r_{\beta}\mu) \cap 
(\CB(\infty) \otimes \tw{r_{\beta}\mu} \otimes u_{-\infty}) \bigr)
\subset
\CB(\mu) \cap (\CB(\infty) \otimes \tw{\mu} \otimes u_{-\infty}),
\end{equation}
where we set $\mu:=r_{\beta}y\lambda$.
The set on the right-hand side of the inclusion $\subset$ in 
\eqref{eq:bruhat1} is the crystal basis of $V^{-}(\mu)=
U_{q}^{-}v_{\mu}$ (see \eqref{eq:be-}), that is, 
%
%
\begin{equation} \label{eq:Bru1-1}
\CB(\mu) \cap 
(\CB(\infty) \otimes \tw{\mu} \otimes u_{-\infty})=
\bigl\{b \in \CB(\mu) \mid G(b) \in V^{-}(\mu)\bigr\}.
\end{equation}
The set on the left-hand side of the inclusion $\subset$ 
in \eqref{eq:bruhat1} is the crystal basis of 
$V_{r_{\beta}}^{-}(\mu)=U_{q}^{-}S_{r_{\beta}}^{\norm}v_{\mu}$ 
by \eqref{eq:b-}, that is, 
%
%
\begin{equation} \label{eq:Bru1-2}
S_{r_{\beta}}^{\ast}
\bigl( \CB(r_{\beta}\mu) \cap 
 (\CB(\infty) \otimes \tw{r_{\beta}\mu} \otimes u_{-\infty}) \bigr) =
\bigl\{b \in \CB(\mu) \mid G(b) \in V_{r_{\beta}}^{-}(\mu)\bigr\}.
\end{equation}
Since $\pair{\beta^{\vee}}{\mu} = - \pair{\beta^{\vee}}{y\lambda} > 0$ as shown above, 
it follows from \cite[Proposition~2.8]{K-rims} that
$S_{r_{\beta}}^{\norm}v_{\mu} \in V^{-}(\mu)=U_{q}^{-}v_{\mu}$, 
which implies that $V^{-}(\mu) \supset V_{r_{\beta}}^{-}(\mu)$. 
Combining this containment with \eqref{eq:Bru1-1} and \eqref{eq:Bru1-2}, 
we conclude \eqref{eq:bruhat1}. This proves the lemma. 
\end{proof}
%
%
\begin{lem} \label{lem:zt01}
For each $y \in W_{\af}$, we have
$S_{y}(u^{\bc}) \in \CB_{y}^{-}(\lambda)$. 
\end{lem}

\begin{proof}
We prove the assertion by induction on $\ell(y)$. 
If $\ell(y)=0$, then $y=e$, and hence the assertion 
follows immediately from Proposition~\ref{prop:ext}, 
\eqref{eq:b-}, and the definition of $u^{\bc}$. 
Assume now that $\ell(y) > 0$, and take $j \in I_{\af}$ such that 
$\ell(r_{j}y) = \ell(y) -1$; by our induction hypothesis, 
we have $S_{r_{j}y}(u^{\bc}) \in \CB_{r_{j}y}^{-}(\lambda)$. 
If $\pair{\alpha_{j}^{\vee}}{r_{j}y\lambda} \ge 0$, then
$S_{y}(u^{\bc}) = S_{r_{j}}S_{r_{j}y}(u^{\bc})=f_{j}^{\max}S_{r_{j}y}(u^{\bc})$ 
since $u^{\bc}$ is an extremal element. 
Therefore, we deduce $S_{y}(u^{\bc}) \in \CB_{y}^{-}(\lambda)$ 
by Corollary~\ref{cor:key}, 
since $S_{r_{j}y}(u^{\bc}) \in \CB_{r_{j}y}^{-}(\lambda)$
by our induction hypothesis. 
If $n:=\pair{\alpha_{j}^{\vee}}{r_{j}y\lambda} \le 0$, then 
$S_{y}(u^{\bc}) = S_{r_{j}}S_{r_{j}y}(u^{\bc}) = e_{j}^{-n}S_{r_{j}y}(u^{\bc})$. 
Since $\pair{\alpha_{j}^{\vee}}{y\lambda} \ge 0$, and 
$S_{r_{j}y}(u^{\bc}) \in \CB_{r_{j}y}^{-}(\lambda)$ by our induction hypothesis, 
it follows from Proposition~\ref{prop:key} that
$S_{y}(u^{\bc}) \in \CB_{y}^{-}(\lambda)$. This proves the lemma. 
\end{proof}

\begin{proof}[Proof of Proposition~\ref{prop:zt}]
The ``if'' part follows immediately from 
Lemmas~\ref{lem:zt0} and \ref{lem:zt01}. 
Indeed, assume that $y \sige x$. 
Then, $\CB_{y}^{-}(\lambda) \subset \CB_{x}^{-}(\lambda)$ 
by Lemma~\ref{lem:zt0}. Therefore, by Lemma~\ref{lem:zt01}, 
$S_{y}(u^{\bc}) \in \CB_{y}^{-}(\lambda) \subset \CB_{x}^{-}(\lambda)$. 
This proves the ``if'' part. 

Now, we prove the ``only if'' part. 
%
%
\begin{claim} \label{c:zt1}
Let $\bc \in \Par(\lambda)$, and $\xi,\,\zeta \in \jad$.  
If $S_{z_{\xi}t_{\xi}}(u^{\bc}) \in \CB_{z_{\zeta}t_{\zeta}}^{-}(\lambda)$, 
then $z_{\xi}t_{\xi} \sige z_{\zeta}t_{\zeta}$.
\end{claim}

\noindent
{\it Proof of Claim~\ref{c:zt1}.}
By \cite[Proposition~6.2.2]{INS}, 
it suffices to show that $[\xi-\zeta] \in \QJp{I \setminus J}$, 
where $[\,\cdot\,]:Q^{\vee} = Q^{\vee}_{I \setminus J} \oplus Q_{J}^{\vee} 
\twoheadrightarrow Q^{\vee}_{I \setminus J}$ is the projection 
(see \eqref{eq:prj}). 
By \eqref{eq:scu} and \eqref{eq:scu2}, we have 
%
%
\begin{equation} \label{eq:zt1-1a}
\Phi_{\lambda}|_{q=0}(S_{z_{\xi}t_{\xi}}(u^{\bc})) = 
S_{z_{\xi}t_{\xi}}(s_{\bc}(z^{-1})\ti{u}_{\lambda}) = 
\underbrace{S_{z_{\xi}t_{\xi}} (\ti{u}_{\lambda}^{[\bm{\nu}_{\bc}]})}_{\in \ti{\CB}(\lambda)} 
+ 
\sum_{\bm{\nu} \in \BZ_{\ge 0}^{\lambda},\,\bm{\nu} \ne \bm{\nu}_{\bc}} 
c_{\bm{\nu}} 
\underbrace{S_{z_{\xi}t_{\xi}} (\ti{u}_{\lambda}^{[\bm{\nu}]})}_{\in \ti{\CB}(\lambda)}; 
\end{equation}
notice that for $\bm{\nu},\,\bm{\nu}' \in \BZ_{\ge 0}^{\lambda}$, 
$S_{z_{\xi}t_{\xi}} (\ti{u}_{\lambda}^{[\bm{\nu}]}) = 
S_{z_{\xi}t_{\xi}} (\ti{u}_{\lambda}^{[\bm{\nu}']})$ if and only if 
$\bm{\nu}=\bm{\nu}'$, and hence that the elements 
$S_{z_{\xi}t_{\xi}} (\ti{u}_{\lambda}^{[\bm{\nu}]}) \in \ti{\CB}(\lambda)$, 
$\bm{\nu} \in \BZ_{\ge 0}^{\lambda}$, are linearly independent over $\BQ$. 
Also, from \cite[Lemma~1.6\,(1)]{AK}, we deduce that 
for every $\bm{\nu}=(\nu^{(i)}_{1},\,\dots,\,\nu^{(i)}_{m_{i}})_{i \in I} 
\in \BZ_{\ge 0}^{\lambda}$, 
\begin{align}
S_{z_{\xi}t_{\xi}}(\ti{u}_{\lambda}^{[\bm{\nu}]}) 
& = 
S_{z_{\xi}t_{\xi}} \Bigl(
\ts{
\bigotimes_{i \in I} 
(u_{\vpi_{i}}^{[-\nu^{(i)}_{1}]} \otimes \cdots \otimes 
 u_{\vpi_{i}}^{[-\nu^{(i)}_{m_i}]}) } \Bigr) \nonumber \\
& = 
\ts{
\bigotimes_{i \in I} 
\bigl(
 S_{z_{\xi}t_{\xi}}(u_{\vpi_{i}}^{[-\nu^{(i)}_{1}]}) \otimes \cdots \otimes 
 S_{z_{\xi}t_{\xi}}(u_{\vpi_{i}}^{[-\nu^{(i)}_{m_i}]})
\bigr) } \in \ti{\CB}(\lambda). \label{eq:Su}
\end{align}

Since the global basis element 
$G(S_{z_{\xi}t_{\xi}}(u^{\bc})) \in \CL(\lambda) \subset V(\lambda)$ 
is contained in $V_{z_{\zeta}t_{\zeta}}^{-}(\lambda) = 
U_{q}^{-} S^{\norm}_{z_{\zeta}t_{\zeta}}v_{\lambda}$ by the assumption, 
and since $\Phi_{\lambda}(\CL(\lambda)) \subset \ti{\CL}(\lambda)$ 
as mentioned above, it follows from Remark~\ref{rem:vlambda} that 
\begin{equation*}
\Phi_{\lambda}\bigl(G(S_{z_{\xi}t_{\xi}}(u^{\bc}))\bigr) \in 
\left(\bigotimes_{i \in I} 
V_{z_{\zeta}t_{\zeta}}^{-}(\vpi_{i})^{\otimes m_{i}}
\right) 
\cap \ti{\CL}(\lambda). 
\end{equation*}
Because $\Phi_{\lambda}|_{q=0}:
\CL(\lambda)/q_{s}\CL(\lambda) \hookrightarrow 
\ti{\CL}(\lambda)/q_{s}\ti{\CL}(\lambda)$ is 
induced by $\Phi_{\lambda}:V(\lambda) \hookrightarrow \ti{V}(\lambda)$, 
and because $V_{z_{\zeta}t_{\zeta}}^{-}(\vpi_{i})$ has the global basis 
$\bigl\{ G(b) \mid b \in \CB_{z_{\zeta}t_{\zeta}}^{-}(\vpi_{i})\bigr\}$ 
for each $i \in I$, we see that 
%
%
\begin{equation} \label{eq:zt1-1b}
\Phi_{\lambda}|_{q=0}(S_{z_{\xi}t_{\xi}}(u^{\bc})) \in 
\underbrace{
\Span_{\BQ} 
 \left(\bigotimes_{i \in I} \CB_{z_{\zeta}t_{\zeta}}^{-}(\vpi_{i})^{\otimes m_{i}}\right)}_{=:U}
\subset \ti{\CL}(\lambda)/q_{s}\ti{\CL}(\lambda)= 
\Span_{\BQ} \ti{\CB}(\lambda). 
\end{equation}
Here we recall that $\ti{\CB}(\lambda)=\bigotimes_{i \in I} \CB(\vpi_{i})^{\otimes m_{i}}$ 
is a $\BQ$-basis of the $\BQ$-vector space $\ti{\CL}(\lambda)/q_{s}\ti{\CL}(\lambda)$, 
and that $\bigotimes_{i \in I} \CB_{z_{\zeta}t_{\zeta}}^{-}(\vpi_{i})^{\otimes m_{i}}$ 
is a subset of $\ti{\CB}(\lambda)$ generating the vector space $U$ over $\BQ$ 
in \eqref{eq:zt1-1b}.
Therefore, we deduce from \eqref{eq:zt1-1a} and \eqref{eq:zt1-1b} that 
%
%
\begin{equation} \label{eq:emb3}
S_{z_{\xi}t_{\xi}} (\ti{u}_{\lambda}^{[\bm{\nu}_{\bc}]}) \in 
 \bigotimes_{i \in I} \CB_{z_{\zeta}t_{\zeta}}^{-}(\vpi_{i})^{\otimes m_{i}}.
\end{equation}
If $\bc = (\rho^{(i)})_{i \in I}$, with 
$\rho^{(i)}=(\rho^{(i)}_1 \ge \cdots \ge \rho^{(i)}_{m_{i}-1} \ge 0)$ 
for $i \in I$, then it follows immediately from \eqref{eq:Su} that 
\begin{equation*}
S_{z_{\xi}t_{\xi}}(\ti{u}_{\lambda}^{[\bm{\nu}_{\bc}]}) = 
\ts{
\bigotimes_{i \in I} 
\bigl(
 S_{z_{\xi}t_{\xi}}(u_{\vpi_{i}}^{[-\rho^{(i)}_{1}]}) \otimes \cdots \otimes 
 S_{z_{\xi}t_{\xi}}(u_{\vpi_{i}}^{[-\rho^{(i)}_{m_i-1}]}) \otimes 
 \underbrace{S_{z_{\xi}t_{\xi}}(u_{\vpi_{i}})}_{\text{(*)}}
\bigr) }. 
\end{equation*}
By combining this equality with \eqref{eq:emb3}, we find that 
for every $i \in I \setminus J$, 
the tensor factor $S_{z_{\xi}t_{\xi}}(u_{\vpi_{i}})$ in the position (*) is 
contained in $\CB_{z_{\zeta}t_{\zeta}}^{-}(\vpi_{i})$. 
Let $i \in I \setminus J$. 
Since the weights of elements in $\CB_{z_{\zeta}t_{\zeta}}^{-}(\vpi_{i})$ 
are contained in $z_{\zeta}t_{\zeta}\vpi_{i}-Q_{\af}^{+}$, 
we conclude that $z_{\xi}t_{\xi}\vpi_{i} = \wt(S_{z_{\xi}t_{\xi}}(u_{\vpi_{i}})) \in 
z_{\zeta}t_{\zeta}\vpi_{i} - Q^{+}_{\af}$; 
since $z_{\xi},\,z_{\zeta} \in W_{J}$, and $i \in I \setminus J$, 
we have $z_{\xi}\vpi_{i}=z_{\zeta}\vpi_{i}=\vpi_{i}$. 
From these, we obtain
\begin{equation*}
\vpi_{i}-\pair{\xi}{\vpi_{i}}\delta = z_{\xi}t_{\xi}\vpi_{i} \in 
z_{\zeta}t_{\zeta}\vpi_{i}-Q_{\af}^{+} = 
\vpi_{i}-\pair{\zeta}{\vpi_{i}}\delta-Q^{+}_{\af},
\end{equation*}
and hence $\pair{\xi-\zeta}{\vpi_{i}}\delta \in Q^{+}_{\af}$. 
Hence it follows that 
$\pair{\xi-\zeta}{\vpi_{i}} \ge 0$ for every $i \in I \setminus J$, 
which implies that $[\xi-\zeta] \in \QJp{I \setminus J}$. 
This proves the claim. \bqed
%
%
\begin{claim} \label{c:zt2}
Let $\bc \in \Par(\lambda)$, and $y \in (W^{J})_{\af}$, $\zeta \in \jad$.  
If $S_{y}(u^{\bc}) \in \CB_{z_{\zeta}t_{\zeta}}^{-}(\lambda)$, 
then $y \sige z_{\zeta}t_{\zeta}$.
\end{claim}

\noindent
{\it Proof of Claim~\ref{c:zt2}.} 
Write $y \in (W^{J})_{\af}$ in the form $y = w z_{\xi}t_{\xi}$, 
with $w \in W^{J}$ and $\xi \in \jad$ (see \eqref{eq:W^J_af}). 
We prove the claim by induction on $\ell(w)$. 
If $\ell(w)=0$, then $w=e$, and hence the claim follows immediately from Claim~\ref{c:zt1}. 
Assume that $\ell(w) > 0$, and take $j \in I$ such that $\ell(r_{j}w)=\ell(w)-1$; 
in this case, we have $-w^{-1}\alpha_{j} \in \Delta^{+} \setminus \Delta_{J}^{+}$ 
(see \cite[Proposition~5.10]{LNSSS} for example), which implies that 
$n:=\pair{\alpha_{j}^{\vee}}{y\lambda}=
\pair{\alpha_{j}^{\vee}}{w\lambda} < 0$. 
Therefore, we obtain $r_{j}y \in (W^{J})_{\af}$ and 
$y \sig r_{j}y$ by Lemma~\ref{lem:si2}. 
Also, since $\wt (u^{\bc}) = \lambda - |\bc|\delta$, we have
\begin{equation*}
S_{r_{j}y}(u^{\bc})= S_{r_{j}}S_{y}(u^{\bc})=
e_{j}^{-n}S_{y}(u^{\bc}).
\end{equation*}
Here, since $j \in I$, we have 
$\pair{\alpha_{j}^{\vee}}{z_{\zeta}t_{\zeta}\lambda}=
 \pair{\alpha_{j}^{\vee}}{\lambda} \ge 0$. 
Hence it follows from Proposition~\ref{prop:key} that 
the set $\CB_{z_{\zeta}t_{\zeta}}^{-}(\lambda) \cup \bigl\{\bzero\bigr\}$ 
is stable under the action of the Kashiwara operator $e_{j}$. 
Because $S_{y}(u^{\bc}) \in \CB_{z_{\zeta}t_{\zeta}}^{-}(\lambda)$ 
by the assumption, we deduce that 
$S_{r_{j}y}(u^{\bc}) = e_{j}^{-n}S_{y}(u^{\bc}) \in 
 \CB_{z_{\zeta}t_{\zeta}}^{-}(\lambda)$. 
Therefore, by our induction hypothesis, 
we obtain $r_{j}y \sige z_{\zeta}t_{\zeta}$. 
Since $y \sig r_{j}y$ as seen above, we conclude that 
$y \sig r_{j}y \sige z_{\zeta}t_{\zeta}$, as desired. \bqed

\vsp

Now, let $x,\,y \in (W^{J})_{\af}$, and 
assume that $S_{y}(u^{\bc}) \in \CB_{x}^{-}(\lambda)$. 
By \cite[Lemma~1.4]{AK}, there exist $j_{1},\,j_{2},\,\dots,\,j_{p} \in I_{\af}$ 
such that 
\begin{enu}

\item 
$\pair{\alpha_{j_{m}}^{\vee}}{r_{j_{m-1}} \cdots r_{j_{2}}r_{j_{1}}x\lambda} > 0$ 
for all $1 \le m \le p$; 
\item 
$r_{j_{p}}r_{j_{p-1}} \cdots r_{j_{2}}r_{j_{1}}x\lambda \in \lambda+\BZ\delta$. 

\end{enu}
By Lemma~\ref{lem:si2}, together with condition (1), we see that 
$r_{j_{m}} \cdots r_{j_{2}}r_{j_{1}}x \in (W^{J})_{\af}$ for all $0 \le m \le p$.
From this, we deduce by condition (2) that 
$r_{j_{p}}r_{j_{p-1}} \cdots r_{j_{2}}r_{j_{1}}x = z_{\zeta}t_{\zeta}$ 
for some $\zeta \in \jad$. We show by induction on the length $p$ 
of the sequence above that $y \sige x$. 
If $p=0$, then $x=z_{\zeta}t_{\zeta}$, and 
hence the assertion follows immediately 
from Claim~\ref{c:zt1}. Assume that $p > 0$. 

\paragraph{Case 1.}
Assume that $\pair{\alpha_{j_{1}}^{\vee}}{y\lambda} > 0$; 
note that $r_{j_1}y \in (W^{J})_{\af}$ by Lemma~\ref{lem:si2}. 
Since $u^{\bc} \in \CB(\lambda)$ is 
an extremal element of weight $\lambda-|\bc|\delta$, we have 
\begin{equation*}
S_{r_{j_1}y}(u^{\bc})= S_{r_{j_1}}S_{y}(u^{\bc}) = 
f_{j_1}^{\max}S_{y}(u^{\bc}). 
\end{equation*}
Since $S_{y}(u^{\bc}) \in \CB_{x}^{-}(\lambda)$ by the assumption, 
it follows from Corollary~\ref{cor:key} that 
$S_{r_{j_1}y}(u^{\bc}) = f_{j_1}^{\max}S_{y}(u^{\bc}) \in 
 \CB_{r_{j_{1}}x}^{-}(\lambda)$. Therefore, by our induction hypothesis 
 (applied to $r_{j_1}x \in (W^{J})_{\af}$), 
we obtain $r_{j_1}y \sige r_{j_{1}}x$. 
Because $\pair{\alpha_{j_1}^{\vee}}{r_{j_1}x\lambda} < 0$ by condition (1), and 
because $\pair{\alpha_{j_1}^{\vee}}{r_{j_1}y\lambda} < 0$ by our assumption above, 
we conclude from Lemma~\ref{lem:si-L41}\,(3) that $y \sige x$. 

\paragraph{Case 2.}
Assume that $\pair{\alpha_{j_{1}}^{\vee}}{y\lambda} \le 0$. 
Since $u^{\bc}$ is an extremal element of weight $\lambda-|\bc|\delta$, 
it follows that $f_{j_{1}}S_{y}(u^{\bc}) = \bzero$, 
and hence $f_{j_{1}}^{\max}S_{y}(u^{\bc}) = 
f_{j_{1}}^{0}S_{y}(u^{\bc}) = S_{y}(u^{\bc})$. 
Since $S_{y}(u^{\bc}) \in \CB^{-}_{x}(\lambda)$ by the assumption, 
we deduce from Corollary~\ref{cor:key} that $S_{y}(u^{\bc}) = 
f_{j_{1}}^{\max}S_{y}(u^{\bc}) \in \CB_{r_{j_{1}}x}^{-}(\lambda)$. 
Therefore, by our induction hypothesis 
(applied to $r_{j_{1}}x \in (W^{J})_{\af}$), we obtain $y \sige r_{j_{1}}x$. 
Since $\pair{\alpha_{j_1}^{\vee}}{x\lambda} > 0$ by condition (1), 
we have $r_{j_1}x \sig x$ by Lemma~\ref{lem:si2}. 
Hence we conclude that $y \sige r_{j_{1}}x \sig x$, as desired. 

This completes the proof of the proposition. 
\end{proof}
%
%
\begin{cor} \label{cor:zt}
Let $x,\,y \in (W^{J})_{\af}$. 
Then, $\CB_{y}^{-}(\lambda) \subset \CB_{x}^{-}(\lambda)$ 
if and only if $y \sige x$. 
\end{cor}

\begin{proof}
The ``if'' part is already proved in Lemma~\ref{lem:zt0}. 
Let us prove the ``only if'' part. 
Since $S_{y}(u_{\lambda}) \in \CB_{y}^{-}(\lambda)$ 
by Lemma~\ref{lem:zt01}, we have 
$S_{y}(u_{\lambda}) \in \CB_{x}^{-}(\lambda)$. 
Therefore, by applying Proposition~\ref{prop:zt} to 
$\bc=(\rho^{(i)})_{i \in I}$ 
with $\rho^{(i)}=\emptyset$ for all $i \in I$, 
we obtain $y \sige x$ (note that in this case, 
$b(\bc)=u_{\infty}$ and $u^{\bc}=u_{\lambda}$). 
This proves the corollary. 
\end{proof}
%
%
\subsection{Fundamental properties of $\BB^{\si}_{\sige x}(\lambda)$.}
\label{subsec:properties2}
%
%
\begin{lem}[cf. Remark~\ref{rem:demazure}] \label{lem:stable-p}
Let $x \in (W^{J})_{\af}$. 
The set $\BB^{\si}_{\sige x}(\lambda) \cup \bigl\{\bzero\bigr\}$ 
(resp., $\BB^{\si}_{x \sige}(\lambda) \cup \bigl\{\bzero\bigr\}$) 
is stable under the action of the root operator 
$f_{j}$ (resp., $e_{j}$) for all $j \in I_{\af}$.
\end{lem}

\begin{proof}
We give a proof only for $\BB^{\si}_{\sige x}(\lambda)$; 
the proof for $\BB^{\si}_{x \sige}(\lambda)$ is similar. 
Let $\eta \in \BB^{\si}_{\sige x}(\lambda)$, i.e., $\kappa(\eta) \sige x$, 
and let $j \in I_{\af}$ be such that $f_{j}\eta \ne \bzero$. 
If $\kappa(f_{j}\eta) = \kappa(\eta)$, then there is nothing to prove. 
Now, assume that $\kappa(f_{j}\eta) = r_{j}\kappa(\eta)$. 
Then we deduce from the comment following \eqref{eq:t-f} that 
$\pair{\alpha_{j}^{\vee}}{\kappa(\eta)\lambda} > 0$, 
and hence $r_{j}\kappa(\eta) \sig \kappa(\eta)$ by Lemma~\ref{lem:si2}.
Since $\kappa(\eta) \sige x$ by our assumption, it follows that 
$\kappa(f_{j}\eta) = r_{j}\kappa(\eta) \sig \kappa(\eta) \sige x$, 
which implies that $f_{j}\eta \in \BB^{\si}_{\sige x}(\lambda)$. 
This proves the lemma. 
\end{proof}
%
%
\begin{prop}[cf. Proposition~\ref{prop:key}] \label{prop:key-p}
Let $x \in (W^{J})_{\af}$, and $j \in I_{\af}$. 
\begin{enu}
\item If $\pair{\alpha_{j}^{\vee}}{x\lambda} > 0$ 
(note that $r_{j}x \in (W^{J})_{\af}$ by Lemma~\ref{lem:si2}), then 
%
%
\begin{equation} \label{keyp2}
\BB^{\si}_{\sige x}(\lambda)=
 \bigl\{e_{j}^{k}\eta \mid 
   \eta \in \BB^{\si}_{\sige r_{j}x}(\lambda),\,k \in \BZ_{\ge 0}\bigr\}
 \setminus \{\bzero\} \quad (\supset \BB^{\si}_{\sige r_{j}x}(\lambda)).
\end{equation}

\item The set $\BB^{\si}_{\sige x}(\lambda) \cup \bigl\{\bzero\bigr\}$ is 
stable under the action of the root operator $e_{j}$ for $j \in I_{\af}$ 
such that $\pair{\alpha_{j}^{\vee}}{x\lambda} \ge 0$. 
\end{enu}
\end{prop}

\begin{proof}
(1) First we prove the inclusion $\subset$. 
Let $\eta \in \BB^{\si}_{\sige x}(\lambda)$; 
note that $\kappa(\eta) \sige x$ by the definition. 
Assume that $\pair{\alpha_{j}^{\vee}}{\kappa(\eta)\lambda} \le 0$. 
Since $\pair{\alpha_{j}^{\vee}}{x\lambda} > 0$ by the assumption, 
we see by Lemma~\ref{lem:si-L41}\,(1) that 
$\kappa(\eta) \sige r_{j}x$, and hence 
$\eta \in \BB^{\si}_{\sige r_{j}x}(\lambda)$. 
Thus, $\eta=e_{j}^{0}\eta$ is contained in the set on 
the right-hand side of \eqref{keyp2}. 
Assume now that $\pair{\alpha_{j}^{\vee}}{\kappa(\eta)\lambda} > 0$. 
It follows from Lemma~\ref{lem:vevp} that 
$\kappa(f_{j}^{\max}\eta)=r_{j}\kappa(\eta)$. 
Also, because $\pair{\alpha_{j}^{\vee}}{\kappa(\eta)\lambda} > 0$ and 
$\pair{\alpha_{j}^{\vee}}{x\lambda} > 0$ by the assumption, we deduce from 
Lemma~\ref{lem:si-L41}\,(3), together with 
our assumption $\kappa(\eta) \sige x$, that 
$\kappa(f_{j}^{\max}\eta) = r_{j}\kappa(\eta) \sige r_{j}x$, 
which implies that 
$f_{j}^{\max}\eta \in \BB^{\si}_{\sige r_{j}x}(\lambda)$. 
From this, we conclude that $\eta$ is contained 
in the set on the right-hand side of \eqref{keyp2}. 
This proves the inclusion $\subset$. 

Next we prove the opposite inclusion $\supset$. 
Let $\eta \in \BB^{\si}_{\sige r_{j}x}(\lambda)$, 
and assume that $e_{j}^{k}\eta \ne \bzero$ for some $k \in \BZ_{\ge 0}$; 
note that $\kappa(\eta) \sige r_{j}x$, and that 
$\kappa(e_{j}^{k}\eta)$ is equal either to $\kappa(\eta)$ or 
to $r_{j}\kappa(\eta)$. 
If $\kappa(e_{j}^{k}\eta)=\kappa(\eta)$, then 
we have $\kappa(e_{j}^{k}\eta) = \kappa(\eta) \sige r_{j}x$. 
Since $\pair{\alpha_{j}^{\vee}}{x\lambda} > 0$ by the assumption, 
it follows from Lemma~\ref{lem:si2} that $r_{j}x \sig x$. 
Combining these, we obtain $\kappa(e_{j}^{k}\eta) \sige x$, 
which implies that $e_{j}^{k}\eta \in \BB^{\si}_{\sige x}(\lambda)$. 
Assume that 
$\kappa(e_{j}^{k}\eta)=r_{j}\kappa(\eta)$. 
Then we see from the definition of the root operator $e_{j}$ 
(see the comment following \eqref{eq:t-e}) that 
$\pair{\alpha_{j}^{\vee}}{\kappa(\eta)\lambda} < 0$. 
Recall that $\pair{\alpha_{j}^{\vee}}{r_{j}x\lambda} < 0$ by the assumption. 
Since $\kappa(\eta) \sige r_{j}x$ by our assumption, 
we deduce from Lemma~\ref{lem:si-L41}\,(3) that 
$\kappa(e_{j}^{k}\eta) = r_{j}\kappa(\eta) \sige x$, 
which implies that $e_{j}^{k}\eta \in \BB^{\si}_{\sige x}(\lambda)$. 
This proves part (1). 

(2) The assertion for $j \in I_{\af}$ such that 
$\pair{\alpha_{j}^{\vee}}{x\lambda} > 0$ follows immediately from \eqref{keyp2}. 
Let $j \in I_{\af}$ be such that $\pair{\alpha_{j}^{\vee}}{x\lambda} = 0$, 
and let $\eta \in \BB^{\si}_{\sige x}(\lambda)$ be such that $e_{j}\eta \ne \bzero$; 
note that $\kappa(\eta) \sige x$ by our assumption, and that 
$\kappa(e_{j}\eta)$ is equal either to $\kappa(\eta)$ or to $r_{j}\kappa(\eta)$. 
If $\kappa(e_{j}\eta) = \kappa(\eta)$, then it is obvious that 
$e_{j}\eta \in \BB^{\si}_{\sige x}(\lambda)$. 
If $\kappa(e_{j}\eta) = r_{j}\kappa(\eta)$, then 
$\pair{\alpha_{j}^{\vee}}{\kappa(\eta)\lambda} < 0$ 
by the same argument as in the proof of part (1). 
Since $\pair{\alpha_{j}^{\vee}}{x\lambda} = 0$, and 
$\kappa(\eta) \sige x$ by our assumption, 
it follows from Lemma~\ref{lem:si-L41}\,(2) 
that $\kappa(e_{j}\eta) = r_{j}\kappa(\eta) \sige x$, 
which implies that $e_{j}\eta \in \BB^{\si}_{\sige x}(\lambda)$. 
This completes the proof of the proposition. 
\end{proof}
%
%
\begin{cor}[cf. Corollary~\ref{cor:key}] \label{cor:key-p3}
Let $x \in (W^{J})_{\af}$ and $j \in I_{\af}$ be such that 
$\pair{\alpha_{j}^{\vee}}{x\lambda} \ne 0$ 
 (note that 
$r_{j}x \in (W^{J})_{\af}$ by Lemma~\ref{lem:si2}). 
For every $\eta \in \BB^{\si}_{\sige x}(\lambda)$, 
we have $f_{j}^{\max}\eta \in \BB^{\si}_{\sige r_{j}x}(\lambda)$. 
\end{cor}

\begin{proof}
If $\pair{\alpha_{j}^{\vee}}{x\lambda} > 0$, then the assertion 
follows immediately from \eqref{keyp2}
and Lemma~\ref{lem:stable-p}. 
Assume now that $\pair{\alpha_{j}^{\vee}}{x\lambda} < 0$. 
Let $\eta \in \BB^{\si}_{\sige x}(\lambda)$. 
We see from Lemma~\ref{lem:stable-p} that 
$f_{j}^{\max}\eta \in \BB^{\si}_{\sige x}(\lambda)$. 
Also, it follows from Proposition~\ref{prop:key-p}\,(1) that 
$\BB^{\si}_{\sige x}(\lambda) \subset \BB^{\si}_{\sige r_{j}x}(\lambda)$.
Combining these, we obtain $f_{j}^{\max}\eta \in 
\BB^{\si}_{\sige x}(\lambda) \subset \BB^{\si}_{\sige r_{j}x}(\lambda)$, 
as desired. 
\end{proof}
%
%
\begin{prop}[cf. Corollary~\ref{cor:zt}] 
\label{prop:bruhat-p}
Let $x,\,y \in (W^{J})_{\af}$. Then, 
$y \sige x$ if and only if 
$\BB^{\si}_{\sige y}(\lambda) \subset \BB^{\si}_{\sige x}(\lambda)$.
\end{prop}

\begin{proof}
The ``only if'' part is obvious from the definitions. 
Let us prove the ``if'' part. 
It is obvious from the definition that 
$(y\,;\,0,\,1) \in \sLS$ is contained in $\BB^{\si}_{\sige y}(\lambda)$. 
Since $\BB^{\si}_{\sige y}(\lambda) \subset \BB^{\si}_{\sige x}(\lambda)$ by the assumption, 
we have $(y \,;\, 0,\,1) \in \BB^{\si}_{\sige x}(\lambda)$, and hence $y \sige x$. 
This proves the proposition. 
\end{proof}
%
%
\subsection{Proof of Theorem~\ref{thm:main}.}
\label{subsec:proof}

We need the following technical lemma. 
%
%
\begin{lem} \label{lem:seq}
For each $\eta \in \sLS$ and $x \in W_{\af}$, 
there exist $j_{1},\,j_{2},\,\dots,\,j_{p} \in I_{\af}$ 
satisfying the following conditions:
\begin{enu}
\item[\rm (i)] $\pair{\alpha_{j_{m}}^{\vee}}
{r_{j_{m-1}} \cdots r_{j_{2}}r_{j_{1}}x\lambda} \ge 0$ 
for all $1 \le m \le p$; 
\item[\rm (ii)] $f_{j_{p}}^{\max}f_{j_{p-1}}^{\max} \cdots 
      f_{j_{2}}^{\max}f_{j_{1}}^{\max}\eta=S_{z_{\xi}t_{\xi}}\eta^{C}$
      for some $\xi \in \jad$ and $C \in \Conn(\sLS)$.
\end{enu}
\end{lem}

\begin{proof}
First, we see from \cite[Lemma~1.4]{AK} that 
there exist $j_{1},\,j_{2},\,\dots,\,j_{a} \in I_{\af}$ such that
\begin{enu}
\item[(1a)] 
$\pair{\alpha_{j_{m}}^{\vee}}{r_{j_{m-1}} \cdots r_{j_{2}}r_{j_{1}}x\lambda} \ge 0$ 
for all $1 \le m \le a$; 
\item[(2a)] 
$r_{j_{a}}r_{j_{a-1}} \cdots r_{j_{2}}r_{j_{1}}x\lambda \in \lambda+\BZ\delta$. 
\end{enu}
We deduce from condition (2a) that 
$r_{j_{a}}r_{j_{a-1}} \cdots r_{j_{2}}r_{j_{1}}x = z t_{\zeta}$ 
for some $z \in W_{J}$ and $\zeta \in Q^{\vee}$. 
Let $w_{0}=r_{j_{b}}r_{j_{b-1}} \cdots r_{j_{a+2}}r_{j_{a+1}}$ be 
a reduced expression of the longest element $w_{0} \in W$. Then, 
there hold the following: 
\begin{enu}
\item[(1b)] for all $a+1 \le m \le b$, 
\begin{equation*}
\pair{\alpha_{j_{m}}^{\vee}}{r_{j_{m-1}} \cdots r_{j_{a+1}}
 \underbrace{r_{j_{a}} \cdots r_{j_{2}}r_{j_{1}}x}_{=zt_{\zeta}}\lambda} = 
\pair{\alpha_{j_{m}}^{\vee}}{r_{j_{m-1}} \cdots r_{j_{a+1}}\lambda} \ge 0;
\end{equation*}
\item[(2b)] 
$\underbrace{r_{j_{b}}r_{j_{b-1}} \cdots r_{j_{a+1}}}_{=w_{0}}
 \underbrace{r_{j_{a}} \cdots r_{j_{2}}r_{j_{1}}x}_{=zt_{\zeta}}
=w_{0}zt_{\zeta}$; 
\item[(3b)]
the element 
\begin{equation*}
\eta':=
\underbrace{f_{j_{b}}^{\max}f_{j_{b-1}}^{\max} \cdots f_{j_{a+1}}^{\max}}_{\text{corresponds to $w_{0}$}}
\underbrace{f_{j_{a}}^{\max} \cdots f_{j_{2}}^{\max}f_{j_{1}}^{\max}\eta}_{\in \sLS}
\end{equation*}
is a lowest weight element with respect to $I$, i.e., 
$f_{j}\eta'=\bzero$ for all $j \in I$ 
(this follows from \cite[Corollarie~9.1.4\,(2)]{K-Fr} 
 since $\sLS \cong \CB(\lambda)$ is a regular crystal; 
see Remark~\ref{rem:Gext}).
\end{enu}
It follows from Lemma~\ref{lem:vevp}, together with condition (3b), that 
$\pair{\alpha_{j}^{\vee}}{\kappa(\eta')\lambda} \le 0$ for all $j \in I$, 
which implies that $\kappa(\eta')\lambda \equiv w_{0}\lambda \mod \BR\delta$ 
since $W\lambda \cap (-P^{+}) = \bigl\{w_{0}\lambda\bigr\}$. 
We deduce from (the dual version of) \cite[Lemma~4.3.2]{NSdeg} that 
for the element $\eta'$ in (3b), 
there exist $j_{b+1},\,j_{b+2},\,\dots,\,j_{p} \in I_{\af}$ 
satisfying the following conditions: 
\begin{enu}

\item[(1c)] for all $b+1 \le m \le p$, 
\begin{equation*}
   \pair{\alpha_{j_{m}}^{\vee}}
      {r_{j_{m-1}} \cdots r_{j_{b+2}}r_{j_{b+1}}w_{0}\lambda} = 
\pair{\alpha_{j_{m}}^{\vee}}
      {r_{j_{m-1}} \cdots r_{j_{b+2}}r_{j_{b+1}}\kappa(\eta')\lambda} > 0;
\end{equation*}
\item[(2c)]
     $\eta'':=f_{j_{p}}^{\max}f_{j_{n-1}}^{\max} \cdots 
      f_{j_{b+2}}^{\max}f_{j_{b+1}}^{\max}\eta'$ is an element of $\sLS$ 
     such that $\ol{\eta}''(t) \equiv t\lambda \mod \BR\delta$
     for all $t \in [0,1]$.
\end{enu}
We see by Remark~\ref{rem:extp}\,(2) that $\eta''=S_{ z_{\xi}t_{\xi} }\eta^{C}$ 
for some $\xi \in \jad$ and $C \in \Conn(\sLS)$. 

Concatenating the three sequences of elements in $I_{\af}$ above, 
we obtain the sequence 
\begin{equation*}
\underbrace{j_{1},\,j_{2},\,\dots,\,j_{a}}_{\text{satisfy (1a), (2a)}},\,
\underbrace{j_{a+1},\,j_{a+2},\,\dots,\,j_{b}}_{\text{satisfy (1b), (2b), (3b)}},\,
\underbrace{j_{b+1},\,j_{b+2},\,\dots,\,j_{p}}_{\text{satisfy (1c), (2c)}}.
\end{equation*}
We show that this sequence indeed satisfies conditions (i) and (ii). 
We see from the definition of $\eta'$ in condition (3b) and condition (2c) that 
the sequence above satisfies condition (ii). For $1 \le m \le b$, 
it follows immediately from (1a) and (1b) that 
the sequence above satisfies condition (i).
Also, for $b+1 \le m \le p$, we have 
\begin{align*}
& \pair{\alpha_{j_{m}}^{\vee}}
  {r_{j_{m-1}} \cdots r_{j_{b+1}}r_{j_{b}} \cdots r_{j_{2}}r_{j_{1}}x\lambda} \\
& \qquad =
  \pair{\alpha_{j_{m}}^{\vee}}
  {r_{j_{m-1}} \cdots r_{j_{b+1}}w_{0}zt_{\zeta}\lambda}
  \quad \text{by condition (2b)} \\
& \qquad =
  \pair{\alpha_{j_{m}}^{\vee}}
  {r_{j_{m-1}} \cdots r_{j_{b+1}}w_{0}\lambda} > 0 
  \quad \text{by (1c)}.
\end{align*}
This completes the proof of the lemma. 
\end{proof}
%
%
\begin{lem} \label{lem:subset}
Let $b \in \CB(\lambda)$, and set 
$\eta:=\Psi_{\lambda}(b) \in \BB^{\si}(\lambda)$. 
If $x \in (W^{J})_{\af}$ satisfies 
$b \in \CB_{x}^{-}(\lambda)$, then $\kappa(\eta) \sige x$, 
and hence $\eta \in \BB^{\si}_{\sige x}(\lambda)$. 
\end{lem}

\begin{proof}
We fix an arbitrary $x \in (W^{J})_{\af}$ such that 
$b \in \CB_{x}^{-}(\lambda)$. 
Take $j_{1},\,j_{2},\,\dots,\,j_{p} \in I_{\af}$ 
satisfying conditions (i) and (ii) in Lemma~\ref{lem:seq} 
for $\eta=\Psi_{\lambda}(b) \in \sLS$ and $x \in (W^{J})_{\af} \subset W_{\af}$. 
We prove the assertion by induction on the length $p$ of this sequence. 
If $p=0$, then $\eta=S_{z_{\xi}t_{\xi}}\eta^{C}$ 
for some $\xi \in \jad$ and $C \in \Conn(\sLS)$; 
note that in this case, $\kappa(\eta)=z_{\xi}t_{\xi}$ by Remark~\ref{rem:extp}\,(2).
By the assumption, we have $S_{z_{\xi}t_{\xi}}(u^{\Theta(C)}) = 
\Psi_{\lambda}^{-1}(\eta) = b \in \CB_{x}^{-}(\lambda)$. 
Therefore, we conclude from Proposition~\ref{prop:zt} that 
$\kappa(\eta)=z_{\xi}t_{\xi} \sige x$, and hence 
$\eta \in \BB^{\si}_{\sige x}(\lambda)$, as desired. 
Now, assume that $p > 0$. 
%
\paragraph{Case 1.}
%
Assume that $\pair{\alpha_{j_{1}}^{\vee}}{x\lambda} > 0$; 
note that $r_{j_1}x \in (W^{J})_{\af}$ by Lemma~\ref{lem:si2}, and that 
$f_{j_1}^{\max}b \in \CB_{r_{j_1}x}^{-}(\lambda)$ by Corollary~\ref{cor:key}.
Observe that the sequence $j_{2},\,\dots,\,j_{p} \in I_{\af}$ satisfies 
conditions (i) and (ii) in Lemma~\ref{lem:seq} for 
$f_{j_1}^{\max}\eta = \Psi_{\lambda}(f_{j_1}^{\max}b)$ and $r_{j_1}x$. 
Therefore, by our induction hypothesis, 
we obtain $\kappa(f_{j_1}^{\max}\eta) \sige r_{j_1}x$, 
and hence $f_{j_1}^{\max}\eta \in \BB^{\si}_{\sige r_{j_1}x}(\lambda)$. 
Since $\pair{\alpha_{j_{1}}^{\vee}}{x\lambda} > 0$ by our assumption, 
we conclude from \eqref{keyp2} that $\eta \in \BB^{\si}_{\sige x}(\lambda)$.

\paragraph{Case 2.}
%
Assume that $\pair{\alpha_{j_{1}}^{\vee}}{x\lambda} = 0$; 
remark that $f_{j_1}^{\max}b \in \CB_{x}^{-}(\lambda)$ by Remark~\ref{rem:demazure}. 
Observe that the sequence $j_{2},\,\dots,\,j_{p} \in I_{\af}$ satisfies 
conditions (i) and (ii) in Lemma~\ref{lem:seq} for 
$f_{j_1}^{\max}\eta = \Psi_{\lambda}(f_{j_1}^{\max}b)$ and $x$; 
notice that $r_{j_{1}}x\lambda=x\lambda$ 
since $\pair{\alpha_{j_{1}}^{\vee}}{x\lambda} = 0$. 
Therefore, by our induction hypothesis, 
we obtain $\kappa(f_{j_1}^{\max}\eta) \sige x$,
and hence $f_{j_1}^{\max}\eta \in \BB^{\si}_{\sige x}(\lambda)$. 
Since $\pair{\alpha_{j_{1}}^{\vee}}{x\lambda} = 0$ by our assumption, 
it follows from Proposition~\ref{prop:key-p}\,(2) that 
$\eta \in \BB^{\si}_{\sige x}(\lambda)$. 

This proves the lemma.
\end{proof}
%
%
\begin{lem} \label{lem:supset}
Let $\eta \in \sLS$, and set 
$b:=\Psi_{\lambda}^{-1}(\eta) \in \B$. 
If $x \in (W^{J})_{\af}$ satisfies 
$\eta \in \BB^{\si}_{\sige x}(\lambda)$, then 
$b \in \CB_{x}^{-}(\lambda)$. 
\end{lem}

\begin{proof}
We fix an arbitrary $x \in (W^{J})_{\af}$ such that 
$\eta \in \BB^{\si}_{\sige x}(\lambda)$. 
Take $j_{1},\,j_{2},\,\dots,\,j_{p} \in I_{\af}$ 
satisfying conditions (i) and (ii) in Lemma~\ref{lem:seq} 
for $\eta$ and $x$. We prove the assertion by induction 
on the length $p$ of this sequence. 
If $p=0$, then $\eta=S_{z_{\xi}t_{\xi}}\eta^{C}$ 
for some $\xi \in \jad$ and $C \in \Conn(\sLS)$; 
note that in this case, $\kappa(\eta)=z_{\xi}t_{\xi}$ 
by Remark~\ref{rem:extp}\,(2).
Since $\eta \in \BB^{\si}_{\sige x}(\lambda)$ by the assumption, 
we have $z_{\xi}t_{\xi} \sige x$. 
Therefore, we conclude from Proposition~\ref{prop:zt} that 
$b = \Psi_{\lambda}^{-1}(\eta) = 
 S_{z_{\xi}t_{\xi}}(u^{\Theta(C)}) \in \CB_{x}^{-}(\lambda)$. 
Now, assume that $p > 0$. 
%
\paragraph{Case 1.}
%
If $\pair{\alpha_{j_{1}}^{\vee}}{x\lambda} > 0$, then 
$r_{j_1}x \in (W^{J})_{\af}$ by Lemma~\ref{lem:si2}, and 
$f_{j_{1}}^{\max}\eta \in \BB^{\si}_{\sige r_{j_1}x}(\lambda)$ 
by Corollary~\ref{cor:key-p3}. 
Observe that the sequence $j_{2},\,\dots,\,j_{p} \in I_{\af}$ satisfies 
conditions (i) and (ii) in Lemma~\ref{lem:seq} for 
$f_{j_1}^{\max}\eta$ and $r_{j_1}x$. 
Therefore, by our induction hypothesis, 
we obtain $f_{j_{1}}^{\max}b \in \CB_{r_{j}x}^{-}(\lambda)$. 
Since $\pair{\alpha_{j_{1}}^{\vee}}{x\lambda} > 0$ by our assumption, 
it follows from \eqref{lem2b} that $b \in \CB^{-}_{x}(\lambda)$. 
%
\paragraph{Case 2.}
%
Assume that $\pair{\alpha_{j_{1}}^{\vee}}{x\lambda} = 0$. 
By Lemma~\ref{lem:stable-p}, we have 
$f_{j_{1}}^{\max}\eta \in \BB^{\si}_{\sige x}(\lambda)$. 
Observe that the sequence $j_{2},\,\dots,\,j_{p} \in I_{\af}$ 
satisfies conditions (i) and (ii) in Lemma~\ref{lem:seq} for 
$f_{j_1}^{\max}\eta$ and $x$; 
notice that $r_{j_{1}}x\lambda=x\lambda$ 
since $\pair{\alpha_{j_{1}}^{\vee}}{x\lambda} = 0$.
Therefore, by our induction hypothesis, 
we obtain $f_{j_{1}}^{\max}b \in \CB_{x}^{-}(\lambda)$. 
Since $\pair{\alpha_{j_{1}}^{\vee}}{x\lambda} = 0$ by our assumption, 
it follows from Proposition~\ref{prop:key} 
that $b \in \CB^{-}_{x}(\lambda)$. 

This proves the lemma.
\end{proof}

\begin{proof}[Proof of Theorem~\ref{thm:main}]
As mentioned at the end of \S\ref{subsec:demazurep}, 
our remaining task is to prove the equality 
%
%
\begin{equation} \label{eq:main1}
\Psi_{\lambda}(\CB_{x}^{-}(\lambda))=\BB^{\si}_{\sige x}(\lambda)
\qquad \text{for all $x \in (W^{J})_{\af}$}.
\end{equation}
The inclusion $\subset$ in \eqref{eq:main1} 
follows immediately from Lemma~\ref{lem:subset}. 
Also, the opposite inclusion $\supset$ in \eqref{eq:main1} 
follows immediately from Lemma~\ref{lem:supset}. 
This completes the proof of Theorem~\ref{thm:main}. 
\end{proof}
%
%
\section{Graded characters.} 
\label{sec:gch}

In this section, we fix 
$\lambda=\sum_{i \in I} m_{i}\vpi_{i} \in P^{+}$, 
and set $J=J_{\lambda}:=\bigl\{i \in I \mid \pair{\alpha_{i}^{\vee}}{\lambda}=0\bigr\}$. 

%
\subsection{Graded character formula for $V_{e}^{-}(\lambda)$.} 
\label{subsec:gch-e}

The weights of the Demazure submodule 
$V_{e}^{-}(\lambda) = U_{q}^{-}v_{\lambda}$ 
(corresponding to $x=e$, the identity element) 
are all contained in $\lambda-Q_{\af}^{+} \subset \lambda-Q_{\af}$, 
and hence every weight space of $V_{e}^{-}(\lambda)$ 
is finite-dimensional. Therefore, we can define 
the (ordinary) character $\ch V_{e}^{-}(\lambda)$ 
of $V_{e}^{-}(\lambda)$ to be
\begin{equation*}
\ch V_{e}^{-}(\lambda) := \sum_{\beta \in Q_{\af}}
\dim V_{e}^{-}(\lambda)_{\lambda-\beta}\,x^{\lambda-\beta}.
\end{equation*}
Here we recall that an element $\beta \in Q_{\af}$ 
can be written uniquely in the form: $\beta=\gamma + k\delta$ 
for $\gamma \in Q$ and $k \in \BZ$; if we set 
$x^{\delta}:=q$, then $x^{\lambda-\beta} = 
x^{\lambda-\gamma} q^{k}$. 
Now we define the graded character
$\gch V_{e}^{-}(\lambda)$ of $V_{e}^{-}(\lambda)$ to be 
\begin{equation*}
\gch V_{e}^{-}(\lambda) := \sum_{\gamma \in Q,\,k \in \BZ} 
\dim V_{e}^{-}(\lambda)_{\lambda-\gamma+k\delta}\,x^{\lambda-\gamma}q^{k},
\end{equation*}
which is obtained from the ordinary character $\ch V_{e}^{-}(\lambda)$ 
by replacing $x^{\delta}$ with $q$. 
%
%
\begin{thm} \label{thm:grch1}
Keep the notation and setting above. 
The graded character $\gch V_{e}^{-}(\lambda)$ of 
$V_{e}^{-}(\lambda)$ can be expressed as 
\begin{equation*}
\gch V_{e}^{-}(\lambda) = 
\biggl(\prod_{i \in I} \prod_{r=1}^{m_{i}}(1-q^{-r})\biggr)^{-1} 
P_{\lambda}(x\,;\,q^{-1},\,0),
\end{equation*}
where $P_{\lambda}(x\,;\,q,\,0)$ 
denotes the specialization at $t=0$ 
of the symmetric Macdonald polynomial $P_{\lambda}(x\,;\,q,\,t)$. 
\end{thm}
%
%
\subsection{Degree function.}
\label{subsec:degree}

Let $\cl : \BR \otimes_{\BZ} P_{\af} \twoheadrightarrow 
(\BR \otimes_{\BZ} P_{\af})/\BR\delta$ denote the canonical projection. 
For an element 
$\eta=(x_{1},\,\dots,\,x_{s}\,;\,a_{0},\,a_{1},\,\dots,\,a_{s}) \in \sLS$, 
we define $\cl(\eta)$ to be the piecewise-linear, continuous map 
$\cl(\eta) : [0,1] \rightarrow \BR \otimes_{\BZ} P_{\af}/\BR\delta$
whose ``direction vector'' for the interval 
$[a_{u-1},\,a_{u}]$ is equal to $\cl(x_{u}\lambda)$ 
for each $1 \le u \le s$; that is, 
\begin{equation} \label{eq:cl-LSpath}
\begin{split}
& (\cl(\eta))(t) := 
  \sum_{p = 1}^{u-1}(a_p - a_{p-1}) \cl(x_{p}\lambda) + (t - a_{u-1}) \cl(x_{u}\lambda) \\[-1.5mm]
& \hspace*{60mm} \text{for $t \in [a_{u-1},\,a_u]$, $1 \le u \le s$}.
\end{split}
\end{equation}
Because the map 
$\ol{\eta}:[0,1] \rightarrow \BR \otimes_{\BZ} P_{\af}$ defined by \eqref{eq:LSpath}
is an LS path of shape $\lambda$, 
the map $\cl(\eta)$ above is a ``projected (by $\cl$)'' 
LS path of shape $\lambda$, 
introduced in \cite[(3.4)]{NStensor} and on page 117 of \cite{NSpext2} 
(see also \cite[\S2.2]{LNSSS2}). By \cite[Theorem~3.3]{LNSSS2}, 
the crystal, denoted by $\BB(\lambda)_{\cl}$ in \cite{NStensor} and \cite{NSpext2}, 
of ``projected'' LS paths of shape $\lambda$ 
is identical to the crystal, denoted by $\QLS(\lambda)$, of 
all quantum Lakshmibai-Seshadri paths (QLS paths for short) of shape $\lambda$ 
under the identification $\cl(W_{\af}\lambda) = W \cl(\lambda) \cong W^{J}$; 
for the definition of QLS paths of shape $\lambda$, see \cite[\S3.2]{LNSSS2}. 

\begin{rem}
Define a surjective map $\cl : (W^{J})_{\af} \twoheadrightarrow W^{J}$ by 
(see \eqref{eq:W^J_af})
\begin{equation*}
\cl (x) = w \quad \text{if $x = wz_{\xi}t_{\xi}$, 
   with $w \in W^{J}$ and $\xi \in \jad$}. 
\end{equation*}
The quantum LS path of shape $\lambda$ corresponding to 
the map $\cl(\eta)$ defined by \eqref{eq:cl-LSpath} is given as:
\begin{equation*}
 (\cl(x_{1}),\,\dots,\,\cl(x_{s})\,;\,a_{0},\,a_{1},\,\dots,\,a_{s}),
\end{equation*}
where, for each $1 \le p < q \le s$ such that $\cl(x_{p})= \cdots = \cl(x_{q})$, 
we drop $\cl(x_{p}),\,\dots,\,\cl(x_{q-1})$ and $a_{p},\,\dots,\,a_{q-1}$. 
\end{rem}

Thus we obtain a map $\cl:\BB^{\si}(\lambda) \rightarrow \QLS(\lambda)$. 
Recall that root operators on $\QLS(\lambda)$ are defined 
in exactly the same manner as those on $\BB^{\si}(\lambda)$ 
(see, e.g., \cite[\S2.3]{LNSSS2}); 
from these definitions of root operators on
$\BB^{\si}(\lambda)$ and $\QLS(\lambda)$, we see that 
the map $\cl:\BB^{\si}(\lambda) \rightarrow \QLS(\lambda)$ commutes with 
root operators. Because the crystal $\QLS(\lambda)$ is connected by 
\cite[Proposition~3.23]{NStensor}, 
the map $\cl$ above is surjective. 
%
%
\begin{lem} \label{lem:cl-inv}
Let $\psi \in \QLS(\lambda)$, and assume that $\psi=X\cl(\eta_{e})$ 
for some monomial $X$ in root operators. Then, 
\begin{equation} \label{eq:cl-inv}
\cl^{-1}(\psi)=\bigl\{X S_{t_{\zeta}}\eta^{C} \mid 
 C \in \Conn(\BB^{\si}(\lambda)),\,\zeta \in Q^{\vee}_{I \setminus J} \bigr\}; 
\end{equation}
for the definition of $\eta^{C}$, see Proposition~\ref{prop:SLS}\,(1).
\end{lem}

\begin{proof}
First we prove the inclusion $\supset$. 
Let $C \in \Conn(\BB^{\si}(\lambda))$, 
and $\zeta \in Q^{\vee}_{I \setminus J}$. 
Since the map $\cl:\BB^{\si}(\lambda) \twoheadrightarrow \QLS(\lambda)$ 
commutes with root operators, and since 
$\cl(S_{t_{\zeta}}\eta^{C}) = \cl(\eta_{e})$ (see Remark~\ref{rem:extp}\,(2)), 
we have
\begin{equation*}
\cl(XS_{t_{\zeta}}\eta^{C}) = X \cl(S_{t_{\zeta}}\eta^{C}) = 
X \cl(\eta_{e}) = \psi, 
\end{equation*}
which implies that $XS_{t_{\zeta}}\eta^{C} \in \cl^{-1}(\psi)$. 

Next we prove the opposite inclusion $\subset$. 
Write $X$ as: $X=x_{j_{1}}x_{j_{2}} \cdots x_{j_{p}}$, where 
$x_{j_{q}}$ is either $e_{j_{q}}$ or $f_{j_{q}}$ for each $1 \le q \le p$. 
We prove the inclusion $\subset$ by induction on $p$. 
Assume that $p=0$, i.e., $\psi=\cl(\eta_{e})$. 
If $\eta \in \cl^{-1}(\psi)$, then we deduce from Remark~\ref{rem:extp}\,(2)
that $\eta$ is of the form:
\begin{equation*}
\eta = 
 (z_{\zeta_1}t_{\zeta_1},\,\dots,\,z_{\zeta_{s}}t_{\zeta_{s}} \,;\, 
  a_{0},\,a_{1},\,\dots,\,a_{s})
\end{equation*}
for some $s \ge 1$ and $\zeta_{1},\,\ldots,\,\zeta_{s} \in \jad$, 
and hence that $\eta=S_{ z_{\zeta_{s}} t_{\zeta_{s}} }\eta^{C}$, 
with $C$ the connected component containing $\eta$. 
Because 
\begin{equation*}
\pair{\alpha_{j}^{\vee}}{ \underbrace{\wt (S_{ t_{\zeta_{s}} }\eta^{C})}_{\in \lambda+\BZ\delta}} = 
\pair{\alpha_{j}^{\vee}}{\lambda} = 0
\end{equation*}
for $j \in J$, we see that
$S_{r_{j}}S_{ t_{\zeta_{s}} }\eta^{C} = S_{ t_{\zeta_{s}} }\eta^{C}$ 
for $j \in J$. Therefore, we have
$S_{ z_{\zeta_{s}} t_{\zeta_{s}} }\eta^{C} = 
S_{ z_{\zeta_{s}} } S_{ t_{\zeta_{s}} }\eta^{C} = 
S_{ t_{\zeta_{s}} }\eta^{C}$.
Also, since $S_{ t_{\beta} }\eta^{C}=\eta^{C}$ 
for $\beta \in Q^{\vee}_{J}$, it follows that 
\begin{equation*}
\eta = S_{ z_{\zeta_{s}} t_{\zeta_{s}} }\eta^{C} =
S_{ t_{\zeta_{s}} }\eta^{C} = 
S_{ t_{ [\zeta_{s}] } } S_{ t_{ \zeta_{s}-[\zeta_{s}] } }\eta^{C} = 
S_{ t_{ [\zeta_{s}] } }\eta^{C}
\quad \text{(note that $\zeta_{s}-[\zeta_{s}] \in Q^{\vee}_{J}$)},
\end{equation*}
where
$[\,\cdot\,]:Q^{\vee} = Q^{\vee}_{I \setminus J} \oplus Q_{J}^{\vee} 
\twoheadrightarrow Q^{\vee}_{I \setminus J}$ denotes the projection 
(see \eqref{eq:prj}). 
Thus, we conclude that 
$\eta$ is contained in the set
$\bigl\{S_{t_{\zeta}}\eta^{C} \mid 
 C \in \Conn(\BB^{\si}(\lambda)),\,
 \zeta \in Q^{\vee}_{I \setminus J} \bigr\}$.

Assume now that $p > 0$, and set $j:=j_{1}$, $X':=x_{j_{2}} \cdots x_{j_{p}}$, 
$\psi':=X'\cl(\eta_{e}) \in \QLS(\lambda)$. We define $y_{j} := e_{j}$ 
(resp., $y_{j}:=f_{j}$) if $x_{j}=f_{j}$ (resp., $x_{j}=e_{j}$); 
note that $y_{j}\psi=\psi'$. 
Let $\eta \in \cl^{-1}(\psi)$. 
Since the map $\cl:\BB^{\si}(\lambda) \twoheadrightarrow \QLS(\lambda)$ 
commutes with root operators, we have 
$\cl(y_{j}\eta)=y_{j}\cl(\eta)=y_{j}\psi = \psi'$, and hence 
$y_{j}\eta \in \cl^{-1}(\psi')$. By our induction hypothesis, 
there exists $C \in \Conn(\BB^{\si}(\lambda))$ and 
$\zeta \in Q^{\vee}_{I \setminus J}$
such that $y_{j}\eta = X'S_{t_{\zeta}}\eta^{C}$. 
Therefore, we deduce that 
$\eta = x_{j} X'S_{t_{\zeta}}\eta^{C} = X S_{t_{\zeta}}\eta^{C}$. 
Thus we have proved the inclusion $\subset$. 
This completes the proof of the lemma.
\end{proof}
%
%
\begin{lem} \label{lem:pi-eta}
For each $\psi \in \QLS(\lambda)$, 
there exists a unique $\eta_{\psi} \in \BB^{\si}_{0}(\lambda)$ 
such that $\cl(\eta_{\psi})=\psi$ and $\kappa(\eta_{\psi}) \in W^{J}$. 
\end{lem}

\begin{proof}
We write $\psi = X\cl(\eta_{e})$ for some monomial $X$ 
in root operators. 
Note that for $C \in \Conn(\BB^{\si}(\lambda))$, 
$C=\BB^{\si}_{0}(\lambda)$ holds if and only if $\eta^{C}=\eta_{e}$ 
(see Proposition~\ref{prop:SLS}\,(1)). 
Therefore, by Lemma~\ref{lem:cl-inv}, 
\begin{equation} \label{eq:cl-inv1}
\cl^{-1}(\psi) \cap \BB^{\si}_{0}(\lambda) = 
 \bigl\{ X S_{t_{\zeta}}\eta_{e} \mid \zeta \in Q^{\vee}_{I \setminus J} \bigr\}. 
\end{equation}
Let us write $\kappa(X\eta_{e}) \in (W^{J})_{\af}$ as 
$\kappa(X\eta_{e}) = wz_{\xi}t_{\xi}$ for 
$w \in W^{J}$ and $\xi \in \jad$, and 
write $\PJ(t_{-\xi}) = z_{-\xi}t_{-\xi+\phi_{J}(-\xi)} \in (W^{J})_{\af}$ 
(see Lemma~\ref{lem:J-adj}\,(2)) as 
$\PJ(t_{-\xi}) = t_{-\xi}y$ for some $y \in (W_{J})_{\af}$. 
Since $S_{t_{\beta}}\eta_{e}=\eta_{e}$ for all $\beta \in Q_{J}^{\vee}$, 
we have $S_{t_{[-\xi]}}\eta_{e} = S_{t_{[-\xi]}}S_{t_{-\xi-[-\xi]}}\eta_{e} = 
S_{t_{-\xi}}\eta_{e}$, where
$[\,\cdot\,]:Q^{\vee} = Q^{\vee}_{I \setminus J} \oplus Q_{J}^{\vee} 
\twoheadrightarrow Q^{\vee}_{I \setminus J}$ is the projection 
(see \eqref{eq:prj}). 
Hence it follows from \eqref{eq:cl-inv1} that 
$XS_{t_{-\xi}}\eta_{e} = XS_{t_{[-\xi]}}\eta_{e}$ is contained in 
$\cl^{-1}(\psi) \cap \BB^{\si}_{0}(\lambda)$.
We show that $\kappa(XS_{t_{-\xi}}\eta_{e}) = w \in W^{J}$. 
Note that
%
%
\begin{equation} \label{eq:Seta}
S_{t_{-\xi}}\eta_{e} = 
(\PJ(t_{-\xi}) \,;\, 0,\,1) = 
(z_{-\xi}t_{-\xi+\phi_{J}(-\xi)} \,;\, 0,\,1) = 
(t_{-\xi}y \,;\, 0,\,1)
\end{equation}
by Remark~\ref{rem:extp}\,(2). 
Here we remark that $\eta_{e} = (e \,;\, 0,\,1)$ and 
$S_{t_{-\xi}}\eta_{e}$ are both of the form as in 
\cite[(7.1.1)]{INS}. 
Because $X\eta_{e}$ is of the form: 
\begin{equation*}
X\eta_{e} = 
(\dots,\, \underbrace{wz_{\xi}t_{\xi}}_{=\kappa(X\eta_{e})} \,;\,
  0,\,\dots,\,1) =
(\dots,\, (wz_{\xi}t_{\xi}) e \,;\,
  0,\,\dots,\,1)
\end{equation*}
by the assumption that $\kappa(X\eta)=wz_{\xi}t_{\xi}$, 
we deduce from \cite[Lemma~7.1.4]{INS}, together with \eqref{eq:Seta}, 
that $XS_{t_{-\xi}}\eta_{e}$ is of the form:
\begin{align*}
XS_{t_{-\xi}}\eta_{e} & = 
(\dots,\, (wz_{\xi}t_{\xi})(z_{-\xi}t_{-\xi+\phi_{J}(-\xi)}) \,;\,
  0,\,\dots,\,1) \\
& = 
(\dots,\, \underbrace{(wz_{\xi}t_{\xi})(t_{-\xi}y)}_{
      =\kappa(XS_{t_{-\xi}}\eta_{e})} \,;\,
  0,\,\dots,\,1),
\end{align*}
and hence that
$\kappa(XS_{t_{-\xi}}\eta_{e}) \in (W^{J})_{\af}$ is 
equal to $(wz_{\xi}t_{\xi})(t_{-\xi}y)$; here, 
$z_{\xi}y$ must be equal to $e$, since 
$(wz_{\xi}t_{\xi})(t_{-\xi}y)=wz_{\xi}y \in (W^{J})_{\af}$, 
with $w \in W^{J} \subset (W^{J})_{\af}$ and 
$z_{\xi}y \in (W_{J})_{\af}$. Hence we conclude that 
$\kappa(XS_{t_{-\xi}}\eta_{e}) = w$. 
This proves the existence of $\eta_{\psi}$. 

It remains to prove the uniqueness of $\eta_{\psi}$. 
Assume that $\kappa(X S_{t_{\zeta_{1}}}\eta_{e}),\,
\kappa(X S_{t_{\zeta_{2}}}\eta_{e}) \in W^{J}$ for some 
$\zeta_{1},\,\zeta_{2} \in Q^{\vee}_{I \setminus J}$. 
Then we have 
$\kappa(X S_{t_{\zeta_{1}}}\eta_{e}) = 
 \kappa(X S_{t_{\zeta_{2}}}\eta_{e})$.
Observe that 
\begin{equation*}
S_{ t_{\zeta_{k}} }\eta_{e} = (\PJ(t_{\zeta_{k}})\,;\,0,\,1) = 
(z_{\zeta_{k}} t_{\zeta_{k}+\phi_{J}(\zeta_{k})}\,;\,0,\,1) \quad 
\text{for $k=1,\,2$}.
\end{equation*}
If we write $\kappa(X\eta_{e}) = wz_{\xi}t_{\xi}$ as above, 
then we see from \cite[Lemma~7.1.4]{INS} that
\begin{equation*}
\kappa(X S_{t_{\zeta_{1}}}\eta_{e}) = (wz_{\xi}t_{\xi})\PJ(t_{\zeta_{1}})
\quad \text{and} \quad
\kappa(X S_{t_{\zeta_{2}}}\eta_{e}) = (wz_{\xi}t_{\xi})\PJ(t_{\zeta_{2}}).
\end{equation*}
Therefore, we have $(wz_{\xi}t_{\xi})\PJ(t_{\zeta_{1}}) = 
(wz_{\xi}t_{\xi})\PJ(t_{\zeta_{2}})$, and hence 
$\PJ(t_{\zeta_{1}}) = \PJ(t_{\zeta_{2}})$. 
From this, it follows that 
\begin{equation*}
S_{t_{\zeta_{1}}}\eta_{e} = (\PJ(t_{\zeta_{1}})\,;\,0,\,1) = 
(\PJ(t_{\zeta_{2}})\,;\,0,\,1) = S_{t_{\zeta_{2}}}\eta_{e},
\end{equation*}
which implies that 
$XS_{t_{\zeta_{1}}}\eta_{e} = XS_{t_{\zeta_{2}}}\eta_{e}$. 
Thus we have proved the uniqueness of $\eta_{\psi}$. 
This completes the proof of the lemma.
\end{proof}

Now we define the (tail) degree function 
$\Degt:\QLS(\lambda) \rightarrow \BZ_{\le 0}$ 
as follows. 
For $\psi \in \QLS(\lambda)$, take 
$\eta_{\psi} \in \BB^{\si}_{0}(\lambda)$ 
as in Lemma~\ref{lem:pi-eta}. We write $\eta_{\psi}$ as:
\begin{equation*}
\eta_{\psi}
= (w_{1}z_{\xi_{1}}t_{\xi_{1}},\,\dots,\,
   w_{s-1}z_{\xi_{s-1}}t_{\xi_{s-1}},\,w;
a_{0},\,a_{1},\,\dots,\,a_{s-1},\,a_{s})
\end{equation*}
for $w_{1},\,\dots,\,w_{s-1},\,w \in W^{J}$ and 
$\xi_{1},\,\dots,\,\xi_{s-1} \in \jad$. Since 
\begin{equation*}
w_{1}z_{\xi_{1}}t_{\xi_{1}} \sig \cdots \sig 
 w_{s-1}z_{\xi_{s-1}}t_{\xi_{s-1}} \sig w
\end{equation*}
in the semi-infinite Bruhat order, 
it follows from Remark~\ref{rem:SiB} (and the definition of 
the semi-infinite Bruhat order) that 
$[\xi_{s-1}] = [\xi_{s-1}-0] \in \QJp{I \setminus J}$, and 
$[\xi_{u}-\xi_{u+1}] \in \QJp{I \setminus J}$ for every $1 \le u \le s-2$; 
in particular, $[\xi_{u}] \in \QJp{I \setminus J}$ for all $1 \le u \le s-1$. 
Therefore, we have
\begin{equation*}
w_{u}z_{\xi_{u}}t_{\xi_{u}}\lambda \in 
w_{u}\lambda+\BZ_{\le 0}\delta \subset 
\lambda-Q^{+}+\BZ_{\le 0}\delta
\end{equation*}
for every $1 \le u \le s-1$. Also, we have 
$w\lambda \in \lambda - Q^{+}$. 
From these, we deduce that 
%
%
\begin{equation} \label{eq:wtep}
\wt(\eta_{\psi}) = 
\underbrace{\lambda-\beta}_{=\wt(\psi)} + K \delta \quad 
\text{for some $\beta \in Q^{+}$ and $K \in \BZ_{\le 0}$}.
\end{equation}
We define $\Degt(\psi):=K \in \BZ_{\le 0}$; 
it is easily seen that the $\Degt$ thus defined agrees with 
$\Deg_{\lambda} \circ S$ in \cite[Corollary~4.8]{LNSSS2}, 
where $\Deg_{\lambda}:\QLS(\lambda)=\BB(\lambda)_{\cl} \rightarrow \BZ_{\le 0}$ 
is the degree function defined in \cite[\S3.1]{NSdeg} (see also \cite[\S4.2]{LNSSS2}), 
and $S$ is the Lusztig involution on $\QLS(\lambda)=\BB(\lambda)_{\cl}$ 
(see \cite[\S4.5]{LNSSS2}). The next proposition follows from the proof of 
\cite[Proposition~7.8]{LNSSS2}.
%
%
\begin{prop} \label{prop:DegP}
Keep the notation and setting above. 
There holds the equality
\begin{equation} \label{eq:DegP}
\sum_{\psi \in \QLS(\lambda)} 
q^{\Degt(\psi)}x^{\wt (\psi)} = 
P_{\lambda}(x\,;\,q^{-1},\,0).
\end{equation}
\end{prop}
%
%
\subsection{Proof of the graded character formula.}
\label{subsec:prf-grch1}

We recall from \S\ref{subsec:demazure} that 
$V_{e}^{-}(\lambda) = 
\bigoplus_{b \in \CB_{e}^{-}(\lambda)} \BQ(q_{s})G(b)$. 
Therefore, by Theorem~\ref{thm:main}, we obtain
\begin{equation*}
\ch V_{e}^{-}(\lambda) = 
\sum_{\eta \in \BB^{\si}_{\sige e}(\lambda)} x^{\wt (\eta)}.
\end{equation*}
Because 
$\BB^{\si}_{\sige e}(\lambda) = 
\bigsqcup_{\psi \in \QLS(\lambda)} 
 \bigl(\cl^{-1}(\psi) \cap \BB^{\si}_{\sige e}(\lambda)\bigr)$, we see that
\begin{equation} \label{eq:ch1}
\ch V_{e}^{-}(\lambda) = 
\sum_{\psi \in \QLS(\lambda)} 
\Biggl(\underbrace{\sum_{\eta \in \cl^{-1}(\psi) \cap \BB^{\si}_{\sige e}(\lambda)} 
x^{\wt (\eta)}}_{(\ast)}\Biggr).
\end{equation}
In order to obtain a graded character formula for $V_{e}^{-}(\lambda)$, 
we will compute the sum ($\ast$) above of the terms $x^{\wt(\eta)}$ 
over all $\eta \in \cl^{-1}(\psi) \cap \BB^{\si}_{\sige e}(\lambda)$
for each $\psi \in \QLS(\lambda)$. 
Let $\psi \in \QLS(\lambda)$, and take 
$\eta_{\psi} \in \BB^{\si}_{0}(\lambda)$ 
as in Lemma~\ref{lem:pi-eta}. 
Let $X$ be a monomial in root operators 
such that $\eta_{\psi}=X\eta_{e}$; we see that 
$\psi = X \cl(\eta_{e})$. By Lemma~\ref{lem:cl-inv}, 
we have 
\begin{equation*}
\cl^{-1}(\psi)=\bigl\{X S_{t_{\zeta}}\eta^{C} \mid 
 C \in \Conn(\BB^{\si}(\lambda)),\,
 \zeta \in Q^{\vee}_{I \setminus J} \bigr\}.
\end{equation*}

We claim that
%
%
\begin{equation} \label{eq:grch1-1}
\cl^{-1}(\psi) \cap \BB^{\si}_{\sige e}(\lambda) = 
\bigl\{X S_{t_{\zeta}}\eta^{C} \mid 
 C \in \Conn(\BB^{\si}(\lambda)),\,
 \zeta \in Q^{\vee+}_{I \setminus J} \bigr\}.
\end{equation}
First we show the inclusion $\subset$. 
Let $\eta \in \cl^{-1}(\psi) \cap \BB^{\si}_{\sige e}(\lambda)$, 
and write $\eta$ as $\eta=X S_{t_{\zeta}}\eta^{C}$ 
for $C \in \Conn(\BB^{\si}(\lambda))$ and 
$\zeta \in Q^{\vee}_{I \setminus J}$. 
Let $w:=\kappa(\eta_{\psi}) \in W^{J}$; note that 
$\eta_{e} = (e \,;\,0,\,1)$ is of the form as in \cite[(7.1.1)]{INS}, 
and $X\eta_{e} = \eta_{\psi}$ is of the form:
\begin{equation*}
X\eta_{e} = (\dots,\,w \,;\,0,\,\dots,\,1) = 
(\dots,\,w e \,;\,0,\,\dots,\,1). 
\end{equation*}
We see from Remark~\ref{rem:extp}\,(2) that 
$S_{t_{\zeta}}\eta^{C}$ is also of the form as in \cite[(7.1.1)]{INS}, 
with $\kappa(S_{t_{\zeta}}\eta^{C}) = \PJ(t_{\zeta})$ 
(recall that $\kappa(\eta^{C})=e$). 
Therefore, we deduce from \cite[Lemma~7.1.4]{INS} that 
$XS_{t_{\zeta}}\eta^{C}$ is of the form:
\begin{equation*}
XS_{t_{\zeta}}\eta^{C} = 
 (\dots,\,w \PJ(t_{\zeta}) \,;\,0,\,\dots,\,1),
\end{equation*}
and hence that 
$\kappa(XS_{t_{\zeta}}\eta^{C}) = w \PJ(t_{\zeta}) = 
 wz_{\zeta}t_{\zeta+\phi_{J}(\zeta)}$ (see Lemma~\ref{lem:J-adj}\,(2)). 
Since $\eta = XS_{t_{\zeta}}\eta^{C} \in \BB^{\si}_{\sige e}(\lambda)$, we have 
$wz_{\zeta}t_{\zeta+\phi_{J}(\zeta)} \sige e$. 
Hence it follows from Remark~\ref{rem:SiB} that 
$\zeta = [\zeta+\phi_{J}(\zeta)] \in Q_{I \setminus J}^{\vee}$ 
is contained in $Q^{\vee+}_{I \setminus J}$. 
Thus, $\eta$ is contained in the set on the right-hand side of \eqref{eq:grch1-1}. 
Conversely, let $C \in \Conn(\BB^{\si}(\lambda))$, and 
$\zeta \in Q^{\vee+}_{I \setminus J}$. 
Then, by the same argument as above, we see that 
$\kappa(X S_{t_{\zeta}}\eta^{C})$ is equal to $wz_{\zeta}t_{\zeta+\phi_{J}(\zeta)}$, 
where $w=\kappa(\eta_{\psi}) \in W^{J}$. 
Since $[(\zeta+\phi_{J}(\zeta))-0] = [\zeta] = \zeta \in 
Q^{\vee+}_{I \setminus J}$, we see from \cite[Proposition~6.2.2]{INS} 
(with $a=1)$ that $z_{\zeta}t_{\zeta+\phi_{J}(\zeta)} 
\sige e \ (=t_{0})$. 
Also, we see from Lemma~\ref{lem:SiB2} that 
$w z_{\zeta}t_{\zeta+\phi_{J}(\zeta)} \sige z_{\zeta}t_{\zeta+\phi_{J}(\zeta)}$, 
since $w \ge e$ in the (ordinary) Bruhat order $\ge$ on $W^{J}$.
Combining these inequalities, 
we obtain $\kappa(X S_{t_{\zeta}}\eta^{C}) = 
wz_{\zeta}t_{\zeta+\phi_{J}(\zeta)} \sige e$, 
which implies that $X S_{t_{\zeta}}\eta^{C} \in 
\cl^{-1}(\psi) \cap \BB^{\si}_{\sige e}(\lambda)$. 
This proves claim \eqref{eq:grch1-1}. 

Let $C \in \Conn(\BB^{\si}(\lambda))$, and write $\Theta(C) \in \Par(\lambda)$ as: 
$\Theta(C)=(\rho^{(i)})_{i \in I}$, with 
$\rho^{(i)} = (\rho^{(i)}_{1} \ge \cdots \ge \rho^{(i)}_{m_{i}-1})$ for each $i \in I$. 
Also, let $\zeta \in Q^{\vee+}_{I \setminus J}$, and 
write it as $\zeta = \sum_{i \in I} c_{i}\alpha_{i}^{\vee}$, 
with $c_{i} \in \BZ_{\ge 0}$, $i \in I$. 
For each $i \in I$, we set $c_{i}+\rho^{(i)} := 
(c_{i}+\rho^{(i)}_{1} \ge \cdots \ge c_{i}+\rho^{(i)}_{m_{i}-1} \ge c_{i})$, 
which is a partition of length less than or equal to $m_{i}$. Then we set 
%
%
\begin{equation} \label{eq:par+}
(c_{i})_{i \in I} + \Theta(C) : = 
(c_{i}+\rho^{(i)})_{i \in I} \in \ol{\Par(\lambda)};
\end{equation}
for the definition of $\ol{\Par(\lambda)}$, see \eqref{eq:olpar}. 
We compute: 
\begin{align*}
\wt (S_{t_{\zeta}}\eta^{C}) & = 
t_{\zeta} (\wt (\eta^{C})) = 
t_{\zeta}\bigl(\lambda-|(\rho^{(i)})_{i \in I}|\delta\bigr) =
\lambda - \pair{\zeta}{\lambda}\delta - |(\rho^{(i)})_{i \in I}|\delta \\[1.5mm]
& = \lambda - \biggl(\sum_{i \in I} m_{i}c_{i}\biggr)\delta  - |(\rho^{(i)})_{i \in I}|\delta 
  = \wt (\eta_{e}) - |(c_{i}+\rho^{(i)})_{i \in I}|\delta. 
\end{align*}
From this computation, together with \eqref{eq:wtep}, 
we deduce that
%
%
\begin{align} 
\wt ( X S_{t_{\zeta}}\eta^{C} ) & = 
\wt ( X\eta_{e} ) - |(c_{i}+\rho^{(i)})_{i \in I}|\delta = 
\wt ( \eta_{\psi} ) - |(c_{i}+\rho^{(i)})_{i \in I}|\delta \nonumber \\
& = \wt (\psi) + \bigl(\Degt(\psi) - |(c_{i}+\rho^{(i)})_{i \in I}|\bigr)\delta. 
\label{eq:wtXS}
\end{align}
Therefore, 
we conclude that for each $\psi \in \QLS(\lambda)$, 
\begin{align*}
\sum_{\eta \in \cl^{-1}(\psi) \cap \BB^{\si}_{\sige e}(\lambda)} x^{\wt (\eta)} 
 & = \sum_{
      \begin{subarray}{c} 
       C \in \Conn(\BB^{\si}(\lambda)) \\[1mm] 
       \zeta \in Q^{\vee+}_{I \setminus J}
      \end{subarray}}
     x^{\wt (XS_{t_{\zeta}}\eta^{C})} 
   = x^{\wt(\psi)}x^{\Degt(\psi)\delta} 
     \sum_{\bc \in \ol{\Par}(\lambda)} x^{-|\bc|\delta} \\[3mm]
 & = x^{\wt(\psi)} q^{\Degt(\psi)} 
     \sum_{\bc \in \ol{\Par}(\lambda)} q^{-|\bc|} \qquad 
     \text{(by replacing $x^{\delta}$ with $q$)} \\[3mm]
 & = x^{\wt(\psi)} q^{\Degt(\psi)}
     \prod_{i \in I} \prod_{r=1}^{m_{i}}(1-q^{-r})^{-1}.
\end{align*}
Substituting this into \eqref{eq:ch1}, 
we finally obtain
\begin{align*}
\gch V_{e}^{-}(\lambda) & = 
  \sum_{\psi \in \QLS(\lambda)} x^{\wt(\psi)} q^{\Degt(\psi)}
     \prod_{i \in I} \prod_{r=1}^{m_{i}}(1-q^{-r})^{-1} \\
  & = \biggl(\prod_{i \in I} \prod_{r=1}^{m_{i}}(1-q^{-r})\biggr)^{-1}
      P_{\lambda}(x\,;\,q^{-1},\,0) \qquad 
      \text{by Proposition~\ref{prop:DegP}}.
\end{align*}
This completes the proof of Theorem~\ref{thm:grch1}.
%
%
\begin{rem} \label{rem:wtXS}
Let $\psi \in \QLS(\lambda)$. 
We see from \eqref{eq:wtXS} that 
for every $\eta \in \cl^{-1}(\psi) \cap \BB^{\si}_{\sige e}(\lambda)$, 
\begin{equation*}
\wt (\eta) - \wt (\eta_{\psi}) \in \BZ_{\le 0} \delta,
\end{equation*}
with $\wt (\eta) - \wt (\eta_{\psi}) = 0$ if and only if 
$\eta =\eta_{\psi}$. 
\end{rem}
%
%
\subsection{Graded character formula for $V_{w_0}^{+}(\lambda)$.} 
\label{subsec:gch-w0}

We define the character $\ch V_{w_0}^{+}(\lambda)$ and 
the graded character $\gch V_{w_0}^{+}(\lambda)$ of 
the Demazure submodule $V_{w_0}^{+}(\lambda)$ in exactly the same manner 
as those of $V_{e}^{-}(\lambda)$ are defined in \S\ref{subsec:gch-e}. 

Recall from \S\ref{subsec:dual-sLS} the bijection 
$\vee:\BB^{\si}(\lambda) \rightarrow \BB^{\si}(-w_{0}\lambda)$. 
It follows from Lemma~\ref{lem:dual} that
%
%
\begin{equation} \label{eq:Demv}
\Bigl(\BB^{\si}_{\mcr{w_{0}}^{J} \sige}(\lambda)\Bigr)^{\vee} = 
\BB^{\si}_{\sige e}(-w_{0}\lambda).
\end{equation}
From this, together with Theorem~\ref{thm:main}, 
we deduce that 
\begin{align*}
\ch V_{w_0}^{+}(\lambda) = 
 \sum_{\eta \in \BB^{\si}_{\mcr{w_{0}}^{J} \sige}(\lambda)} x^{\wt(\eta)}
  = 
 \sum_{\eta \in \BB^{\si}_{\sige e}(-w_0\lambda)} x^{\wt(\eta^{\vee})}
   =
 \sum_{\eta \in \BB^{\si}_{\sige e}(-w_0\lambda)} x^{-\wt(\eta)}
 \quad \text{by \eqref{eq:veep}}, 
\end{align*}
which is equal to the one obtained from $\ch V_{e}^{-}(-w_{0}\lambda)$ 
by replacing $x$ with $x^{-1}$. Therefore,
the graded character $\gch V_{w_0}^{+}(\lambda)$ is 
obtained from the graded character $\gch V_{e}^{-}(-w_{0}\lambda)$ by 
replacing $x$ with $x^{-1}$, and $q$ with $q^{-1}$. 
%
%
%
%
\begin{thm} \label{thm:grch2}
Keep the notation and setting above. 
The graded character $\gch V_{w_0}^{+}(\lambda)$ of 
$V_{w_0}^{+}(\lambda)$ can be expressed as 
\begin{align*}
\gch V_{w_0}^{+}(\lambda) 
& = 
\biggl(\prod_{i \in I} \prod_{r=1}^{m_{i}}(1-q^{r})\biggr)^{-1} 
P_{-w_{0}\lambda}(x^{-1}\,;\,q,\,0) \\[1.5mm]
& = 
\biggl(\prod_{i \in I} \prod_{r=1}^{m_{i}}(1-q^{r})\biggr)^{-1} 
P_{\lambda}(x\,;\,q,\,0).
\end{align*}
\end{thm}
\begin{proof}
The first equality follows immediately from Theorem~\ref{thm:grch1}, 
together with the comment preceding this theorem. 
For the second equality, it suffices to show that 
$P_{-w_{0}\lambda}(x^{-1}\,;\,q,\,0) = P_{\lambda}(x\,;\,q,\,0)$.
In the proof of \cite[Proposition~7.8]{LNSSS2}, we showed that 
\begin{equation} \label{eq:P}
P_{\lambda}(x\,;\,q,\,0) =
\sum_{\eta \in \QLS(\lambda)}
q^{-\Deg_{\lambda}(S(\eta))}x^{\wt (\eta)}, 
\end{equation}
where $S$ is the Lusztig involution on $\QLS(\lambda)$ 
(see \cite[\S4.5]{LNSSS2}). Replacing $\lambda$ and $x$ with
$-w_{0}\lambda$ and $x^{-1}$, respectively, we obtain 
from \eqref{eq:P}
\begin{equation*}
P_{-w_{0}\lambda}(x^{-1}\,;\,q,\,0) = 
\sum_{\eta \in \QLS(-w_{0}\lambda)}
q^{-\Deg_{-w_{0}\lambda}(S(\eta))}x^{-\wt (\eta)}. 
\end{equation*}
Here, we know from \cite[Corollary~7.4]{LNSSS2} that 
$\Deg_{-w_{0}\lambda}(S(\eta)) = \Deg_{\lambda}(\eta^{\ast})$ for all
$\eta \in \QLS(-w_{0}\lambda)$, where $\eta^{\ast} \in \QLS(\lambda)$ 
is the ``dual'' QLS path of $\eta \in \QLS(-w_{0}\lambda)$ 
(see \cite[(4.17)]{LNSSS2}). Also, we have $-\wt (\eta) = \wt(\eta^{\ast})$ 
by \cite[(4.18)]{LNSSS2}. Therefore, we deduce that 
\begin{equation*}
P_{-w_{0}\lambda}(x^{-1}\,;\,q,\,0) = 
\sum_{\eta \in \QLS(-w_{0}\lambda)}
q^{-\Deg_{\lambda}(\eta^{\ast})}x^{\wt (\eta^{\ast})}. 
\end{equation*}
Since $\eta^{\ast} \in \QLS(\lambda)$ if and only if 
$\eta \in \QLS(-w_{0}\lambda)$, we conclude that
\begin{equation*}
P_{-w_{0}\lambda}(x^{-1}\,;\,q,\,0) = 
\sum_{\eta \in \QLS(\lambda)}
q^{-\Deg_{\lambda}(\eta)}x^{\wt (\eta)}. 
\end{equation*}
By \cite[Proposition~7.8]{LNSSS2}, 
the right-hand side of this equality is 
identical to $P_{\lambda}(x\,;\,q,\,0)$. Thus, we have shown that 
$P_{-w_{0}\lambda}(x^{-1}\,;\,q,\,0) = P_{\lambda}(x\,;\,q,\,0)$, as desired. 
This proves the theorem. 
\end{proof}

%
\section{Certain quotients of Demazure submodules.}
\label{sec:quotient}

In this section, we fix 
$\lambda=\sum_{i \in I} m_{i}\vpi_{i} \in P^{+}$, 
and set $J=J_{\lambda}:=\bigl\{i \in I \mid \pair{\alpha_{i}^{\vee}}{\lambda}=0\bigr\}$. 
%
%
\subsection{Some technical lemmas.}
\label{subsec:somelems}

The following is an easy lemma.
%
%
\begin{lem}[see Lemma~\ref{lem:pi-eta}] \label{lem:epc}
The subset $\bigl\{\eta_{\psi} \mid \psi \in \QLS(\lambda) \bigr\} 
\cup \bigl\{\bzero\bigr\}$ of $\BB^{\si}(\lambda) \cup \bigl\{\bzero\bigr\}$ 
is stable under the action of the root operators $e_{j}$ and $f_{j}$ for $j \in I$.
\end{lem}

\begin{proof}
It suffices to show that $x_{j}\eta_{\psi} = \eta_{x_{j}\psi}$, 
assuming that $x_{j}\eta_{\psi} \ne \bzero$ 
for $\psi \in \QLS(\lambda)$ and $j \in I$, 
where $x_{j}$ is either $e_{j}$ or $f_{j}$. 
Since the map $\cl:\BB^{\si}(\lambda) \twoheadrightarrow \QLS(\lambda)$ 
commutes with root operators, it follows that 
$\cl(x_{j}\eta_{\psi})=x_{j}\psi \ (\ne \bzero)$. 
It is obvious that $x_{j}\eta_{\psi} \in \BB^{\si}_{0}(\lambda)$. 
Also, we deduce from the definition of root operators 
that $\kappa(x_{j}\eta_{\psi})$ is equal either 
to $\kappa(\eta_{\psi}) \in W^{J}$ or 
to $r_{j}\kappa(\eta_{\psi}) \in (W^{J})_{\af}$. 
In the latter case, we have 
$r_{j}\kappa(\eta_{\psi}) \in W \cap (W^{J})_{\af}$ 
since $\kappa(\eta_{\psi}) \in W^{J}$ and $j \in I$, 
which implies that $r_{j}\kappa(\eta_{\psi}) \in W^{J}$
since $W \cap (W^{J})_{\af} = W^{J}$ by \eqref{eq:W^J_af}. 
Therefore, by the uniqueness of $\eta_{x_{j}\psi}$, we obtain 
$x_{j}\eta_{\psi} = \eta_{x_{j}\psi}$, as desired. 
\end{proof}

We know from \cite[\S5.1]{LNSSS2} that there exists 
a bijection $\ast : \QLS(\lambda) \rightarrow \QLS(-w_{0}\lambda)$ 
such that for $\psi \in \QLS(\lambda)$, 
\begin{equation} \label{eq:qls*}
\begin{cases}
\wt(\psi^{\ast})=-\wt(\psi), \quad \text{and} \\[1.5mm]
(e_{j}\psi)^{\ast}=f_{j}\psi^{\ast}, 
(f_{j}\psi)^{\ast}=e_{j}\psi^{\ast} \quad 
\text{for all $j \in I_{\af}$}, 
\end{cases}
\end{equation}
where we set $\bzero^{\ast}:=\bzero$. 
We see easily from the definitions that 
the following diagram commutes: 
%
%
\begin{equation} \label{eq:CDva}
\begin{CD}
\BB^{\si}(\lambda) @>{\vee}>> \BB^{\si}(-w_{0}\lambda) \\
@V{\cl}VV @VV{\cl}V \\
\QLS(\lambda) @>{\ast}>> \QLS(-w_{0}\lambda).
\end{CD}
\end{equation}
The next lemma follows immediately from Lemma~\ref{lem:pi-eta}, 
by using the commutative diagram \eqref{eq:CDva}, 
together with \eqref{eq:veep} and \eqref{eq:qls*}. 
%
%
\begin{lem}[{cf. \cite[Proposition~3.1.3]{NSdeg}}] \label{lem:pi-eta2}
For each $\psi \in \QLS(\lambda)$, 
the element $\ti{\eta}_{\psi} : = (\eta_{\psi^{\ast}})^{\vee}$ is 
a unique element in $\BB^{\si}_{0}(\lambda)$ 
such that $\cl(\ti{\eta}_{\psi})=\psi$ and $\iota(\eta_{\psi}) \in W^{J}$. 
\end{lem}
We can prove the next lemma by an argument similar 
to the one for Lemma~\ref{lem:epc}. 
%
%
\begin{lem} \label{lem:epc2}
The subset $\bigl\{\ti{\eta}_{\psi} \mid \psi \in \QLS(\lambda) \bigr\} 
\cup \bigl\{\bzero\bigr\}$ of $\BB^{\si}(\lambda) \cup \bigl\{\bzero\bigr\}$ 
is stable under the action of the root operators $e_{j}$ and $f_{j}$ for $j \in I$.
\end{lem}

\begin{rem}
By Lemmas~\ref{lem:epc} and \ref{lem:epc2}, each of the sets
$\bigl\{\eta_{\psi} \mid \psi \in \QLS(\lambda) \bigr\}$ and 
$\bigl\{\ti{\eta}_{\psi} \mid \psi \in \QLS(\lambda) \bigr\}$
has a crystal structure for $U_{q}(\Fg)$, where 
$\Fg$ is the canonical finite-dimensional simple Lie subalgebra of $\Fg_{\af}$.
Moreover, these crystals for $U_{q}(\Fg)$ are both isomorphic to 
$\QLS(\lambda)$, regarded as a crystal for $U_{q}(\Fg)$ by restriction. 
\end{rem}
%
%
\subsection{Certain quotients of Demazure submodules and their crystal bases.}
\label{subsec:quot}

We define
\begin{equation*}
X_{e}^{-}(\lambda):=
\sum_{
 \begin{subarray}{c}
  \bc \in \ol{\Par(\lambda)} \\[1.5mm]
  \bc \ne (\emptyset)_{i \in I}
 \end{subarray}
 } U_{q}^{-} S_{\bc}^{-}v_{\lambda}; 
\end{equation*}
note that 
$X_{e}^{-}(\lambda) \subset V_{e}^{-}(\lambda)=U_{q}^{-}v_{\lambda}$ 
since $S_{\bc}^{-} \in U_{q}^{-}$ for all $\bc \in \ol{\Par(\lambda)}$ 
(see \S\ref{subsec:isom}). We denote by 
$\Xi_{\lambda}^{-} : V_{e}^{-}(\lambda) \twoheadrightarrow 
V_{e}^{-}(\lambda)/X_{e}^{-}(\lambda)$ the canonical projection, and set
\begin{equation}
U_{w}^{-}(\lambda):=\Xi_{\lambda}^{-}(V_{w}^{-}(\lambda)) \qquad 
 \text{for each $w \in W^{J}$}; 
\end{equation}
we have $V_{w}^{-}(\lambda) \subset V_{e}^{-}(\lambda)$
since $w \sige e$ (see Corollary~\ref{cor:zt}). 
%
%
\begin{thm} \label{thm:quotient}
Keep the notation and setting above. 
\begin{enu}

\item There exists a subset $\CB(X_{e}^{-}(\lambda))$ of $\CB(\lambda)$ such that 
%
%
\begin{equation} \label{eq:GX}
X_{e}^{-}(\lambda) = 
 \bigoplus_{b \in \CB(X_{e}^{-}(\lambda))} \BQ(q_{s}) G(b).
\end{equation}
Under the isomorphism $\Psi_{\lambda}:\CB(\lambda) \stackrel{\sim}{\rightarrow} \BB^{\si}(\lambda)$, 
the subset $\CB(X_{e}^{-}(\lambda)) \subset \CB(\lambda)$ is mapped to the subset
\begin{equation*}
\BB^{\si}_{\sige e}(\lambda) \setminus \bigl\{\eta_{\psi} \mid \psi \in \QLS(\lambda)\bigr\}
\end{equation*}
of $\BB^{\si}(\lambda)$. Therefore, 
if we define $\CB(U_{e}^{-}(\lambda)) \subset \CB(\lambda)$ 
to be the inverse image of 
$\bigl\{\eta_{\psi} \mid \psi \in \QLS(\lambda)\bigr\} \subset \BB^{\si}_{0}(\lambda)$ 
under the isomorphism $\Psi_{\lambda}$, then 
$\bigl\{ \Xi_{\lambda}^{-}(G(b)) \mid b \in \CB(U_{e}^{-}(\lambda)) \bigr\}$ 
is a $\BQ(q_{s})$-basis of the quotient 
$U_{e}^{-}(\lambda)=V_{e}^{-}(\lambda)/X_{e}^{-}(\lambda)$.

\item For each $w \in W^{J}$, the quotient 
$U_{w}^{-}(\lambda)=\Xi_{\lambda}^{-}(V_{w}^{-}(\lambda))$ 
of $V_{w}^{-}(\lambda)$ has a $\BQ(q_{s})$-basis 
$\bigl\{ \Xi_{\lambda}^{-}(G(b)) \mid b \in \CB(U_{w}^{-}(\lambda)) \bigr\}$, where 
$\CB(U_{w}^{-}(\lambda))$ is defined to be the inverse image of the following subset of 
$\BB^{\si}_{0}(\lambda)$ under the isomorphism $\Psi_{\lambda}$: 
\begin{equation} \label{eq:qls-dem}
\bigl\{\eta_{\psi} \mid 
 \text{\rm $\psi \in \QLS(\lambda)$ such that $\kappa(\eta_{\psi}) \ge w$} \bigr\}, 
\end{equation}
where $\kappa(\eta_{\psi}) \ge w$ means that 
$\kappa(\eta_{\psi}) \in W^{J}$ is greater than or equal to $w \in W^{J}$ 
in the (ordinary) Bruhat order on $W^{J}$. 
\end{enu}
\end{thm}

We will prove Theorem~\ref{thm:quotient} in the next subsection. 

Similarly, we define
\begin{equation*}
X_{w_0}^{+}(\lambda):=
\sum_{
 \begin{subarray}{c}
  \bc \in \ol{\Par(\lambda)} \\[1.5mm]
  \bc \ne (\emptyset)_{i \in I}
 \end{subarray}
 } U_{q}^{+} S_{\bc}S_{w_{0}}^{\norm}v_{\lambda}; 
\end{equation*}
note that 
$X_{w_0}^{+}(\lambda) \subset V_{w_0}^{+}(\lambda)=
U_{q}^{+}S_{w_0}^{\norm}v_{\lambda}$ since 
$S_{\bc} \in U_{q}^{+}$ for all $\bc \in \ol{\Par(\lambda)}$. 
We denote by 
$\Xi_{\lambda}^{+} : V_{w_0}^{+}(\lambda) \twoheadrightarrow 
V_{w_0}^{+}(\lambda)/X_{w_0}^{+}(\lambda)$ the canonical projection, and set
\begin{equation}
U_{w}^{+}(\lambda):=\Xi_{\lambda}^{+}(V_{w}^{+}(\lambda)) \qquad 
 \text{for each $w \in W^{J}$}; 
\end{equation}
we have $V_{w}^{+}(\lambda) \subset V_{w_0}^{+}(\lambda)$
since $\mcr{w_0}^{J} \sige w$. 
%
%
\begin{thm} \label{thm:quotient2}
Keep the notation and setting above. 
\begin{enu}

\item There exists a subset 
$\CB(X_{w_0}^{+}(\lambda))$ of $\CB(\lambda)$ such that 
\begin{equation*}
X_{w_0}^{+}(\lambda) = 
 \bigoplus_{b \in \CB(X_{w_0}^{+}(\lambda))} \BQ(q_{s}) G(b).
\end{equation*}
Under the isomorphism $\Psi_{\lambda}^{\vee}:
\CB(\lambda) \stackrel{\sim}{\rightarrow} \BB^{\si}(\lambda)$, 
the subset $\CB(X_{w_0}^{+}(\lambda)) \subset \CB(\lambda)$ is mapped to 
the subset
\begin{equation*}
\BB^{\si}_{\mcr{w_0}^{J} \sige}(\lambda) \setminus 
 \bigl\{\ti{\eta}_{\psi} \mid \psi \in \QLS(\lambda)\bigr\}
\end{equation*}
of $\BB^{\si}(\lambda)$. Therefore, 
if we define $\CB(U_{w_0}^{+}(\lambda)) \subset \CB(\lambda)$ 
to be the inverse image of 
$\bigl\{\ti{\eta}_{\psi} \mid \psi \in \QLS(\lambda)\bigr\} \subset \BB^{\si}_{0}(\lambda)$ 
under the isomorphism $\Psi_{\lambda}^{\vee}$, then 
$\bigl\{ \Xi_{\lambda}^{+}(G(b)) \mid b \in \CB(U_{w_0}^{+}(\lambda)) \bigr\}$ 
is a $\BQ(q_{s})$-basis of the quotient 
$U_{w_0}^{+}(\lambda)=V_{w_0}^{+}(\lambda)/X_{w_0}^{+}(\lambda)$.

\item For each $w \in W^{J}$, the quotient 
$U_{w}^{+}(\lambda)=\Xi_{\lambda}^{+}(V_{w}^{+}(\lambda))$ 
of $V_{w}^{+}(\lambda)$ has a $\BQ(q_{s})$-basis 
$\bigl\{ \Xi_{\lambda}^{+}(G(b)) \mid b \in \CB(U_{w}^{+}(\lambda)) \bigr\}$, where 
$\CB(U_{w}^{+}(\lambda))$ is defined to be the inverse image of the following subset of 
$\BB^{\si}_{0}(\lambda)$ under the isomorphism $\Psi_{\lambda}^{\vee}$: 
\begin{equation*}
\bigl\{\ti{\eta}_{\psi} \mid 
 \text{\rm $\psi \in \QLS(\lambda)$ such that $w \ge \iota(\ti{\eta}_{\psi})$} \bigr\}. 
\end{equation*}
\end{enu}
\end{thm}

We leave the proof of Theorem~\ref{thm:quotient2} to the reader
since it is similar to that of Theorem~\ref{thm:quotient}; 
use also \cite[Lemma~5.2]{BN}. 

In our forthcoming paper \cite{LNSSS3}, 
we will prove that for each $w \in W^{J}$, 
the graded character of $U_{w}^{+}(\lambda)$ 
is identical to the specialization at $t = 0$ of 
the nonsymmetric Macdonald polynomial $E_{w\lambda}(x;\,q,\,t)$. 
This generalizes \cite[Corollary~7.10]{LNSSS2}, since 
$E_{w_{0}\lambda}(x;\,q,\,0)=P_{\lambda}(x;\,q,\,0)$. 
%
%
\subsection{Proof of Theorem~\ref{thm:quotient}.}
\label{subsec:prf-quot}

Recall from \S\ref{subsec:properties1b} 
the $U_{q}$-module embedding
\begin{equation*}
\Phi_{\lambda} : V(\lambda) \hookrightarrow \ti{V}(\lambda)=
\bigotimes_{i \in I} V(\vpi_{i})^{\otimes m_{i}}
\end{equation*}
that maps $v_{\lambda}$ to $\ti{v}_{\lambda}=
\bigotimes_{i \in I} v_{\vpi_{i}}^{\otimes m_{i}}$. 
For each $\bc = (\rho^{(i)})_{i \in I} \in \ol{\Par(\lambda)}$, 
we define a $U_{q}'$-module homomorphism 
$s_{\bc}(z^{-1}):\ti{V}(\lambda) \rightarrow \ti{V}(\lambda)$ by:
\begin{equation*}
s_{\bc}(z^{-1})=\prod_{i \in I} 
s_{\rho^{(i)}}(z_{i,1}^{-1},\,\dots,\,z_{i,m_i}^{-1}), 
\end{equation*}
where for $i \in I$ and $1 \le l \le m_{i}$, 
$z_{i,\,l}:\ti{V}(\lambda) \stackrel{\sim}{\rightarrow} \ti{V}(\lambda)$ 
is the $U_{q}'$-module automorphism of $\ti{V}(\lambda)$, and 
for $i \in I$, $s_{\rho^{(i)}}(x_{1},\,\dots,\,x_{m_i})$ 
denotes the Schur polynomial corresponding to the partition $\rho^{(i)}$.
We claim that
\begin{equation*}
s_{\bc}(z^{-1})(\Img \Phi_{\lambda}) \subset \Img \Phi_{\lambda}. 
\end{equation*}
Indeed, since $V(\lambda) = U_{q}v_{\lambda} = U_{q}'v_{\lambda}$, 
it follows that $\Img \Phi_{\lambda} = U_{q} \ti{v}_{\lambda} 
= U_{q}' \ti{v}_{\lambda}$. Therefore, we see from \cite[Proposition 4.10]{BN} that 
\begin{equation*}
s_{\bc}(z^{-1})(\Img \Phi_{\lambda}) = 
s_{\bc}(z^{-1})(U_{q}'\ti{v}_{\lambda}) = 
U_{q}'s_{\bc}(z^{-1})\ti{v}_{\lambda} = 
\underbrace{U_{q}'S_{\bc}^{-}}_{\in U_{q}}\ti{v}_{\lambda} \subset \Img \Phi_{\lambda}, 
\end{equation*}
as desired. Hence we can define a $U_{q}'$-module homomorphism
$z_{\bc}:V(\lambda) \rightarrow V(\lambda)$ in such a way that 
the following diagram commutes:
\begin{equation}
\begin{CD}
V(\lambda) @>{\Phi_{\lambda}}>> \ti{V}(\lambda) \\
@V{z_{\bc}}VV @VV{s_{\bc}(z^{-1})}V \\
V(\lambda) @>{\Phi_{\lambda}}>> \ti{V}(\lambda).
\end{CD}
\end{equation}
Here, note that $z_{\bc}v_{\lambda} = S_{\bc}^{-}v_{\lambda}$, and that 
$z_{\bc}$ commutes with Kashiwara operators on $V(\lambda)$. 
Because $z_{i,l}$ preserves the crystal lattice $\ti{\CL}(\lambda)=
\bigotimes_{i \in I} \CL(\vpi_{i})^{\otimes m_{i}} \subset \ti{V}(\lambda)$ 
for all $i \in I$ and $1 \le l \le m_{i}$ (see \S\ref{subsec:properties1b}), and 
because $\Phi_{\lambda}(\CL(\lambda)) \subset \ti{\CL}(\lambda)$, 
we deduce that $z_{\bc}(\CL(\lambda)) \subset \CL(\lambda)$. 
Thus, we obtain an induced $\BQ$-linear map 
$z_{\bc}:\CL(\lambda)/q\CL(\lambda) \rightarrow \CL(\lambda)/q\CL(\lambda)$, 
for which the following diagram commutes: 
\begin{equation}
\begin{CD}
\CL(\lambda)/q_{s}\CL(\lambda) 
  @>{\Phi_{\lambda}|_{q=0}}>> 
\ti{\CL}(\lambda)/q_{s}\ti{\CL}(\lambda) \\
@V{z_{\bc}}VV @VV{s_{\bc}(z^{-1})}V \\
\CL(\lambda)/q_{s}\CL(\lambda) 
  @>{\Phi_{\lambda}|_{q=0}}>> 
\ti{\CL}(\lambda)/q_{s}\ti{\CL}(\lambda).
\end{CD}
\end{equation}
It follows from \cite[p.\,371]{BN} (see also \eqref{eq:BN413}) that
%
%
\begin{equation} \label{eq:CBlam}
\CB(\lambda) = \bigl\{
 z_{\bc}b \mid 
 \bc \in \Par(\lambda),\,b \in \CB_{0}(\lambda) \bigr\}. 
\end{equation}
Also, by \eqref{eq:scu}, we have 
$z_{\bc}u_{\lambda} = u^{\bc}$ for $\bc \in \Par(\lambda)$ 
(for the definition of $u^{\bc}$, see Proposition~\ref{prop:ext}). 
%
%
\begin{rem} \label{rem:par}
Let $\bc = (\rho^{(i)})_{i \in I} \in \ol{\Par(\lambda)}$.
Let $c_{i} \in \BZ_{\ge 0}$, $i \in I$, be the number of columns of 
length $m_{i}$ in the Young diagram corresponding to the partition $\rho^{(i)}$, 
and set $\xi : = \sum_{i \in I} c_{i}\alpha_{i}^{\vee} \in Q^{\vee+}$; 
note that $\xi \in Q^{\vee+}_{I \setminus J}$.
Also, let $\varrho^{(i)}$, $i \in I$, denote
the partition corresponding to the Young diagram 
obtained from the Young diagram corresponding to $\rho^{(i)}$
by removing all columns of length $m_{i}$ 
(i.e., the first $c_{i}$-columns), 
and set $\bc':=(\varrho^{(i)})_{i \in I}$; 
note that $\bc' \in \Par(\lambda)$. Then we deduce from 
\cite[Lemma~4.14 and its proof]{BN} that
%
%
\begin{equation} \label{eq:zcu}
z_{\bc}u_{\lambda} = S_{t_{\xi}}z_{\bc'}u_{\lambda} = 
S_{t_{\xi}} u^{\bc'}.
\end{equation}
\end{rem}
%
%
\begin{lem} \label{lem:Be-}
We have
%
%
\begin{equation} \label{eq:Be-1}
\CB_{e}^{-}(\lambda) = \bigl\{
 z_{\bc}b \mid 
 \bc \in \Par(\lambda),\,
 b \in \CB_{e}^{-}(\lambda) \cap \CB_{0}(\lambda) \bigr\}.
\end{equation}
Moreover, for every 
$\bc \in \ol{\Par(\lambda)}$ and 
$b \in \CB_{e}^{-}(\lambda) \cap \CB_{0}(\lambda)$, 
the element $z_{\bc}b$ is contained in $\CB_{e}^{-}(\lambda)$. 
\end{lem}

\begin{proof}
First we prove the inclusion $\supset$. 
Let $b \in \CB_{e}^{-}(\lambda) \cap \CB_{0}(\lambda)$, and 
write it as $b = X u_{\lambda}$ for a monomial $X$ in Kashiwara operators. 
For $\bc \in \Par(\lambda)$, we have $z_{\bc}b=X z_{\bc}u_{\lambda}=Xu^{\bc}$. 
Set $\eta : = \Psi_{\lambda}(b)$ and $\eta' : = \Psi_{\lambda}(z_{\bc}b)$, 
where $\Psi_{\lambda}:\CB(\lambda) \stackrel{\sim}{\rightarrow} \BB^{\si}(\lambda)$ 
is the isomorphism of crystals. 
Then, we have $\eta = X \eta_{e}$ and $\eta' = X \Psi_{\lambda}(u^{\bc})=X\eta^{C}$, 
with $C:=\Theta^{-1}(\bc) \in \Conn(\BB^{\si}(\lambda))$. 
Therefore, we deduce from \cite[Lemma~7.1.4]{INS} that 
$\kappa(\eta) = \kappa(\eta')$. 
Also, since $b \in \CB_{e}^{-}(\lambda) \cap \CB_{0}(\lambda)$, it follows that 
$\kappa(\eta) \sige e$, and hence $\kappa(\eta') = \kappa(\eta) \sige e$. 
Thus we obtain $\eta' \in \BB^{\si}_{\sige e}(\lambda)$, which implies that 
$z_{\bc}b \in \CB_{e}^{-}(\lambda)$. 

Next we prove the opposite inclusion $\subset$. 
Let $b' \in \CB_{e}^{-}(\lambda)$, and write it as 
$b' = z_{\bc}b$ for some $\bc \in \Par(\lambda)$ and 
$b \in \CB_{0}(\lambda)$ (see \eqref{eq:CBlam}); 
we need to show that $b \in \CB_{e}^{-}(\lambda)$. 
Set $\eta : = \Psi_{\lambda}(b) \in \BB^{\si}(\lambda)$, and 
$\eta' : = \Psi_{\lambda}(b') \in \BB^{\si}(\lambda)$. 
Then, by entirely the same argument as above, we deduce that 
$\kappa(\eta) = \kappa(\eta') \sige e$. 
Thus we obtain $\eta \in \BB^{\si}_{\sige e}(\lambda)$, 
which implies that $b \in \CB_{e}^{-}(\lambda)$. 

For the second assertion, 
let $\bc = (\rho^{(i)})_{i \in I} \in \ol{\Par(\lambda)}$, and $b \in 
\CB_{e}^{-}(\lambda) \cap \CB_{0}(\lambda)$.
We write $b$ as $b=Xu_{\lambda}$ 
for a monomial $X$ in Kashiwara operators. 
Define $\xi = \sum_{i \in I} c_{i}\alpha_{i}^{\vee} \in Q^{\vee+}$, and 
$\bc'=(\varrho^{(i)})_{i \in I} \in \Par(\lambda)$ as in Remark~\ref{rem:par} 
(with $\bc = (\rho^{(i)})_{i \in I}$ above). Then, we have
\begin{equation*}
z_{\bc}b = X z_{\bc}u_{\lambda} = X S_{t_{\xi}}z_{\bc'}u_{\lambda} = 
X S_{t_{\xi}} u^{\bc'}. 
\end{equation*}
Now we set $\eta:=\Psi_{\lambda}(b) \in \BB^{\si}(\lambda)$, 
and $\psi:=\cl(\eta) \in \QLS(\lambda)$; note that 
$\eta = X \eta_{e}$, and hence $\psi = X \cl(\eta_{e})$. 
We see that
\begin{equation*}
\Psi_{\lambda}(z_{\bc}b) 
 = XS_{t_{\xi}}\Psi_{\lambda}(u^{\bc'})
 = XS_{t_{\xi}}\eta^{C}, \quad 
\text{with $C:=\Theta^{-1}(\bc') \in \Conn(\BB^{\si}(\lambda))$}.
\end{equation*}
Since $\xi \in Q_{I \setminus J}^{\vee+}$, it follows from \eqref{eq:grch1-1} that 
$\Psi_{\lambda}(z_{\bc}b) \in \BB^{\si}_{\sige e}(\lambda)$, which implies that 
$z_{\bc}b \in \CB^{-}_{e}(\lambda)$. This proves the lemma. 
\end{proof}

\begin{proof}[Proof of Theorem~\ref{thm:quotient}]
First, we prove that if we set
%
%
\begin{equation} \label{eq:GX1}
\CB := 
 \bigl\{z_{\bc}b \mid 
 \bc \in \ol{\Par(\lambda)} \setminus (\emptyset)_{i \in I},\,
 b \in \CB_{e}^{-}(\lambda) \cap \CB_{0}(\lambda) \bigr\} \subset \CB(\lambda),
\end{equation}
then we have
%
%
\begin{equation} \label{eq:GX2}
X_{e}^{-}(\lambda) =
\bigoplus_{b \in \CB} \BQ(q_{s}) G(b).
\end{equation}
Because $S_{\bc}^{-}v_{\lambda}=z_{\bc}v_{\lambda}$ 
for every $\bc \in \ol{\Par(\lambda)}$, 
we have 
\begin{align}
X_{e}^{-}(\lambda) 
 & = 
\sum_{
\begin{subarray}{c}
\bc \in \ol{\Par(\lambda)} \\[1.5mm]
\bc \ne (\emptyset)_{i \in I_0}
\end{subarray}
} U_{q}^{-} S_{\bc}^{-}v_{\lambda}
=
\sum_{
\begin{subarray}{c}
\bc \in \ol{\Par(\lambda)} \\[1.5mm]
\bc \ne (\emptyset)_{i \in I_0}
\end{subarray}
} U_{q}^{-} z_{\bc}v_{\lambda} \nonumber \\[3mm]
& =
\sum_{
\begin{subarray}{c}
\bc \in \ol{\Par(\lambda)} \\[1.5mm]
\bc \ne (\emptyset)_{i \in I_0}
\end{subarray}
} z_{\bc} (U_{q}^{-}v_{\lambda})
= 
\sum_{
\begin{subarray}{c}
\bc \in \ol{\Par(\lambda)} \\[1.5mm]
\bc \ne (\emptyset)_{i \in I_0}
\end{subarray}
} z_{\bc}(V_{e}^{-}(\lambda)). \label{eq:X1}
\end{align}
Let $\bc \in \ol{\Par(\lambda)} \setminus (\emptyset)_{i \in I}$, and 
$b \in \CB_{e}^{-}(\lambda) \cap \CB_{0}(\lambda)$. 
Then we deduce that $G(z_{\bc}b) = z_{\bc} G(b)$; 
indeed, since $b \in \CB_{0}(\lambda)$, 
we see by \cite[Theorem~4.16\,(ii)]{BN} that
$z_{\bc} G(b) = G(b')$ for some $b' \in \CB(\lambda)$. 
Here, 
\begin{equation*}
b' =G(b') + q_{s}\CL(\lambda)
 = z_{\bc} G(b) + q_{s}\CL(\lambda) 
 = z_{\bc}(G(b) + q_{s}\CL(\lambda)) 
 = z_{\bc}b, 
\end{equation*}
from which we get $G(z_{\bc}b) = z_{\bc} G(b)$, as desired. 
Since $G(b) \in V_{e}^{-}(\lambda)$, 
it follows from \eqref{eq:X1} that 
$G(z_{\bc}b) = z_{\bc} G(b) \in X_{e}^{-}(\lambda)$, 
and hence 
$X_{e}^{-}(\lambda) \supset 
\bigoplus_{b \in \CB} \BQ(q_{s}) G(b)$.
Now we show the opposite inclusion $\subset$ in \eqref{eq:GX2}. 
Since $\bigl\{G(b) \mid b \in \CB_{e}^{-}(\lambda)\bigr\}$
is a $\BQ(q_{s})$-basis of $V_{e}^{-}(\lambda)$, 
we see from \eqref{eq:X1} that
%
%
\begin{equation} \label{eq:Xspan}
X_{e}^{-}(\lambda) = \Span_{\BQ(q_{s})}
\bigl\{
 z_{\bc}G(b) \mid 
 \bc \in \ol{\Par(\lambda)} \setminus (\emptyset)_{i \in I},\,
 b \in \CB_{e}^{-}(\lambda)
\bigr\}.
\end{equation}
Let $\bc \in \ol{\Par(\lambda)} \setminus (\emptyset)_{i \in I}$, 
and $b \in \CB_{e}^{-}(\lambda)$. By Lemma~\ref{lem:Be-}, 
we can write the $b$ as $b = z_{\bc'}b'$ for some 
$\bc' \in \Par(\lambda)$ and $b' \in \CB_{e}^{-}(\lambda) \cap 
\CB_{0}(\lambda)$. Then we have $z_{\bc}b = z_{\bc} z_{\bc'}b'$. 
Because $z_{\bc}$ and $z_{\bc'}$ are defined by using 
Schur polynomials (see \eqref{eq:Schur}), 
their product $z_{\bc} z_{\bc'}$ can be expressed as:
\begin{equation*}
z_{\bc} z_{\bc'} = 
 \sum_{
   \begin{subarray}{c}
     \bc'' \in \ol{\Par(\lambda)} \\[1.5mm]
     |\bc''|=|\bc|+|\bc'|
   \end{subarray}
 } n_{\bc''} z_{\bc''}, \qquad 
n_{\bc''} \in \BZ_{\ge 0}; 
\end{equation*}
here we remark that $|\bc|+|\bc'| \ge 1$ since 
$\bc \ne (\emptyset)_{i \in I}$. Therefore, we deduce that
\begin{equation*}
z_{\bc}G(b) = z_{\bc}G(z_{\bc'}b') = 
 \sum_{
   \begin{subarray}{c}
     \bc'' \in \ol{\Par(\lambda)} \\[1.5mm]
     |\bc''|=|\bc|+|\bc'|
   \end{subarray}
 } n_{\bc''} G(z_{\bc''}b') \in 
\bigoplus_{b \in \CB} \BQ(q_{s}) G(b).
\end{equation*}
From this, together with \eqref{eq:Xspan}, we obtain
$X_{e}^{-}(\lambda) \subset 
\bigoplus_{b \in \CB} \BQ(q_{s}) G(b)$.
Combining these, we obtain \eqref{eq:GX2}, 
as desired; we write $\CB(X_{e}^{-}(\lambda))$ for the set $\CB$. 

Next, we prove that 
\begin{equation*}
\Psi_{\lambda}\bigl(\CB(X_{e}^{-}(\lambda))\bigr) = 
\BB^{\si}_{\sige e}(\lambda) \setminus 
\bigl\{\eta_{\psi} \mid \psi \in \QLS(\lambda)\bigr\}. 
\end{equation*}
For this purpose, 
it suffices to show that for each $\psi \in \QLS(\lambda)$, 
%
%
\begin{equation} \label{eq:X3}
\cl^{-1}(\psi) \cap 
\Psi_{\lambda}\bigl(\CB(X_{e}^{-}(\lambda))\bigr) = 
\Bigl(\cl^{-1}(\psi) \cap \BB^{\si}_{\sige e}(\lambda)\Bigr) 
 \setminus \bigl\{\eta_{\psi}\bigr\}.
\end{equation}
Let $\psi \in \QLS(\lambda)$, and 
write the $\eta_{\psi} \in \BB^{\si}_{0}(\lambda)$ as 
$\eta_{\psi}=X \eta_{e}$ for some monomial $X$ 
in root operators; recall from \eqref{eq:grch1-1} that
\begin{equation*}
\cl^{-1}(\psi) \cap \BB^{\si}_{\sige e}(\lambda) = 
\bigl\{X S_{t_{\zeta}}\eta^{C} \mid 
 C \in \Conn(\BB^{\si}(\lambda)),\,
 \zeta \in Q^{\vee+}_{I \setminus J} \bigr\}.
\end{equation*}
Let us show the inclusion $\supset$ in \eqref{eq:X3}. 
Let $\eta$ be an element in the set on the right-hand side of \eqref{eq:X3}, 
and write it as: $\eta = X S_{t_{\zeta}}\eta^{C}$, 
with $C \in \Conn(\BB^{\si}(\lambda))$ and $\zeta \in Q^{\vee+}_{I \setminus J}$. 
We write $\zeta$ as 
$\zeta=\sum_{i \in I} c_{i}\alpha_{i}^{\vee}$, 
$c_{i} \in \BZ_{\ge 0}$, $i \in I$. 
Then we define 
$\bc:=(c_{i})_{i \in I} + \Theta(C) \in \ol{\Par(\lambda)}$ 
as in \eqref{eq:par+}.
Here we claim that $\bc \ne (\emptyset)_{i \in I}$. 
Indeed, by the computation in \eqref{eq:wtXS}, 
we have 
\begin{equation*}
\wt(\eta) = 
 \wt (X S_{t_{\zeta}}\eta^{C}) = 
 \wt (\psi) + \bigl(\Degt(\psi) - |\bc|\bigr) \delta.
\end{equation*}
Since $\eta \ne \eta_{\psi}$ by our assumption, 
it follows from Remark~\ref{rem:wtXS} that 
$\wt ( \eta ) \ne \wt (\eta_{\psi}) = \wt(\psi)+\Degt(\psi)\delta$. 
Therefore, we deduce that $|\bc| \ne 0$, which implies that 
$\bc \ne (\emptyset)_{i \in I}$. 
Now, we set $b:=\Psi_{\lambda}^{-1}(\eta_{\psi}) \in 
\CB_{e}^{-}(\lambda) \cap \CB_{0}(\lambda)$; note that $b=Xu_{\lambda}$. 
Then we see by \eqref{eq:GX1} that $z_{\bc}b \in \CB(X_{e}^{-}(\lambda))$. 
Also, since 
$z_{\bc}b = z_{\bc} X u_{\lambda} = X z_{\bc}u_{\lambda} = 
X S_{t_{\zeta}}u^{\Theta(C)}$
by Remark~\ref{rem:par}, we have 
$\Psi_{\lambda}(z_{\bc}b) = X S_{t_{\zeta}}\eta^{C} = \eta$. 
Hence we conclude that $\eta \in \Psi_{\lambda}(\CB(X_{e}^{-}(\lambda))$. 
Thus we have shown the inclusion $\supset$. 

Let us show the opposite inclusion $\subset$. 
Since $\CB(X_{e}^{-}(\lambda)) \subset \CB_{e}^{-}(\lambda)$, 
it follows immediately from Theorem~\ref{thm:main} that 
\begin{equation*}
\cl^{-1}(\psi) \cap 
 \Psi_{\lambda}\bigl(\CB(X_{e}^{-}(\lambda))\bigr) \subset
\cl^{-1}(\psi) \cap \BB^{\si}_{\sige e}(\lambda). 
\end{equation*}
Therefore, it suffices to show that 
$\eta_{\psi} \not\in \Psi_{\lambda}\bigl(\CB(X_{e}^{-}(\lambda))\bigr)$. 
Suppose, for a contradiction, that 
there exists $b' \in \CB(X_{e}^{-}(\lambda))$ 
such that $\Psi_{\lambda}(b') = \eta_{\psi}$. 
By \eqref{eq:GX1}, we can write it as: $b'= z_{\bc}b$ 
for some $\bc \in \ol{\Par(\lambda)} \setminus (\emptyset)_{i \in I}$ and 
$b \in \CB_{e}^{-}(\lambda) \cap \CB_{0}(\lambda)$. 
We show that $\eta:=\Psi_{\lambda}(b) \in \cl^{-1}(\psi)$. 
Let us write $b$ as $b=Yu_{\lambda}$ 
for some monomial $Y$ in Kashiwara operators (note that $\eta=Y\eta_{e}$), 
and define $\zeta = \sum_{i \in I} c_{i}\alpha_{i}^{\vee} \in Q^{\vee+}$ and 
$\bc'=(\varrho^{(i)})_{i \in I} \in \Par(\lambda)$ in such a way that 
$\bc = (c_{i})_{i \in I} + \bc'$ (see Remark~\ref{rem:par}). 
Then, by \eqref{eq:zcu}, we have 
\begin{equation*}
b' = z_{\bc}b = z_{\bc}Yu_{\lambda} 
   = Yz_{\bc}u_{\lambda} = YS_{t_{\zeta}}u^{\bc'}.
\end{equation*}
Also, we see that 
\begin{equation} \label{eq:cep}
\eta_{\psi} = \Psi_{\lambda}(b') = 
  \Psi_{\lambda}(YS_{t_{\zeta}}u^{\bc'}) = Y S_{t_{\zeta}} \eta^{C}, 
\quad \text{with $C:=\Theta^{-1}(\bc') \in \Conn(\BB^{\si}(\lambda))$}, 
\end{equation}
and hence that
\begin{align*}
\psi 
 & = \cl(\eta_{\psi}) = \cl(Y S_{t_{\zeta}} \eta^{C}) 
   = Y \cl(S_{t_{\zeta}} \eta^{C}) = Y \cl(\eta_{e}) 
   \quad \text{(see Remark~\ref{rem:extp})} \\
 & = \cl (Y \eta_{e}) = \cl(\Psi_{\lambda}(Yu_{\lambda})) = \cl(\Psi_{\lambda}(b)).
\end{align*}
Thus, we obtain $\eta=\Psi_{\lambda}(b) \in \cl^{-1}(\psi)$, as desired.
Since $b \in \CB_{e}^{-}(\lambda)$ by our assumption, we have 
$\eta=\Psi_{\lambda}(b) \in \BB^{\si}_{\sige e}(\lambda)$.
Hence it follows from Remark~\ref{rem:wtXS} that 
$\wt (\eta) - \wt (\eta_{\psi}) \in \BZ_{\le 0} \delta$. 
On the other hand, by \eqref{eq:cep},we have 
\begin{equation*}
\wt (\eta_{\psi}) = 
\wt (Y S_{t_{\zeta}} \eta^{C}) = 
\wt (Y\eta_{e}) - |\bc|\delta = 
\wt (\eta) - |\bc|\delta,
\end{equation*}
and hence 
$\wt (\eta) - \wt (\eta_{\psi}) = |\bc|\delta \in \BZ_{\ge 0} \delta$. 
Combining these, we deduce that $|\bc|=0$, 
which implies that $\bc=(\emptyset)_{i \in I}$. 
However, this contradicts our assumption that 
$\bc \in \ol{\Par(\lambda)} \setminus (\emptyset)_{i \in I}$. 
Thus we have shown the inclusion $\subset$. 
This completes the proof of part (1) of Theorem~\ref{thm:quotient}. 

Finally, we prove part (2) of Theorem~\ref{thm:quotient}. 
Let $w \in W^{J}$. Because
\begin{equation*}
U_{w}^{-}(\lambda) \cong 
V_{w}^{-}(\lambda)/\bigl(V_{w}^{-}(\lambda) \cap X_{e}^{-}(\lambda)\bigr), 
\end{equation*}
we deduce that $V_{w}^{-}(\lambda) \cap X_{e}^{-}(\lambda)$ 
has a $\BC(q)$-basis $\bigl\{G(b) \mid 
b \in \CB_{w}^{-}(\lambda) \cap \CB(X_{e}^{-}(\lambda)) \bigr\}$. 
It follows immediately from part (1) that 
\begin{align*}
\Psi_{\lambda}(\CB_{w}^{-}(\lambda) \cap \CB(X_{e}^{-}(\lambda))
&= \BB_{\sige w}^{\si}(\lambda) \cap 
   \Bigl(\BB_{\sige e}^{\si}(\lambda) \setminus 
   \bigl\{\eta_{\psi} \mid \psi \in \QLS(\lambda)\bigr\}\Bigr) \\
& = \BB_{\sige w}^{\si}(\lambda) \setminus 
   \Bigl(\BB_{\sige w}^{\si}(\lambda) \cap 
   \bigl\{\eta_{\psi} \mid \psi \in \QLS(\lambda)\bigr\}\Bigr);
\end{align*}
note that $\CB_{w}^{-}(\lambda) \subset \CB_{e}^{-}(\lambda)$, 
since $w \in W^{J}$ and hence $w \sige e$ (see Lemma~\ref{lem:SiB2}).
Therefore, if we define $\CB(U_{w}^{-}(\lambda))$ to be 
the inverse image of the set 
\begin{equation*}
\BB_{\sige w}^{\si}(\lambda) \cap 
   \bigl\{\eta_{\psi} \mid \psi \in \QLS(\lambda)\bigr\}
\end{equation*}
under the isomorphism $\Psi_{\lambda}$, then 
the set $\bigl\{G(b) \mid b \in \CB(U_{w}^{-}(\lambda))\bigr\}$ 
is a $\BC(q)$-basis of $U_{w}^{-}(\lambda)$. 
Here, observe that 
\begin{equation*}
\BB_{\sige w}^{\si}(\lambda) \cap 
   \bigl\{\eta_{\psi} \mid \psi \in \QLS(\lambda)\bigr\} = 
   \bigl\{\eta_{\psi} \mid 
   \psi \in \QLS(\lambda) \text{ such that } 
   \kappa(\eta_{\psi}) \sige w \bigr\}.
\end{equation*}
Since $\kappa(\eta_{\psi})$ and $w$ are both contained in $W^{J}$, 
it follows from Lemma~\ref{lem:SiB2} that 
$\kappa(\eta_{\psi}) \sige w$ if and only if 
$\kappa(\eta_{\psi}) \ge w$ in the (ordinary) Bruhat order on $W^{J}$. 
Thus we have proved part (2). 
This completes the proof of Theorem~\ref{thm:quotient}. 
\end{proof}
%
%
{\small
 \setlength{\baselineskip}{15pt}

}

\end{document}